\documentclass[12pt, pdflatex, a4paper]{article}
\usepackage[utf8]{inputenc}
\usepackage{amssymb, graphicx, amsmath, amsthm, mathtools, tikz-cd, enumitem, nccmath, sectsty, mathdots}
\usepackage[mathscr]{eucal} 
\usepackage{xcolor} 
\usepackage{comment} 
\usepackage[affil-it]{authblk} 
\usepackage{mathrsfs}
\usepackage{makeidx}
\usepackage{setspace}
\usepackage{titlesec}

\usepackage{titletoc}
\usepackage{lipsum} 

\let\OLDthebibliography\thebibliography
\renewcommand\thebibliography[1]{
  \OLDthebibliography{#1}
  \setlength{\parskip}{0pt}
  \setlength{\itemsep}{1pt plus 0.3ex}
}

\usepackage{geometry}
\usepackage{hyperref}
\usepackage{wrapfig}

\usepackage[font=footnotesize]{caption} 

\usepackage{tocloft}
\usepackage{fancyhdr}

\theoremstyle{plain}
	\newtheorem{thm}{Theorem}[section]
	\newtheorem{prop}[thm]{Proposition}
	
	\newtheorem{lemma}[thm]{Lemma}
	\newtheorem{prob}[thm]{Problem}
	\newtheorem{probb}{Problem}

\theoremstyle{definition}
	\newtheorem{defn}[thm]{Definition}
	\newtheorem{rem}[thm]{Remark}

\providecommand{\keywords}[1]
{
  \small	
  \textbf{\textit{Keywords---}} #1
}

\counterwithout{figure}{section}

\usepackage{chngcntr}

\counterwithin*{equation}{section}

\sectionfont{\fontsize{14}{15}\selectfont}
\subsectionfont{\fontsize{13}{15}\selectfont}

\def\a{\mathfrak{a}}
\def\Ad{\text{\normalfont{Ad}}}

\def\ad{\text{\normalfont{ad}}}

\def\C{{\mathbb C}}

\def\Geg{\widetilde{C}}

\def\Diff{\text{\normalfont{Diff}}}
\def\spanned{\text{\normalfont{span}}}
\def\Weyl{\mathcal{D}}
\def\D{\mathbb{D}}

\def\End{\text{\normalfont{End}}}

\def\g{\mathfrak{g}}

\def\Harm{\mathcal{H}}
\def\Hom{\text{\normalfont{Hom}}}

\def\id{\text{\normalfont{id}}}
\def\Ind{\text{\normalfont{Ind}}}

\def\Im{\text{\normalfont{Im}}}

\def\l{\mathfrak{l}}
\def\L{\mathcal{L}}

\def\m{\mathfrak{m}}

\def\N{{\mathbb N}}
\def\n{\mathfrak{n}}

\def\p{\mathfrak{p}}

\def\Pol{\text{\normalfont{Pol}}}

\def\R{{\mathbb R}}
\def\Re{\text{\normalfont{Re}}}
\def\Rest{\text{\normalfont{Rest}}}

\def\so{\mathfrak{so}}
\def\Sol{\text{\normalfont{Sol}}}
\def\stab{\text{\normalfont{stab}}}
\def\supp{\text{\normalfont{supp }}}
\def\Symb{\text{\normalfont{Symb}}}

\def\V{\mathcal{V}}

\def\W{\mathcal{W}}

\def\Z{{\mathbb Z}}

\newcommand\arrowsimeq{\stackrel{\mathclap{\thicksim}}{\longrightarrow}}

\newcommand\arrowiota{\stackrel{\mathclap{\iota}}{\hookrightarrow}}

\AtBeginEnvironment{pmatrix}{\setlength{\arraycolsep}{4pt}}

\setlist[enumerate]{topsep=1pt, itemsep=-2pt} 
\setlist[itemize]{topsep=1pt, itemsep=0pt}

\title{Construction and classification of differential symmetry breaking operators for principal series representations of the pair $(SO_0(4,1), SO_0(3,1))$ for special parameters}
\author{V\'{i}ctor P\'{E}REZ-VALD\'{E}S}
\date{}

\begin{document}
\maketitle
\vspace*{-1cm} 
{\centering\textit{Dedicated to Professor Toshiyuki Kobayashi for his countless and invaluable contributions to the field.}\par}
\begin{abstract}
We construct and give a complete classification of all the differential symmetry breaking operators $\mathbb{D}_{\lambda, \nu}^{N,m}: C^\infty(S^3, \mathcal{V}_\lambda^{2N+1}) \rightarrow C^\infty(S^2, \mathcal{L}_{m, \nu})$, between the spaces of smooth sections of a vector bundle of rank $2N+1$ over the $3$-sphere $\mathcal{V}_\lambda^{2N+1}  \rightarrow S^3$, and a line bundle over the $2$-sphere $\mathcal{L}_{m, \nu} \rightarrow S^2$ in the special case $|m| = N$.
\end{abstract}
\begin{flushleft}
\end{flushleft}

\renewcommand{\contentsname}{Table of contents}
\tableofcontents

\section{Introduction}
Given a representation $\Pi$ of a Lie group $G$, and a representation $\pi$ of a Lie subgroup $G^\prime \subset G$, one can think about the $G^\prime$-intertwining operators going from the restriction $\Pi\rvert_{G^\prime}$ to $\pi$. These operators are called \textbf{symmetry breaking operators}, and their study may help us understand the nature behind the branching problems that appear in representation theory. In fact, the
problem of describing the space of these operators $\Hom_{G^\prime}(\Pi\rvert_{G^\prime}, \pi)$ corresponds to Stage C of the ABC program that T. Kobayashi proposed to address branching problems for real reductive Lie groups (\cite{kob2015}):

\begin{enumerate}[itemsep=1pt]
\item[$\bullet$] \textbf{Stage A}: Abstract features of the restriction.
\item[$\bullet$] \textbf{Stage B}: Branching laws.
\item[$\bullet$] \textbf{Stage C}: Construction of symmetry breaking operators.
\end{enumerate}

In this paper, we focus on the task of understanding these symmetry breaking operators. Concretely, we focus on the construction and classification of symmetry breaking operators that can be written as \emph{differential} operators for a concrete pair of reductive Lie groups $(G, G^\prime)$. One classical example following this line is given by the Rankin--Cohen bidifferential operators (\cite{cohen}, \cite{rankin}), which are a special case of symmetry breaking operators for the tensor product of two holomorphic discrete series representations of $SL(2,\R)$. Another example is given by the conformally covariant differential symmetry breaking operators for spherical principal series representations of the Lorentz group constructed by A. Juhl in the frame of conformal geometry (\cite{juhl}). 

One reason to think about differential symmetry breaking operators is that the space of such operators is isomorphic to the space of homomorphisms between certain generalized Verma modules, thanks to the duality theorem (see \cite[Thm. 2.9]{kob-pev1}). Thus, the problem of determining all differential symmetry breaking operators can be approached from an algebraic point of view.

Another reason is that depending on the setting, the differential operators may be a quite large class inside all symmetry breaking operators, as it can occur that no other symmetry breaking operators exist. In other words, it may happen that any symmetry breaking operator is a differential operator. This type of phenomenon is called the localess theorem (cf. \cite[Thm. 5.3]{kob-pev1}, \cite[Thm. 3.6]{kob-speh2}).

In the following, we give a specific geometric setting that we consider throughout the paper. This setting allows us to consider Stage C in a significantly wide class of cases. 

Given two smooth manifolds $X$ and $Y$, two vector bundles over them $\V \rightarrow X$ and $\W \rightarrow Y$ and a smooth map $p: Y \rightarrow X$, we can talk about the notion of a \textit{differential operator} between the spaces of smooth sections $T:C^\infty(X, \V) \rightarrow C^\infty(Y, \W)$, following the line of the famous result of J. Peetre (\cite{peetre}, \cite{peetre2}) about the characterization of differential operators in terms of their support (see Definition \ref{def-diffop}). In addition, suppose that we have a pair of Lie groups $G^\prime \subset G$ acting equivariantly on $\W \rightarrow Y$ and $\V \rightarrow X$ respectively, and that $p$ is a $G^\prime$-equivariant map between $Y$ and $X$. In this setting, we can consider the following problem:

\begin{prob}\label{prob-first} Give a description of the space of $G^\prime$-intertwining differential operators (differential symmetry breaking operators)	
$$D: C^\infty(X,\V) \rightarrow C^\infty(Y, \W).$$
\end{prob}

The setting for Problem \ref{prob-first} is very general and differential symmetry breaking operators may or may not exist. For instance, if $\W$ is  isomorphic to the pullback $p^*\V$, the restriction map $f \mapsto f\lvert_Y$ is clearly a $G^\prime$-intertwining differential operator from $C^\infty(X, \V)$ to $C^\infty(Y, \W)$. 

In the general setting where there is no morphism between $p^*\V$ and $\W$, there are cases where non-zero differential symmetry breaking operators exist, but determining when this occurs is a considerably hard task, and there are many unsolved problems. As for the ones that have been solved during the past years, for example, T. Kobayashi and B. Speh constructed and classified not only the ones that can be written as differential operators, but all the symmetry breaking operators for the tuple $(X, Y, G, G^\prime) = (S^n, S^{n-1}, O(n+1,1), O(n,1))$ and for a concrete pair of vector bundles $(\V, \W)$ (\cite{kob-speh1}, \cite{kob-speh2}).

In the case where $X = G/P \supset Y = G^\prime/P^\prime$ are flag varieties and the vector bundles $\V, \W$ are those associated to representations of the parabolic subgroups $P$ and $P^\prime$ respectively, T. Kobayashi proposed a method (\textit{the F-method} \cite{kob2013}) to construct and classify all differential symmetry breaking operators
\begin{equation*}
D: C^\infty(G/P, \V) \rightarrow C^\infty(G^\prime/P^\prime, \W).
\end{equation*}

This F-method allows us to reduce the problem of constructing differential symmetry breaking operators to the problem of finding polynomial solutions of a system of PDE's by applying the \lq\lq algebraic Fourier transform\rq\rq{} to certain generalized Verma modules, and
to use invariant theory to solve the latter. It is known (\cite[Prop. 3.10]{kob-pev1}) that if the nilradical of the parabolic subgroup $P$ is abelian, then the principal term of the equations of the system of PDE's is of order two.

In this setting, T. Kobayashi and M. Pevzner used the F-method to construct differential symmetry breaking operators for some concrete pairs of hermitian symmetric spaces $X \supset Y$ (\cite{kob-pev2}). In a slightly different context, T. Kobayashi together with B. \O rsted, P. Somberg and V. Souček used the F-method to solve a similar problem in the setting of conformal geometry. In addition, in \cite{kkp} all conformal symmetry breaking operators for differential forms on spheres were constructed and classified by T. Kobayashi, T. Kubo and M. Pevzner.

In the present paper, we consider the case $(X,Y, G, G^\prime) = (S^3, S^2, SO_0(4,1), SO_0(3,1))$ and study Problem \ref{prob-first} for a concrete pair of vector bundles $(\V, \W) = (\V_\lambda^{2N+1},\L_{m, \nu})$ (see (\ref{vectorbundle-V}) and (\ref{vectorbundle-L}) for the definition). We consider these specific bundles in order to think of differential symmetry breaking operators between \emph{any} pair of principal series representations of $SO_0(4,1)$ and $SO_0(3,1)$.

In particular, we divide Problem \ref{prob-first} into the following two problems:

\begin{probb} Give necessary and sufficient conditions on the parameters $\lambda, \nu \in \C$, $N \in \N$ and $m \in \Z$,
such that the space
\begin{equation}\label{DSBO-space}
\Diff_{SO_0(3,1)}\left(C^\infty(S^3, \V_\lambda^{2N+1}), C^\infty(S^2, \L_{m,\nu})\right)
\end{equation}
of differential symmetry breaking operators
$\D_{\lambda, \nu}^{N, m}: C^\infty(S^3, \V_\lambda^{2N+1})\rightarrow C^\infty(S^2, \L_{m,\nu})$
is non-zero. In particular, determine
\begin{equation*}
\dim_\C \Diff_{SO_0(3,1)}\left(C^\infty(S^3, \V_\lambda^{2N+1}), C^\infty(S^2, \L_{m,\nu})\right).
\end{equation*}
\end{probb}

\begin{probb} Construct explicitly the generators
 \begin{equation*}
\D_{\lambda, \nu}^{N, m} \in \Diff_{SO_0(3,1)}\left(C^\infty(S^3, \V_\lambda^{2N+1}), C^\infty(S^2, \L_{m,\nu})\right).
\end{equation*}
\end{probb}

In the case $N = m = 0$, Problems A and B above are known to be completely solved (see, for instance \cite{juhl, koss}). When $N = 1$, a solution to Problems A and B is given in \cite{kkp} for $m = 0$ (see \cite[Rem 1.5]{perez-valdes}) and in \cite{perez-valdes} for $|m| \geq 1$. Therefore, the problems above have been completely solved when $N = 0, 1$ and $m \in \Z$.

In the present paper, we consider the general case $N \in \N$ and solve Problems A and B when the parameter $m$ takes the special value $|m| = N$ (see Theorems \ref{mainthm1} and \ref{mainthm2} respectively). Moreover, although for a general $(N,m) \in \N \times \Z$, Problems A and B are still unsolved, by using the F-method we prove that, when $|m|\geq N$, they are equivalent to the problem of solving a certain system of ordinary differential equations (see Theorem \ref{Thm-findingequations} and Remark \ref{remark-Step2}). Since the nature of the appearing system is quite different when $|m| = N$ and when $|m| > N$, we solve this system when $|m| = N$ in this paper (which allows us to solve Problems A, B), and consider the case $|m| > N$ in another paper we expect to publish soon. 

At the end, we show the existence of a duality between the cases $m$ and $-m$ when $|m| \geq N$ (see Proposition \ref{prop-duality_m_and_-m}), proving that the solution to Problems A and B for $m \leq -N$ is equivalent to that for $m > N$; therefore, it suffices to solve one of the cases to obtain the whole solution.

Throughout the text, we use the following notation: $\N:= \{0, 1, 2, \ldots\}$, $\N_+:= \{1, 2, 3, \ldots\}$.

\subsection{Organization of the paper}
The paper is divided in 9 sections (including one appendix), and is structured as follows. In Section \ref{section-mainthms} we state the main results (Theorems \ref{mainthm1} and \ref{mainthm2}) and do some remarks. In Section \ref{section-Fmethod} we do a quick review of the F-method in a general setting, the machinery used to prove the main theorems. In Section \ref{section-setting} we specialize the setting to our case $(G, G^\prime, X, Y) = (SO_0(4,1), SO_0(3,1), S^3, S^2)$, giving detail in the realizations and fixing notations. Sections \ref{section-step1} and \ref{section-step2} are devoted to apply the F-method in two steps, which we call Step 1 and Step 2. By using the main result obtained in Step 2 (Theorem \ref{Thm-step2}), we prove in Section \ref{section-proof-of-main-thms} the main theorems of the paper for $m = N$. After that, we focus on proving Theorem \ref{Thm-step2}, which can be divided into two results: Theorems \ref{Thm-findingequations} and \ref{Thm-solvingequations}. These are proved in Sections \ref{section-proof_of_finding_equations} and \ref{section-proof_of_solving_equations} respectively.
Section \ref{section-case_m_lessthan_-N} shows the existence of a duality between the cases $m$ and $-m$ for $|m|\geq N$, and allows us to prove the main theorems for $m = -N$ by using the obtained results for $m = N$. The appendix at the end collects some properties of Gegenbauer polynomials we use throughout the paper, specially in Section \ref{section-proof_of_solving_equations}.

\subsection{Main results}\label{section-mainthms}
In this subsection, we present our main results, offering a complete solution to Problems A and B for $|m| = N$. We begin with the solution to Problem A.

\begin{thm}\label{mainthm1} Let $\lambda, \nu \in \C$, $N \in \N$ and $m = \pm N$. Then, the following three conditions on the quadruple $(\lambda, \nu, N, m)$ are equivalent:
\begin{enumerate}[label=\normalfont{(\roman*)}]
\item $\Diff_{SO_0(3,1)}\left(C^\infty(S^3, \V_\lambda^{2N+1}), C^\infty(S^2, \L_{m,\nu})\right) \neq \{0\}$.
\item $\dim_\C \Diff_{SO_0(3,1)}\left(C^\infty(S^3, \V_\lambda^{2N+1}), C^\infty(S^2, \L_{m,\nu})\right) = 1$.
\item $\lambda-\nu \in \N$.
\end{enumerate}
\end{thm}

Theorem \ref{mainthm1} establishes that the dimension of the space of differential symmetry breaking operators (\ref{DSBO-space}) is at most one. In Theorem \ref{mainthm2} below, we provide an explicit formula of the generator $\D_{\lambda, \nu}^{N, m}$ expressed in terms of the coordinates $(x_1, x_2, x_3) \in \R^3$ via the conformal compactification $\R^3 \arrowiota S^3$:
\begin{equation*}
\begin{tikzcd}
C^\infty(S^3, \V_\lambda^{2N+1}) \arrow[d, hookrightarrow, "{\displaystyle{\iota^*}}"'] \arrow[r, dashed] & C^\infty(S^2, \L_{m, \nu}) \arrow[d, hookrightarrow, "{\displaystyle{\iota^*}}"']\\
C^\infty(\R^3, V^{2N+1}) \arrow[r, "\D_{\lambda, \nu}^{N, m}"] & C^\infty(\R^2)
\end{tikzcd}
\end{equation*}
where $\R^2 \subset \R^3$ is realized as $\R^2 = \{(x_1, x_2, 0) : x_1, x_2 \in \R\}$.

To provide an explicit formula of $\D_{\lambda, \nu}^{N, m}$, let $\{u_d : d = 0, 1, \ldots, 2N\}$ be the standard basis of $V^{2N+1} \simeq \C^{2N+1}$ (see \eqref{basis-V2N+1}), let $\{u_d^\vee : d = 0, 1, \ldots, 2N\}$ denote the dual basis of $(V^{2N+1})^\vee$. Define $z = x_1 +ix_2$ so that the Laplacian on $\R^2$, $\Delta_{\R^2} = \frac{\partial^2}{\partial x_1^2} + \frac{\partial^2}{\partial x_2^2}$, can be written as $\Delta_{\R^2} = 4\frac{\partial^2}{\partial z \partial \overline{z}}$.

Suppose $\nu-\lambda \in \N$ and let $\widetilde{\C}_{\lambda, \nu}$ denote the following scalar-valued differential operator (cf. \cite[Eq. (2.22)]{kkp})

\begin{equation}\label{def-operator-Ctilda}
\begin{aligned}
\widetilde{\C}_{\lambda, \nu} & = \Rest_{x_3 = 0} \circ \left(I_{\nu-\lambda} \widetilde{C}^{\lambda-1}_{\nu-\lambda}\right)\left(-4\frac{\partial^2}{\partial z \partial \overline{z}}, \frac{\partial}{\partial x_3}\right),
\end{aligned}
\end{equation}
where $(I_\ell\Geg_\ell^\mu)(x,y) := x^{\frac{\ell}{2}}\Geg_\ell^\mu\left(\frac{y}{\sqrt{x}}\right)$ is a polynomial in two variables associated with the renormalized Gegenbauer polynomial $\Geg_\ell^\mu(z)$ (see (\ref{Gegenbauer-polynomial(renormalized)})). 

For $k = 0, 1, \ldots, 2N$ we define the following constant:
\begin{equation}\label{const-A}
A_{k} :=
\begin{cases}
{\displaystyle{ \frac{\Gamma\left(\lambda + N -1 + \left[\frac{\nu - \lambda -k+1}{2}\right]\right)}{\Gamma\left(\lambda + N-1\right)}}}, &\text{ if } 0 \leq \nu - \lambda \leq N,\\[10pt]
{\displaystyle{\frac{\Gamma\left(\lambda + N -1 + \left[\frac{\nu - \lambda -k+1}{2}\right]\right)}{\Gamma\left(\lambda + N-1 +\left[\frac{\nu - \lambda -2N +1}{2}\right]\right)}}}, &\text{ if } \nu - \lambda > N.
\end{cases}
\end{equation}

Now, we give the complete solution to Problem B for $|m| = N$.

\begin{thm}\label{mainthm2} Let $\lambda, \nu \in \C$ with $\nu - \lambda \in \N$, and let $m = \pm N$. Then, any differential symmetry breaking operator in {\normalfont{(\ref{DSBO-space})}} is proportional to the differential operator $\D_{\lambda, \nu}^{N, m}$ of order $\nu-\lambda$ given as follows:
\begin{itemize}[leftmargin=0.5cm]
\item[\normalfont{$\bullet$}] $m = N:$
\begin{equation}\label{Operator-general+}
\D_{\lambda, \nu}^{N,N} = \sum_{k = 0}^{2N} 2^k A_k \widetilde{\C}_{\lambda + N, \nu + N-k}\frac{\partial^k}{\partial \overline{z}^k} \otimes u_{k}^\vee,
\end{equation}
\item[\normalfont{$\bullet$}] $m = -N:$
\begin{equation}\label{Operator-general-}
\D_{\lambda, \nu}^{N,-N} =\sum_{k = 0}^{2N} (-2)^k A_k \widetilde{\C}_{\lambda + N, \nu + N-k}\frac{\partial^k}{\partial \overline{z}^k} \otimes u_{2N-k}^\vee
\end{equation}
\end{itemize}
\end{thm}
Both Theorems \ref{mainthm1} and \ref{mainthm2} will be proved in Section \ref{section-proof-of-main-thms} for $m = N$ and in Section \ref{section-case_m_lessthan_-N} for $m = -N$.

\begin{rem}\label{rem-mainthms}
As we pointed out in the introduction, the case $N = |m| = 1$ is done in \cite{perez-valdes}. Concretely, the operators $\D_{\lambda, \nu}^{\pm 1} \equiv \D_{\lambda, \nu}^{1, \pm 1}$ appearing in \cite[Eqs. (1.6)--(1.9)]{perez-valdes} are as follows:
\begin{itemize}
\item[\normalfont{$\bullet$}] If $\nu - \lambda = 0:$
\begin{equation}\label{Operator-particular1}
\D_{\lambda, \lambda}^1 = \Rest_{x_3 = 0} \otimes u_0^\vee,
\end{equation}
\begin{equation}\label{Operator-particular-1}
\D_{\lambda, \lambda}^{-1} = \Rest_{x_3 = 0} \otimes u_2^\vee.
\end{equation}
\item[\normalfont{$\bullet$}] If $\nu - \lambda \geq 1:$
\begin{multline}\label{Operator-general1}
\D_{\lambda,\nu}^{1}  =  \left(\lambda + \left[\frac{\nu-\lambda-1}{2}\right]\right) \widetilde{\C}_{\lambda+1, \nu+1} \otimes u_0^\vee\\
+2\gamma(\lambda-1, \nu-\lambda)\widetilde{\C}_{\lambda+1, \nu} \frac{\partial}{\partial \overline{z}} \otimes u_1^\vee
+ 4\widetilde{\C}_{\lambda+1, \nu-1} \frac{\partial^2}{\partial \overline{z}^2} \otimes u_2^\vee.
\end{multline}
\begin{multline}\label{Operator-general-1}
\D_{\lambda,\nu}^{-1} = 4\widetilde{\C}_{\lambda+1, \nu-1} \frac{\partial^2}{\partial z^2} \otimes u_0^\vee\\ -2\gamma(\lambda-1, \nu-\lambda)\widetilde{\C}_{\lambda+1, \nu} \frac{\partial}{\partial z} \otimes u_1^\vee
+ \left(\lambda + \left[\frac{\nu-\lambda-1}{2}\right]\right) \widetilde{\C}_{\lambda+1, \nu+1} \otimes u_2^\vee.
\end{multline}
\end{itemize}
As one can check by a straightforward computation, these expressions coincide with \eqref{Operator-general+} and \eqref{Operator-general-}.
\end{rem}

\section{Review of the F-method}\label{section-Fmethod}
In this section we provide a summary of the key tool used to prove our main results: The F-method. We do a quick review in order to state the main theorem in this section (see Theorem \ref{F-method-thm}), so that we can apply it in the following sections. For more details, we would like to refer to \cite{kob2013, kob2014, kob-pev1, kob-pev2, kubo-orsted}.

Let $X$ and $Y$ be smooth manifolds, and suppose a smooth map $p: Y \rightarrow X$ is given. Let $\V$ and $\W$ be vector bundles over $X$ and $Y$, respectively. We denote by $C^\infty(X, \V)$ (resp. $C^\infty(Y, \W)$) for the space of smooth sections, equipped with the Fréchet topology defined by the uniform convergence of sections and their derivatives of finite order on compact sets. 

We recall the definition of a \textit{differential operator}, not only for sections on the same manifold, but more generally, for sections on two different manifolds.

\begin{defn}[{\cite[Def. 2.1]{kob-pev1}}]\label{def-diffop} A continuous linear map $T: C^\infty(X, \V)\rightarrow C^\infty(Y, \W)$ is said to be a \textbf{differential operator} if it satisfies
\begin{equation*}
p\left(\supp Tf \right) \subset \supp f, \quad \text{for all } f \in C^\infty(X,\V).
\end{equation*}
\end{defn}

Furthermore, suppose that $G$ is a reductive Lie group acting transitively on $X$, such that $X = G/P$ for some parabolic subgroup $P$ of $G$. Similarly, let $G^\prime \subset G$ be a reductive subgroup of $G$ acting transitively on $Y$, with $Y = G^\prime/P^\prime$ for some parabolic subgroup $P^\prime \subset G^\prime$.

Let $P = MAN_+$ be the Langlands decomposition of $P$, and denote the Lie algebras of $P, M, A, N_+$ by $\p(\R), \m(\R), \a(\R), \n_+(\R)$, respectively. Their complexified Lie algebras are denoted by $\p, \m, \a, \n_+$. The Gelfand--Naimark decomposition of $\g(\R)$ is given by $\g(\R) = \n_-(\R) + \p(\R)$. A similar notation is adopted for $P^\prime$.

Given $\lambda \in \a^* \simeq \Hom_\R(\a(\R), \C)$, we define a one-dimensional representation $\C_\lambda$ of $A$ by 
\begin{equation*}
A \rightarrow \C^\times, \enspace a\mapsto a^\lambda := e^{\langle\lambda, \log a\rangle}.
\end{equation*}
For a representation $(\sigma, V)$ of $M$, and $\lambda \in \a^*$, let $\sigma_\lambda := \sigma \boxtimes \C_\lambda$. This representation can also be viewed as a representation of $P$ by letting $N_+$ act trivially. We define $\V := G \times_P V$ as the $G$-equivariant vector bundle over $X = G/P$ associated with $\sigma_\lambda$.

Similarly, for a representation $(\tau, W)$ of $M^\prime$, and $\nu \in \a^*$, we set $\tau_\nu \equiv \tau \boxtimes \C_\nu$ and define $\W := G^\prime \times_{P^\prime} W$, the $G^\prime$-equivariant vector bundle over $Y$ associated with $\tau_\nu$.

Given $\sigma_\lambda$, we define a new representation $\mu$ of $P$ as follows:
\begin{equation*}
\mu \equiv \sigma^*_\lambda := \sigma^\vee \boxtimes \C_{2\rho-\lambda},
\end{equation*}
where $\sigma^\vee$ denotes the dual representation of $\sigma$. We then form the induced representation
\begin{equation*}
\pi_\mu \equiv \pi_{(\sigma, \lambda)^*} = \Ind_{P}^G(\sigma_\lambda^*)
\end{equation*}
of $G$ on $C^\infty(X,\V^*)$, where $\V^* := G \times_P \sigma_\lambda^*$. 
Here, $\C_{2\rho}$ is the following one-dimensional representation of $P$:
\begin{equation*}
P \ni p \mapsto |\det \left(\Ad(p): \n_+(\R) \rightarrow \n_+(\R)\right)|.
\end{equation*}

The infinitesimal action $d \pi_\mu \equiv d \pi_{(\sigma, \lambda)^*}$ composed with the algebraic Fourier transform $\widehat{\cdot}$ induces a well-defined Lie algebra homomorphism
\begin{equation*}
\widehat{d\pi_\mu}: \g\rightarrow \Weyl(\n_+) \otimes \End(V^\vee),
\end{equation*}
where $\Weyl(\n_+)$ denotes the Weyl algebra on $\n_+$  (see \cite{kob-pev1}). In fact, a closed formula for $d \pi_\mu$ was given in \cite[Eq. (3.13)]{kob-pev1}. 

Now, let
\begin{equation*}
\Hom_{L^\prime}\left(V, W \otimes \Pol(\n_+)\right) := \{\psi \in \Hom_\C\left(V, W \otimes \Pol(\n_+)\right) :\text{\normalfont{ (\ref{F-system-1}) holds}}\},
\end{equation*}
\begin{equation}\label{F-system-1}
\psi \circ \sigma_\lambda(\ell) = \tau_{\nu}(\ell) \circ \Ad_\#(\ell)\psi \quad \forall \enspace \ell \in L^\prime,
\end{equation}
where the $L^\prime$-action on $\Pol(\n_+)$ is given by
\begin{equation*}
\begin{aligned}
\Ad_\#(\ell): \Pol(\n_+) &\longrightarrow \Pol(\n_+)\\
\hspace{2.8cm} p(\cdot) &\longmapsto p\left(\Ad(\ell^{-1}) \cdot\right).
\end{aligned}
\end{equation*}
Furthermore, we set
\begin{equation*}
\Sol(\n_+; \sigma_\lambda, \tau_\nu) := \{\psi \in \Hom_{L^\prime}\left(V, W \otimes \Pol(\n_+)\right) :\text{\normalfont{(\ref{F-system-2}) holds}}\},
\end{equation*}
\begin{equation}\label{F-system-2}
\left(\widehat{d \pi_\mu}(C) \otimes \id_{W} \right)\psi = 0 \quad \forall \enspace C\in \n_+^\prime.
\end{equation}
If $\n_+$ is abelian, then we have (see \cite[Sec. 4.3]{kob-pev1}):
\begin{equation*}
\Diff_{G^\prime}\left(C^\infty(X, \V), C^\infty(Y, \W)\right) \subset \Diff^\text{const}(\n_-) \otimes \Hom_\C(V,W).
\end{equation*}
By applying the symbol map (see \cite[Sec. 4.1]{kob-pev1}), we obtain the following:
\begin{thm}[{\cite[Thm. 4.1]{kob-pev1}}]\label{F-method-thm} Suppose that the Lie algebra $\n_+$ is abelian. Then, there exists a linear isomorphism
\begin{equation*}
\begin{tikzcd}[column sep = 1cm]
\Diff_{G^\prime}\left(C^\infty(X, \V), C^\infty(Y, \W)\right) \arrow[r, "\displaystyle{_{\Symb \otimes \id}}", "\thicksim"'] & \Sol(\n_+; \sigma_\lambda, \tau_\nu).
\end{tikzcd}
\end{equation*}
\end{thm}

By Theorem \ref{F-method-thm} above, to determine $\Diff_{G^\prime}\left(C^\infty(X, \V), C^\infty(Y, \W)\right)$, it suffices to compute $\Sol(\n_+; \sigma_\lambda, \tau_\nu)$. This can be achieved by following the two steps below:
\begin{itemize}
\item \underline{Step 1}. Determine the generators of $\Hom_{L^\prime}\left(V, W \otimes \Pol(\n_+)\right)$.
\item \underline{Step 2}. Solve the equation $\left(\widehat{d \pi_\mu}(C) \otimes \id_{W} \right)\psi = 0, \quad \forall \enspace C\in \n_+^\prime$.
\end{itemize}
In Sections \ref{section-step1} and \ref{section-step2}, we will carry out Steps 1 and 2 for the pair $\left(G, G^\prime\right) = \left(SO_0(4,1), SO_0(3,1)\right)$.

\section{Principal Series Representations of $G = SO_0(4,1)$ and $G^\prime = SO_0(3,1)$}\label{section-setting}
In the previous section we reviewed the F-method in a general setting for two flag varieties $X=G/P$ and $Y=G^\prime/P^\prime$, along with vector bundles $\V \rightarrow X$ and $\W \rightarrow Y$. From this point onward, we focus on the specific case where $(G, G^\prime) = (SO_0(4,1), SO_0(3,1))$ and $(X,Y) = (S^3,S^2)$. In this section, we briefly review the geometric realization of $G$ and $G^\prime$, and define their principal series representations $\Ind_P^G(\sigma_\lambda^{2N+1})$ and $\Ind_{P^\prime}^{G^\prime}(\tau_{m, \nu})$. Throughout, we adopt the notation in \cite{perez-valdes}.

We realize the de Sitter group $O(4,1)$ as
\begin{equation*}
O(4,1) = \{g \in GL(5,\R) : Q_{4,1}(gx) = Q_{4,1}(x) \enspace \forall x \in \R^5\},
\end{equation*}
where $Q_{4,1}$ is the following bilinear form on $\R^5$ of signature $(4,1)$:
\begin{equation*}
Q_{4,1}(x) := x_0^2 + x_1^2 + x_2^2 + x_3^2 - x_4^2, \quad \text{for } x = {}^t(x_0, \dots, x_4) \in \R.
\end{equation*}

Let $\g(\R) := \so(4,1)$ denote the Lie algebra of $O(4,1)$ and define the following elements:
\begin{equation}\label{elements}
\begin{aligned}
H_0 & := E_{1,5} + E_{5,1},\\
C_j^+ & := -E_{1,j+1} + E_{j+1,1} - E_{5,j+1} - E_{j+1,5} & (1 \leq j  \leq 3),\\
C_j^- & := -E_{1,j+1} + E_{j+1,1} + E_{5,j+1} + E_{j+1,5} & (1 \leq j  \leq 3),\\
X_{p,q} & := E_{q+1, p+1} - E_{p+1, q+1} & (1 \leq p \leq q \leq 3).
\end{aligned}
\end{equation}
Here, $E_{p,q} \enspace (1 \leq p, q \leq 5)$ represent the elementary $5\times 5$ matrix with a $1$ in the $(p,q)$-entry and zeros elsewhere. One clearly has
\begin{equation}\label{brackets-Cpm}
\frac{1}{2}[C_j^+, C_k^-] = X_{j, k} -\delta_{j,k}H_0, \quad X_{p,q} = -X_{q,p}.
\end{equation}
Note that the sets $\{C_j^+: 1 \leq j  \leq 3\}$, $\{C_j^-: 1 \leq j  \leq 3\}$ and $\{X_{p,q} : 1 \leq p,q \leq 3\}\cup\{H_0\}$ determine bases of the following Lie subalgebras of $\g(\R)$:
\begin{equation*}
\begin{aligned}
\n_+(\R) & := \ker(\ad(H_0) - \id)\subset \g(\R),\\
\n_-(\R) & := \ker(\ad(H_0) + \id)\subset \g(\R),\\
\l(\R) & := \ker\ad(H_0)\subset \g(\R).
\end{aligned}
\end{equation*}

Now, let $\Xi$ be the isotropic cone defined by $Q_{4,1}$:
\begin{equation*}
\Xi := \{x \in \R^5\setminus\{0\} : Q_{4,1}(x) = 0\} \subset \R^5,
\end{equation*}
and observe that the projection $p: \Xi \rightarrow S^3, \enspace x \mapsto p(x) := \frac{1}{x_4}{}^t(x_0, x_1, x_2, x_3)$ determines a bijection $\widetilde{p}: \Xi/\R^\times \arrowsimeq S^3$. Let $G := SO_0(4,1)$ be the identity component of $O(4,1)$ and define 
\begin{equation*}
P := \stab_G(1:0:0:0:1).
\end{equation*}
Then, since the action of $G$ on $\Xi/\R^\times$ is transitive, we have bijections
\begin{equation*}
G/P \arrowsimeq \Xi/\R^\times \arrowsimeq S^3.
\end{equation*}
The group $P$ is a parabolic subgroup of $G$ with Langlands decomposition $P = LN_+ = MAN_+$, where
\begin{equation*}
\begin{gathered}
M = \left\{
\begin{pmatrix}
1 & &\\
& g &\\
& & 1
\end{pmatrix} : g \in SO(3)
\right\} \simeq SO(3),\\[6pt]
A = \exp(\R H_0) \simeq SO_0(1,1), \quad
N_+ = \exp(\n_+(\R)) \simeq \R^3.
\end{gathered}
\end{equation*}

Moreover, for $x = {}^t(x_1, x_2, x_3) \in \R^3$ let $Q_3(x) \equiv \|x\|_{\R^3}^2 := x_1^2 + x_2^2 + x_3^2$ and let $N_- := \exp(\n_-(\R))$. We define a diffeomorphism $n_-: \R^3 \rightarrow N_-$ by
\begin{equation}
\label{open-Bruhat-cell}
\begin{gathered}
n_-(x) := \exp\left(\sum_{j = 1}^3 x_jC_j^-\right) =  I_3 +
\begin{pmatrix}
- \frac{1}{2}Q_3(x) & -{}^tx & -\frac{1}{2}Q_3(x)\\
x & 0 & x\\
\frac{1}{2}Q_3(x) & {}^t x & \frac{1}{2}Q_3(x)
\end{pmatrix}.
\end{gathered}
\end{equation}
Then, this map gives the coordinates of the open Bruhat cell:
\begin{equation*}
\n_-(\R) \hookrightarrow G/P \simeq S^3, \enspace Z \mapsto \exp(Z) \cdot P.
\end{equation*}

Now, for any $\lambda \in \C$ and any $N \in \N$, let $(\sigma^{2N+1}_\lambda, V_\lambda^{2N+1})$ be the following $(2N+1)$-dimensional representation of $L \equiv MA$
\begin{equation*}
\begin{aligned}
\sigma_\lambda^{2N+1}: L = M \times A &\longrightarrow GL_\C(V^{2N+1})\\
\hspace{-0.2cm}(B, \exp(t H_0)) &\longmapsto e^{\lambda t}\sigma^{2N+1}(B),
\end{aligned}
\end{equation*}
where $\sigma^{2N+1}$ denotes the $(2N+1)$-dimensional irreducible representation of $SO(3)$ realized in $V^{2N+1}$, the space of homogeneous polynomials of degree $2N$ in two variables (see Section \ref{section-fdrepsofSO(3)}). For instance, if $N = 0$ we have the trivial representation of $SO(3)$, while if $N=1$, $\sigma^3$ is isomorphic to the natural representation on $\C^3$. We observe that $\sigma_\lambda^{2N+1}$ can be written as $\sigma^{2N+1} \boxtimes \C_\lambda$, where $\C_\lambda$ is the character of $A$ defined by exponentiation $A \rightarrow \C^\times, a=e^{tH_0}\mapsto a^\lambda = e^{\lambda t}$.

We let $N_+$ act trivially and think of $\sigma_\lambda^{2N+1}$ as a representation of $P$. We define the following homogeneous vector bundle over $G/P\simeq S^3$ associated to $(\sigma_\lambda^{2N+1}, V_\lambda^{2N+1}):$
\begin{equation}\label{vectorbundle-V}
\V_\lambda^{2N+1} := G \times_P V_\lambda^{2N+1}.
\end{equation}
Then, we form the principal series representation $\Ind_P^G(\sigma_\lambda^{2N+1})$ of $G$ on the space of smooth sections $C^\infty(G/P, \V_\lambda^{2N+1})$, or equivalently, on 
\begin{equation*}
\begin{aligned}
C^\infty(G, V_\lambda^{2N+1})^P =
\{f \in C^\infty(G, V_\lambda^{2N+1}) : f(gman) &= \sigma^{2N+1}(m)^{-1} a^{-\lambda}f(g) \\
& \text{for all } m\in M, a\in A, n\in N_+, g\in G\}.
\end{aligned}
\end{equation*}
Its $N$-picture is defined on $C^\infty(\R^3, V^{2N+1})$ via the restriction to the open Bruhat cell (\ref{open-Bruhat-cell}):
\begin{equation*}
\begin{aligned}
C^\infty(G/P, \V_\lambda^{2N+1}) \simeq C^\infty(G, V_\lambda^{2N+1})^P &\longrightarrow C^\infty(\R^3, V^{2N+1})\\
\hspace{5.2cm}f &\longmapsto F(x) := f(n_-(x)).
\end{aligned}
\end{equation*}

Now, we realize $G^\prime = SO_0(3,1)$ as $G^\prime \simeq \stab_G{}^t(0,0,0,1,0) \subset G$, and denote its Lie algebra by $\g^\prime(\R)$. The action of this subgroup leaves invariant $\Xi\cap\{x_3 = 0\}$, thus it acts transitively on the submanifold
\begin{equation*}
S^2 = \{(y_0, y_1, y_2, y_3) \in S^3 : y_3 = 0\} \simeq \left(\Xi\cap\{x_3 = 0\}\right)/\R^\times.
\end{equation*}
Then, $P^\prime := P \cap G^\prime$ is a parabolic subgroup of $G^\prime$ with Langlands decomposition $P^\prime = L^\prime N_+^\prime = M^\prime A N_+^\prime$, where
\begin{equation}\label{realization-Mprime}
\begin{gathered}
M^\prime = M\cap G^\prime =\left\{
\begin{pmatrix}
1 & & &\\
& B & &\\
& & 1 &\\
& & & 1
\end{pmatrix} : B\in SO(2) \right\} \simeq SO(2),\\[6pt]
N^\prime_+ =  N_+ \cap G^\prime = \exp(\n_+^\prime(\R)) \simeq \R^2.
\end{gathered}
\end{equation}
Here, the Lie algebras $\n_\pm^\prime(\R)$ stand for $\n_\pm^\prime(\R) := \n_\pm(\R)\cap \g^\prime$. In particular, the sets $\{C_j^+: 1 \leq j  \leq 2\}$ and $\{C_j^-: 1 \leq j  \leq 2\}$ 
determine bases of $\n_+^\prime(\R)$ and $\n_-^\prime(\R)$ respectively.
 
For any $\nu \in \C$ and any $m \in \Z$, let $(\tau_{m, \nu}, \C_{m, \nu})$ be the following one-dimensional representation of $L^\prime \equiv M^\prime A$ 

\begin{equation*}
\begin{gathered}
\tau_{m, \nu}: L^\prime = M^\prime \times A \longrightarrow \C^\times\\
\hspace{1.4cm}(
\begin{pmatrix}
\cos s & -\sin s\\
\sin s & \cos s
\end{pmatrix}
, \exp(t H_0)) \longmapsto e^{ims}e^{\nu t}.
\end{gathered}
\end{equation*}
As before, by letting $N_+^\prime$ act trivially, we treat $\tau_{m, \nu}$ as a representation of $P^\prime$ and define the following homogeneous vector bundle over $G^\prime/P^\prime\simeq S^2$ associated to $(\tau_{m,\nu}, \C_{m,\nu}):$
\begin{equation}\label{vectorbundle-L}
\L_{m,\nu} := G^\prime \times_{P^\prime} \C_{m,\nu}.
\end{equation}
We form the principal series representation $\Ind_{P^\prime}^{G^\prime}(\tau_{m,\nu})$ of $G^\prime$ on $C^\infty(G^\prime/P^\prime, \L_{m,\nu})$.

It follows from Theorem \ref{F-method-thm}, that the symbol map gives a linear isomorphism
\begin{equation}\label{isomorphism-Fmethod-concrete}
\begin{tikzcd}[column sep = 1cm]
\Diff_{G^\prime}\left(C^\infty(S^3, \V_\lambda^{2N+1}), C^\infty(S^2, \L_{m,\nu})\right) \arrow[r, "\displaystyle{_{\Symb \otimes \id}}", "\thicksim"'] & \Sol(\n_+; \sigma_\lambda^{2N+1}, \tau_{m, \nu}),
\end{tikzcd}
\end{equation}
where the right-hand side stands for
\begin{equation}\label{F-system}
\Sol(\n_+; \sigma_\lambda^{2N+1}, \tau_{m, \nu}) := \{\psi \in \Hom_\C\left(V_\lambda^{2N+1}, \C_{m,\nu} \otimes \Pol(\n_+)\right) :\text{(\ref{F-system-1-concrete}), (\ref{F-system-2-concrete}) hold}\},
\end{equation}
\begin{equation}\label{F-system-1-concrete}
\psi \circ \sigma_\lambda^{2N+1}(\ell) = \tau_{m, \nu}(\ell) \circ \Ad_\#(\ell)\psi, \quad \forall \enspace \ell \in L^\prime=M^\prime A,
\end{equation}
\begin{equation}\label{F-system-2-concrete}
\left(\widehat{d \pi_\mu}(C) \otimes \id_{\C_{m,\nu}} \right)\psi = 0, \quad \forall \enspace C\in \n_+^\prime.
\end{equation}
In the following sections, we study Steps 1 and 2 of the F-method to determine the space $\Sol(\n_+; \sigma_\lambda^{2N+1}, \tau_{m, \nu})$.

\section{Step 1: Generators of $\Hom_{L^\prime}\left(V_\lambda^{2N+1}, \C_{m,\nu} \otimes \Pol(\n_+)\right)$}\label{section-step1}
In this section, we apply finite-dimensional representation theory to identify the generators of the space
\begin{equation}\label{space-step1}
 \Hom_{L^\prime}\left(V_\lambda^{2N+1}, \C_{m,\nu} \otimes \Pol(\n_+)\right),
\end{equation}
for $|m| \geq N$ (see Proposition \ref{prop-FirststepFmethod}). This corresponds to the first step in the F-method.
The second step, which involves solving the differential equation (\ref{F-system-2-concrete}), is summarized in Section \ref{section-step2}.

In the following, we do the identification $\Pol(\n_+) \simeq \Pol[\zeta_1,\zeta_2, \zeta_3]$, where $(\zeta_1, \zeta_2, \zeta_3)$ are coordinates in $\n_+$ with respect to the basis $\{C_j^+: 1\leq j\leq 3\}$. A direct computation shows that, using the realizations given in (\ref{realization-Mprime}), the action $\Ad_\#$ of $L^\prime = SO(2) \times A$ on $\Pol(\n_+)$ corresponds to the following standard action:
\begin{equation}\label{action-SO(2)-on-Pol}
\begin{aligned}
SO(2) \times A \times \Pol[\zeta_1,\zeta_2, \zeta_3] & \longrightarrow \Pol[\zeta_1,\zeta_2, \zeta_3]\\
\enspace (g, \exp(tH_0), p(\cdot))& \longmapsto p\left(e^{-t}\begin{pmatrix}
g &\\
& 1
\end{pmatrix}^{-1}\cdot\right).
\end{aligned}
\end{equation}

\subsection{Finite-Dimensional Representation Theory of $SO(3)$}\label{section-fdrepsofSO(3)}
Now, we recall a well-known result regarding the classification of finite-dimensional irreducible representations of $SO(3)$ (see, for instance \cite{hall}, \cite{knapp}).

It is a well-established fact that any finite-dimensional irreducible representation of $SU(2)$ is parametrized by a non-negative integer $k \in \N$ and can be realized in the vector space of homogeneous polynomials of degree $k$ in two variables as follows:
\begin{equation*}
\begin{gathered}
\theta^k: SU(2) \longrightarrow GL_\C\left(V^{k+1}\right), \quad
g \longmapsto \left(p(\cdot) \mapsto p\left(g^{-1} \cdot\right)\right),
\end{gathered}
\end{equation*}
where $V^{k+1} := \{p \in \C[\xi_1, \xi_2]: \deg(p) = k\}$.

Let $\varpi: SU(2) \twoheadrightarrow SO(3)$ be the $2:1$ covering map of $SO(3)$, which has kernel $\ker \varpi = \{\pm I_2\}$. The map can be written explicitly as follows:
\begin{equation}\label{expression-covering}
\begin{gathered}
\varpi(U(x,y)) = \begin{pmatrix}
\Re(x^2-y^2) & \Im(x^2+y^2) & -2\Re(xy)\\[3pt]
-\Im(x^2-y^2) & \Re(x^2+y^2) & 2\Im(xy)\\[3pt]
2\Re(x\overline{y}) & 2\Im(x\overline{y}) & |x|^2-|y|^2
\end{pmatrix},
\end{gathered}
\end{equation}
for $U(x,y) := \begin{pmatrix}
x & y\\
-\overline{y} &\overline{x}
\end{pmatrix} \in SU(2)$.

It can be shown that the representation $\theta^k$ factors through $\varpi$ as a representation of $SO(3)$ if and only if $k = 2N$ for some $N \in \N$. Thus, any finite-dimensional representation of $SO(3)$ is of the form
\begin{equation*}
\begin{gathered}
\sigma^{2N+1}: SO(3) \longrightarrow GL_\C\left(V^{2N+1}\right), \quad
B \longmapsto \left(p(\cdot) \mapsto p\left(\left(\widetilde{\varpi}^{-1}(B)\right)^{-1} \cdot\right)\right),
\end{gathered}
\end{equation*}
where $\widetilde{\varpi}$ denotes the associated isomorphism $\widetilde{\varpi}: SU(2)/\{\pm I_2\} \arrowsimeq SO(3)$.

We consider the following basis of $V^{2N+1}$
\begin{equation}\label{basis-V2N+1}
B(V^{2N+1}) = \{u_j : j = 0,1,2, \ldots, 2N\} = \{\xi_1^{2N-j}\xi_2^{j} : j=0, 1, \ldots, 2N\}.
\end{equation}
Note that the index notation is slightly different from that of \cite{perez-valdes}, where the first element of the basis $B(V^{2N+1})$ was denoted by $u_1$ while here it is $u_0$.  

\subsection{Characterization of $\Hom_{L^\prime}\left(V_\lambda^{2N+1}, \C_{m,\nu} \otimes \Pol(\n_+)\right)$} 
In this subsection we fully characterize the space (\ref{space-step1}) for $|m| \geq N$ (see Proposition \ref{prop-FirststepFmethod}). Following (\ref{realization-Mprime}), we embed $SO(2)$ into $SO(3)$ in the natural way as
\begin{equation}\label{realization-SO(2)inSO(3)}
SO(2) \simeq \left\{\begin{pmatrix}
g &\\
&1
\end{pmatrix}: g\in SO(2)\right\} \subset SO(3).
\end{equation}
We begin with a straightforward but valuable result about harmonic polynomials.

For $k, d \in \N$ let
\begin{equation*}
\Harm^k(\C^d) := \{h\in \Pol^k[\zeta_1, \dots, \zeta_d] : \Delta h = 0\}
\end{equation*}
denote the space of harmonic polynomials, homogeneous of degree $k$ on $\C^d$. When $d = 2$, this space has dimension at most $2$ and can be explicitly described as follows (see, for instance \cite[Thm. 3.1]{helgason}):
\begin{equation}\label{decomposition-Harm-2}
\Harm^k(\C^2) = 
\begin{cases}
\C, & \text{if } k = 0,\\
\C(\zeta_1 - i\zeta_2)^k + \C(\zeta_1 + i\zeta_2)^k, &\text{if } k \geq 0.
\end{cases}
\end{equation}

\begin{lemma}\label{lemma-Harm} Let $N, k \in \N$ and $m \in \Z$, with $|m| \geq N$. Then, the following three conditions on the triple $(N, m, k)$ are equivalent.
\begin{enumerate}[label=\normalfont{(\roman*)}, topsep=2pt]
\item $\Hom_{SO(2)}\left(V^{2N+1}, \C_m\otimes \Harm^k(\C^2)\right) \neq \{0\}$.
\item $\dim_\C \Hom_{SO(2)}\left(V^{2N+1}, \C_m\otimes \Harm^k(\C^2)\right) = 1$.
\item $k \in K_{N,m} := \{|m-\ell|, |m+\ell|: \ell = 0, 1, \ldots, N\}$.
\end{enumerate}
Moreover, if one (therefore all) of the above conditions is satisfied, the generator $h_k^\pm$ of the space $\Hom_{SO(2)}\left(V^{2N+1}, \C_m\otimes \Harm^k(\C^2)\right)$ is given in \eqref{generators-Hom(V2N+1,CmHk)} below.
\begin{equation}\label{generators-Hom(V2N+1,CmHk)}
\begin{array}{l l}
h_{m+\ell}^+: V^{2N+1} \longrightarrow \C_m \otimes \Harm^{m+\ell}(\C^2), &\enspace
u_{N+\ell}=\xi_1^{N-\ell}\xi_2^{N+\ell} \longmapsto 1 \otimes(\zeta_1 + i\zeta_2)^{m+\ell},\\[7pt]
h_{m-\ell}^+: V^{2N+1} \longrightarrow \C_m \otimes \Harm^{m-\ell}(\C^2), &\enspace
u_{N-\ell} = \xi_1^{N+\ell}\xi_2^{N-\ell} \longmapsto 1 \otimes (\zeta_1 + i\zeta_2)^{m-\ell},\\[7pt]
h_{-m-\ell}^-: V^{2N+1} \longrightarrow \C_m \otimes \Harm^{-m-\ell}(\C^2), &\enspace
u_{N+\ell} = \xi_1^{N-\ell}\xi_2^{N+\ell} \longmapsto 1 \otimes (\zeta_1 - i\zeta_2)^{-m-\ell},\\[7pt]
h_{-m+\ell}^-: V^{2N+1} \longrightarrow \C_m \otimes \Harm^{-m+\ell}(\C^2), &\enspace
u_{N-\ell}= \xi_1^{N+\ell}\xi_2^{N-\ell} \longmapsto 1 \otimes(\zeta_1 - i\zeta_2)^{-m+\ell}.
\end{array}
\end{equation}
\end{lemma}
\begin{proof}
Set $g := \begin{pmatrix}
\cos t & -\sin t & 0\\
\sin t & \cos t & 0\\
0 & 0 & 1 
\end{pmatrix} \in SO(2)\hookrightarrow SO(3)$. It follows from (\ref{expression-covering}) that $\varpi(U(e^{-it/2},0)) = g$. Thus, for the basic element $u_j = p(\xi_1, \xi_2) = \xi_1^{2N-j}\xi_2^j$ $(j = 0,1,2, \ldots, 2N)$ (see \eqref{basis-V2N+1}), we have 
\begin{equation*}
\sigma^{2N+1}(g)p(\xi_1, \xi_2) = e^{(N-j)it}p(\xi_1, \xi_2).
\end{equation*}
Therefore,
\begin{equation}\label{decomp-V2N+1}
\begin{aligned}
V^{2N+1} = \bigoplus_{j=0}^{2N}\C u_j &= \bigoplus_{j=0}^{N-1}\C u_j \oplus \C u_N \oplus \bigoplus_{N+1}^{2N} \C u_j\\
 & \simeq \bigoplus_{\ell=1}^{N} \C_{\ell} \oplus \C_0 \oplus \bigoplus_{\ell = 1}^{N} \C_{-\ell} \quad \text{ as } SO(2)\text{-modules.}
\end{aligned}
\end{equation}

On the other hand, the decomposition (\ref{decomposition-Harm-2}) is nothing but the decomposition of $\Harm^k(\C^2)$ as an $SO(2)$-module:
\begin{equation*}
\C(\zeta_1 - i\zeta_2)^k + \C(\zeta_1 + i\zeta_2)^k \simeq \C_{k} + \C_{-k}.
\end{equation*}
Hence
\begin{equation}\label{decomp-CmHk}
\C_m \otimes \Harm^k(\C^2) \simeq \begin{cases}
\C_m, &\text{if } k = 0,\\
\C_{m+k} \oplus \C_{m-k},&\text{if } k > 0.
\end{cases}
\end{equation}
Combining \eqref{decomp-V2N+1} and \eqref{decomp-CmHk} gives us
\begin{align}
\label{decomp-V2N+1CmH0}
(V^{2N+1})^\vee \otimes \C_m \otimes \Harm^0(\C^2) & \simeq 
\bigoplus_{\ell = 1}^N \C_{m-\ell} \oplus \C_m \oplus \bigoplus_{\ell = 1}^N \C_{m+\ell}\enspace (k = 0),\\
\begin{split}\label{decomp-V2N+1CmHk}
(V^{2N+1})^\vee \otimes \C_m \otimes \Harm^k(\C^2) & \simeq 
\bigoplus_{\ell = 1}^N (\C_{m+k-\ell}\oplus\C_{m-k-\ell}) \oplus \C_{m+k} \oplus \C_{m-k} \\
& \quad \oplus \bigoplus_{\ell = 1}
^N (\C_{m+k+\ell}\oplus\C_{m-k+\ell}) \enspace (k > 0).
\end{split}
\end{align}
The equivalence (i) $\Leftrightarrow$ (iii) follows from taking $SO(2)$-invariants in the decompositions above. In order to find the dimension of $\Hom_{SO(2)}\left(V^{2N+1}, \C_m\otimes \Harm^k(\C^2)\right)$ we make use of the condition $|m| \geq N$. If $k = 0 \in K_{N,m}$, then necessarily $m = \pm N$. In this case there is only one term that does not vanish when taking $SO(2)$-invariants in \eqref{decomp-V2N+1CmH0}, which is $\C_{m-N}$ if $m = -N$ and $\C_{m+N}$ if $m = N$.
Similarly, if $k > 0$ we note that the number of terms in \eqref{decomp-V2N+1CmHk} that are subject to remain when taking $SO(2)$-invariants is $2N+1$, one for each value of $k \in K_{N,m}$. But since $|m| \geq N$, we have $\#K_{N,m} = 2N+1$, so only one term remains for each $k \in K_{N,m}$. Therefore $\dim_\C \Hom_{SO(2)}\left(V^{2N+1}, \C_m\otimes \Harm^k(\C^2)\right) = 1$.

The generators $h_k^\pm$ can be obtained by a straightforward computation by taking a look at which terms do not vanish in the decompositions \eqref{decomp-V2N+1CmH0} and \eqref{decomp-V2N+1CmHk}.
\end{proof}

The sign index $\pm$ in $h_k^\pm$ is chosen to coincide with that of $m\pm\ell$, being $+$ if $m \geq N$ and $-$ if $m\leq -N$. This choice also matches with the sign of $(\zeta_1 \pm i\zeta_2)$.

Now, given $b\in \Z$ and $g \in \C[t]$, we define a meromorphic function $(T_bg)$ on the variables $\zeta = (\zeta_1, \zeta_2,\zeta_3)$ as follows:
\begin{equation}\label{Def-T_b}
(T_bg)(\zeta) := (\zeta_1^2+\zeta_2^2)^\frac{b}{2} g\left(\frac{\zeta_3}{\sqrt{\zeta_1^2+\zeta_2^2}}\right).
\end{equation}

Note that, if $b \in \N$ and $g \in \Pol_b[t]_\text{{\normalfont{even}}}$, then $T_bg$ is a homogeneous polynomial on three variables of degree $b$, where
\begin{equation}\label{def-Pol_even}
\Pol_b[t]_\text{{\normalfont{even}}}:= \spanned_\C\left\{t^{b-2j}: j = 0, \dots, \left[\frac{b}{2}\right]\right\}.
\end{equation}
Moreover, we obtain:
\begin{equation}\label{bijection-T_b}
T_b: \Pol_b[t]_\text{{\normalfont{even}}} \arrowsimeq \bigoplus_{2b_1+b_2 = b} \Pol^{b_1}[\zeta_1^2 + \zeta_2^2]\otimes\Pol^{b_2}[\zeta_3].
\end{equation}

\begin{lemma}[{\cite[Lem. 4.2]{kkp}}]\label{Harm-Pol-isomorphism} For any $a \in \N$, we have the following bijection.
\begin{equation*}
\begin{aligned}
\bigoplus_{k = 0}^a\Pol_{a-k}[t]_\text{{\normalfont{even}}} \otimes \Hom_{SO(2)}\left(V^{2N+1}, \C_m\otimes \Harm^k(\C^2)\right) &\arrowsimeq \Hom_{SO(2)}\left(V^{2N+1}, \C_m \otimes \Pol^a[\zeta_1, \zeta_2, \zeta_3]\right),\\
\sum_{k=0}^a g_k \otimes H_k &\longmapsto \sum_{k = 0}^a \left(T_{a-k}g_k\right)H_k.
\end{aligned}
\end{equation*}
\end{lemma}
This result comes from combining (\ref{bijection-T_b}) with the natural decomposition
\begin{equation*}
\begin{gathered}
\Pol[\zeta_1^2 + \cdots + \zeta_d^2] \otimes \Harm(\C^d) \arrowsimeq \Pol[\zeta_1, \ldots, \zeta_d],
\end{gathered}
\end{equation*}
where $\Harm(\C^d) := \bigoplus_{k=0}^\infty \Harm^k(\C^d)$.
\begin{prop}\label{prop-FirststepFmethod} Let $\lambda, \nu \in \C$, $N \in \N$, $m \in \Z$ with $|m| \geq N$ and $a \in \N$. Then, the following two conditions on the tuple $(\lambda, \nu, N, m, a)$ are equivalent.
\begin{enumerate}[label=\normalfont{(\roman*)}, topsep=3pt]
\item $\Hom_{L^\prime}\left(V^{2N+1}_\lambda, \C_{m, \nu} \otimes \Pol^a(\n_+)\right) \neq \{0\}$.
\item $a = \nu - \lambda$ and $a \geq \min K_{N,m}$.
\end{enumerate}
Moreover, if one (therefore all) of the conditions above is satisfied, we have
\begin{equation*}
\begin{gathered}
\Hom_{L^\prime}\left(V^{2N+1}_\lambda, \C_{m, \nu} \otimes \Pol^a(\n_+)\right)= 
\spanned_\C\{\left(T_{a-k}g_k\right)h_k^{\pm}: g_k \in \Pol_{a-k}[t]_\text{{\normalfont{even}}}, k\in K_{N,m}\},
\end{gathered}
\end{equation*}
Here, we regard $\Pol_{d}[t]_\text{{\normalfont{even}}} = \{0\}$ for $d < 0$.
\end{prop}

\begin{proof}
By (\ref{action-SO(2)-on-Pol}), the action of $A$ on $\Pol^a(\n_+)$ is given by:
\begin{equation*}
A \times \Pol^a(\n_+) \rightarrow \Pol^a(\n_+), \enspace (\exp(tH_0), p) \mapsto e^{-ta}p.
\end{equation*}
Thus, (i) is satisfied if and only if $e^{\lambda t} = e^{(\nu - a)t}$ for all $t\in \R$, or equivalently, if and only if $a = \nu-\lambda$. The proof of the proposition reduces now to the action of $SO(2)$, in which case the result is clear from Lemmas \ref{lemma-Harm} and \ref{Harm-Pol-isomorphism}.
\end{proof}

\section{Step 2: Solving the Differential Equation (\ref{F-system-2-concrete}), case $m \geq N$}\label{section-step2}
In this section, we examine the second step of the F-method for the case $m \geq N$. The case $m \leq -N$ can be proved to be deduced from the first one (see Section \ref{section-case_m_lessthan_-N}). 

Although we will focus mainly on the case $m = N$ and prove Theorems \ref{mainthm1} and \ref{mainthm2} in this section, we also prove some results that hold in general for $m \geq N$. In particular, we give a description of $\Sol(\n_+; \sigma_\lambda^{2N+1}, \tau_{m, \nu})$ in terms of the solution of a system of ordinary differential equations (see Theorem \ref{Thm-findingequations}). Since the way of solving this system is considerably different in the cases $m = N$ and $m >N$, in this paper we solve this system for $m = N$ and cover the case $m > N$ in another paper we expect to publish soon. 

\subsection{Description of $\Sol(\n_+; \sigma_\lambda^{2N+1}, \tau_{N, \nu})$ and proof of Theorems \ref{mainthm1} and \ref{mainthm2} for $m = N$}\label{section-proof-of-main-thms} 
Step 2 aims to characterize the space $\Sol(\n_+; \sigma_\lambda^{2N+1}, \tau_{m, \nu})$. By Proposition \ref{prop-FirststepFmethod}, this amounts to solving the system of partial differential equations (F-system) below
\begin{equation*}
\left(\widehat{d\pi_\mu}(C)\otimes \id_{\C_{m, \nu}}\right)\psi = 0 \quad (\forall C\in \n_+^\prime),
\end{equation*}
where  $\psi = \sum_{k \in K_{N, m}}\left(T_{\nu-\lambda-k}g_k\right)h_k^{\pm}$,  and  $g_k \in \Pol_{\nu-\lambda-k}[t]_\text{{\normalfont{even}}}$.
The next result gives a complete description of $\Sol(\n_+; \sigma_\lambda^{2N+1}, \tau_{m, \nu})$ for $m = N$.

\begin{thm}\label{Thm-step2} The following three conditions on the triple $(\lambda, \nu, N) \in \C^2 \times \N$ are equivalent.
\begin{enumerate}[label=\normalfont{(\roman*)}, topsep=3pt]
\item $\Sol(\n_+; \sigma_\lambda^{2N+1}, \tau_{N, \nu}) \neq \{0\}$.
\item $\dim_\C \Sol(\n_+; \sigma_\lambda^{2N+1}, \tau_{N, \nu}) = 1.$
\item $\nu - \lambda \in \N$.
\end{enumerate}
Moreover, if one of the above (therefore all) conditions is satisfied, then
\begin{equation*}
\Sol(\n_+; \sigma_\lambda^{2N+1}, \tau_{N, \nu}) = \C\sum_{k = 0}^{2N}\left(T_{a -k}g_k\right)h_k^+,
\end{equation*}
where $a := \nu - \lambda$ and the polynomials $g_{k}(t) \enspace (k = 0, \dots, 2N)$ are given as follows:
\begin{equation}\label{solution-all}
g_k(t) = i^{N+k}A_k \Geg_{a-k}^{\lambda+N-1}(it).
\end{equation}
Here, $\Geg_\ell^\mu(z)$ denotes the renormalized Gegenbauer polynomial (see \eqref{Gegenbauer-polynomial(renormalized)}), and we regard $\Geg_d^\mu \equiv 0$ for $d < 0$. The constant $A_k$ is defined as in \eqref{const-A}.
\end{thm}
We will discuss the proof of Theorem \ref{Thm-step2} in Section \ref{section-reductiontheorems}. Using Theorem \ref{Thm-step2}, we first establish proofs of Theorems \ref{mainthm1} and \ref{mainthm2}.

\begin{proof}[Proof of Theorem \ref{mainthm1} for $m =N$]
From (\ref{isomorphism-Fmethod-concrete}) we know that there exists a linear isomorphism
\begin{equation*}
\begin{tikzcd}[column sep = 1cm]
\Diff_{G^\prime}\left(C^\infty(S^3, \V_\lambda^{2N+1}), C^\infty(S^2, \L_{N,\nu})\right) \arrow[r, "\displaystyle{_{\Symb \otimes \id}}", "\thicksim"'] & \Sol(\n_+; \sigma_\lambda^{2N+1}, \tau_{N, \nu}).
\end{tikzcd}
\end{equation*}
Theorem \ref{mainthm1} follows now from Theorem \ref{Thm-step2}.
\end{proof}

\begin{proof}[Proof of Theorem \ref{mainthm2} for $m = N$]
By \eqref{isomorphism-Fmethod-concrete} and Theorem \ref{Thm-step2}, it is enough to compute the inverse of the symbol map $\Symb \otimes \id$ of $\psi = \sum_{k = 0}^{2N}\left(T_{a-k}g_k\right)(\zeta)h_k^+$, where $g_k$ are the polynomials given in \eqref{solution-all} and $h_k^+$ are defined in \eqref{generators-Hom(V2N+1,CmHk)}.

A straightforward computation using \eqref{Def-T_b} gives us
\begin{equation*}
\begin{array}{ll}
\left(T_{a-k}g_{k}\right)(\zeta_1, \zeta_2, \zeta_3)&=(-1)^\frac{a-k}{2}i^{N+k}A_k\left(I_{a-k}\Geg_{a-k}^{\lambda+N-1}\right)\left(-(\zeta_1^2 + \zeta_2^2), \zeta_3\right)\\
& = (-1)^\frac{a+N}{2}A_k\left(I_{a-k}\Geg_{a-k}^{\lambda+N-1}\right)\left(-(\zeta_1^2 + \zeta_2^2), \zeta_3\right),
\end{array}
\end{equation*}
where $(I_b\Geg_\ell^\mu)(x,y) := x^{\frac{b}{2}}\Geg_\ell^\mu\left(\frac{y}{\sqrt{x}}\right)$. Hence, by \eqref{generators-Hom(V2N+1,CmHk)} and by dividing each coordinate by $(-1)^\frac{a+N}{2}$, we obtain that the generator of (\ref{DSBO-space}) is
\begin{equation*}
\sum_{k=0}^{2N}A_k \widetilde{\C}_{\lambda+N, \nu+N-k} \left(\frac{\partial}{\partial x_1} + i \frac{\partial}{\partial x_2}\right)^k \otimes u_{k}^\vee
\end{equation*}
where
\begin{equation*}
\begin{gathered}
\widetilde{\C}_{\lambda, \nu} := \Rest_{x_3 = 0} \circ \left(I_{\nu-\lambda} \widetilde{C}^{\lambda-1}_{\nu-\lambda}\right)\left(-\Delta_{\R^2}, \frac{\partial}{\partial x_3}\right).
\end{gathered}
\end{equation*}
 
Multiplying by $2^k$ and using the identities below, which arise from the change of variables $z = x_1 + ix_2$ in the generator above, give us the expression (\ref{Operator-general+}) of the operator $\D_{\lambda, \nu}^{N, N}$.
\begin{equation*}
2 \frac{\partial}{\partial \overline{z}} = \left(\frac{\partial}{\partial x_1} + i\frac{\partial}{\partial x_2}\right), \quad 4\frac{\partial^2}{\partial z \partial \overline{z}} = \Delta_{\R^2}.
\end{equation*}
\end{proof}

\subsection{Reduction Theorems}\label{section-reductiontheorems}
Theorem \ref{Thm-step2} will be divided into to reduction theorems: Theorem \ref{Thm-findingequations} (finding equations), and Theorem \ref{Thm-solvingequations} (solving equations). We present them in this subsection and provide the proofs in the following sections. From now on, we assume $a = \nu - \lambda\in \N$, $N\in \N$ and $m \geq N$.

We recall the definition of the \textit{imaginary Gegenbauer differential operator} $S_\ell^\mu$, which is given as follows for any $\mu \in \C$ and any $\ell \in \N$ (see \eqref{Gegen-imaginary}):
\begin{equation*}
S_\ell^\mu = - (1 + t^2) \frac{d^2}{dt^2} - (1 + 2\mu)t\frac{d}{dt} + \ell(\ell + 2\mu).
\end{equation*}
Now, for $f_0, f_{\pm 1}, f_{\pm 2} \ldots, f_{\pm N}  \in \C[t]$ set $\mathbf{f} := {}^t(f_{-N}, \ldots, f_0, \ldots, f_N)$ and let $L_j^{A, \pm}(\mathbf{f})(t)$ and $L_j^{B, \pm}(\mathbf{f})(t)$  $(j = 0, \dots, N)$ be the following polynomials:
\begin{equation}\label{expression-operators-L-AB}
\begin{array}{l}
\begin{rcases*}
L_j^{A, +}(\mathbf{f})(t) := S_{a+m-j}^{\lambda+j-1} f_j - 2(N-j)\frac{d}{dt}f_{j+1}.\\[6pt]

L_j^{A, -}(\mathbf{f})(t) := S_{a-m-j}^{\lambda+j-1} f_{-j} + 2(N-j)\frac{d}{dt}f_{-j-1}.
\end{rcases*} (j = 0, 1, \ldots, N)\\[30pt]

\begin{rcases*}
L_j^{B, +}(\mathbf{f})(t) := 2(-m(\lambda + a-1) + j(\lambda-1 + \vartheta_t))f_j \\
\hspace*{\fill}+ (N-j)\frac{d}{dt}f_{j+1} +(N+j) \frac{d}{dt}f_{j-1}.\\[6pt]

L_j^{B, -}(\mathbf{f})(t) := 2(m(\lambda + a-1) + j(\lambda-1 + \vartheta_t))f_{-j}\\
\hspace*{\fill} - (N+j)\frac{d}{dt}f_{-j+1} -(N-j) \frac{d}{dt}f_{-j-1}.
\end{rcases*} (j = 1, \ldots, N)
\end{array}
\end{equation}
Here, $\vartheta_t$ stands for the Euler operator $\vartheta_t := t\frac{d}{dt}$. For $j=0$, we set $L_0^{B,\pm} \equiv 0$.

Given $\mathbf{g} = (g_{k_d})_{d = 0}^{2N} = {}^t(g_{k_0}, \ldots, g_{k_{2N}})$ with $g_{k_d} \in \C[t]$, we write 
\begin{equation}\label{def-g-prime}
\widetilde{\mathbf{g}} = {}^t(\widetilde{g}_{-N}, \ldots, \widetilde{g}_N) := \begin{pmatrix}
& & 1\\
& \iddots &\\
1 & &
\end{pmatrix}\mathbf{g} = {}^t(g_{k_{2N}}, \ldots, g_{k_0}).
\end{equation}
Now, we have

\begin{thm}[\textbf{Finding equations for $m \geq N$}]\label{Thm-findingequations} Let $N \in \N$, $m \in \Z$ with $m \geq N$ and set $\psi = \sum_{k = m-N}^{m+N}\left(T_{a-k}g_k\right)h_k^+$. Then, the following two conditions on $\mathbf{g} := (g_k)_{k = m-N}^{m+N}$ are equivalent:
\begin{enumerate}[label=\normalfont{(\roman*)}, topsep=2pt]
\setlength{\itemsep}{3pt}
\item $\left(\widehat{d\pi_\mu}(C)\otimes \id_{\C_{m,\nu}}\right)\psi = 0$, for all $C \in \n_+^\prime$.
\item $L_j^{A, \pm}(\widetilde{\mathbf{g}}) = L_j^{B, \pm}(\widetilde{\mathbf{g}}) = 0$, for all $j = 0, \dots, N$.
\end{enumerate}
\end{thm}
We define
\begin{equation}\label{equation-space}
\Xi(\lambda, a, N, m) := \left\{(g_k)_{k=m-N}^{m+N} \in \bigoplus_{k = m-N}^{m+N} \Pol_{a-k}[t]_\text{{\normalfont{even}}} : \begin{array}{l}
L_j^{A,\pm}(\widetilde{\mathbf{g}}) = 0\\[4pt]
L_j^{B,\pm}(\widetilde{\mathbf{g}}) = 0\\[4pt]
\forall j = 0, \dots, N
\end{array}
\right\}.
\end{equation}
Then, we obtain
\begin{thm}[\textbf{Solving equations for $m = N$}]\label{Thm-solvingequations} Let $m = N$ and let $g_k \in \Pol_{a-k}[t]_\text{even}$ for $k = 0, 1, \ldots, 2N$. Then, the system $\Xi(\lambda, a, N, N)$ has, up to scalar multiple, a unique non-zero solution $(g_k)_{k = 0}^{2N}$ and it is given by \eqref{solution-all}.
\end{thm}

\begin{rem}\label{remark-Step2} Note that Theorem \ref{Thm-findingequations} holds for any $m \in \Z$ with $m\geq N$, while Theorem \ref{Thm-solvingequations} is stated for $m = N$. By using Theorem \ref{Thm-findingequations}, we can reduce the problem of construction and classification of all differential symmetry breaking operators in \eqref{DSBO-space} (i.e., Problems A and B), to the problem of solving the system of ordinary differential equations \eqref{equation-space} for $m \geq N$ in general. In Theorem \ref{Thm-solvingequations} we solve this system for $m = N$, allowing us to solve completely Problems A and B for $m = N$. We do not consider the case $m > N$ in the present paper, since the way of solving the system \eqref{equation-space} changes considerably. However, we expect to publish the resolution of \eqref{equation-space} in another paper soon.
\end{rem}

The proofs of Theorems \ref{Thm-findingequations} and \ref{Thm-solvingequations} will be given in Sections \ref{section-proof_of_finding_equations} and \ref{section-proof_of_solving_equations} respectively.

\section{Proof of Theorem \ref{Thm-findingequations} (Finding Equations)}\label{section-proof_of_finding_equations} 
This section is devoted to prove Theorem \ref{Thm-findingequations}. We start with the decomposition of the differential operator $\widehat{d\pi_\mu}(C_1^+)\otimes \id_{\C_{m,\nu}}$. Throughout the section, we suppose $m \geq N$.

\subsection{Decomposition of the equation $\left(\widehat{d\pi_\mu}(C_1^+)\otimes \id_{\C_{m,\nu}}\right)\psi=0$}
Using the same strategy as in \cite{kkp, perez-valdes}, we decompose the equation \eqref{F-system-2-concrete} 
into a \textit{scalar} part (a second-order differential operator), and a \textit{vector} part (a first-order differential operator). With a slight abuse of notation, we regard $\lambda$ as the one-dimensional representation $\C_\lambda$ of $A$ defined via exponentiation.

We consider the basis $\{C_j^-: 1 \leq j  \leq 3\}$ of $\n_-$ and write $(z_1, z_2, z_3)$ for the corresponding coordinates. We identify $\n_-^\vee \simeq \n_+$ and write $(\zeta_1, \zeta_2, \zeta_3)$ for the dual coordinates.

For $Y \in \g$, we define $d\pi_{(\sigma^{2N+1}, \lambda)^*}^\text{{\normalfont{scalar}}}(Y), d\pi_{(\sigma^{2N+1}, \lambda)^*}^\text{{\normalfont{vect}}}(Y) \in \Weyl(\n_-) \otimes \End\left((V^{2N+1})^\vee\right)$ by
\begin{align*}
d\pi_{(\sigma^{2N+1}, \lambda)^*}^\text{{\normalfont{scalar}}}(Y)
& = d\pi_{\lambda^*}(Y) \otimes \id_{(V^{2N+1})^\vee},\\
d\pi_{(\sigma^{2N+1}, \lambda)^*}^\text{{\normalfont{vect}}}(Y)F & =  -\sum_{\ell=1}^3 z_\ell F \circ d\sigma^{2N+1}\left([Y, C_\ell^-]\big\rvert_\m\right),
\end{align*}
for $F\in C^\infty\left(\n_-, (V^{2N+1})^\vee\right)$. Then, the map $d\pi_\mu \equiv d\pi_{(\sigma^{2N+1}, \lambda)^*}$ is given as follows (see \cite[Eq. (3.13)]{kob-pev1}):
\begin{equation*}
d\pi_{(\sigma^{2N+1}, \lambda)^*} = d\pi_{(\sigma^{2N+1}, \lambda)^*}^\text{{\normalfont{scalar}}} + d\pi_{(\sigma^{2N+1}, \lambda)^*}^\text{{\normalfont{vect}}}.
\end{equation*}
We call $d\pi_{(\sigma^{2N+1}, \lambda)^*}^\text{{\normalfont{scalar}}}$ and $d\pi_{(\sigma^{2N+1}, \lambda)^*}^\text{{\normalfont{vect}}}$ the \emph{scalar part} and the \emph{vector part} of $d\pi_{(\sigma^{2N+1}, \lambda)^*}$ respectively.
Now, by taking the algebraic Fourier transform, we have
\begin{equation}\label{dpimu-decomposition}
\begin{gathered}
\widehat{d\pi_{(\sigma^{2N+1}, \lambda)^*}}\otimes \id_{\C_{m,\nu}} = \widehat{d\pi_{\lambda^*}} \otimes \id_{\Hom(V^{2N+1}, \C_{m,\nu})} + \widehat{d\pi_{(\sigma^{2N+1}, \lambda)^*}^\text{{\normalfont{vect}}}} \otimes \id_{\C_{m, \nu}},
\end{gathered}
\end{equation}
with $\widehat{d\pi_{(\sigma^{2N+1}, \lambda)^*}^\text{{\normalfont{vect}}}}(Y) :=  -\sum_{\ell=1}^3 \frac{\partial}{\partial \zeta_\ell} \circ d\sigma^{2N+1}\left([Y, C_\ell^-]\big\rvert_\m\right)$ 
(see \cite[Prop. 3.5]{kkp}).

Recall that, since the Levi subgroup $L^\prime$ acts irreducibly on
the nilradical $\n_+^\prime(\R)$ of $\p^\prime(\R) = \l^\prime(\R) + \n_+^\prime(\R)$, the equation (\ref{F-system-2-concrete}) is satisfied for all $C \in \n_+^\prime$ if and only if it is satisfied for some non-zero $C_0 \in \n_+^\prime$. In our case, we take $C_0 = C_1^+$ (see (\ref{elements})), and by (\ref{dpimu-decomposition}), we think of the equation
\begin{equation}\label{dpimu-equation}
\left(\widehat{d\pi_{(\sigma^{2N+1}, \lambda)^*}^\text{{\normalfont{scalar}}}}(C_1^+)\otimes \id_{\C_{m,\nu}}\right)\psi + \left(\widehat{d\pi_{(\sigma^{2N+1}, \lambda)^*}^\text{{\normalfont{vect}}}}(C_1^+)\otimes \id_{\C_{m,\nu}}\right)\psi= 0,
\end{equation}
where
\begin{equation}\label{dpimu-scalar-vect-formulas}
\begin{aligned}
\widehat{d\pi_{(\sigma^{2N+1}, \lambda)^*}^\text{{\normalfont{scalar}}}}(C_1^+) &=
\widehat{d\pi_{\lambda^*}}(C_1^+) \otimes \id_{\left(V^{2N+1}\right)^\vee},\\
\widehat{d\pi_{(\sigma^{2N+1}, \lambda)^*}^\text{{\normalfont{vect}}}}(C_1^+) &= \sum_{\ell=1}^3 d\sigma^{2N+1}\left(2X_{\ell, 1}\right)\frac{\partial}{\partial \zeta_\ell}.
\end{aligned}
\end{equation}

For actual computations, it is convenient to rewrite $\widehat{d\pi_{(\sigma^{2N+1}, \lambda)^*}^\text{{\normalfont{scalar}}}}$ and $\widehat{d\pi_{(\sigma^{2N+1}, \lambda)^*}^\text{{\normalfont{vect}}}}$ in terms of their \textit{vector coefficients} $M_s^\text{{\normalfont{scalar}}}$ and $M_s^\text{{\normalfont{vect}}}$ with respect to the basis $\{u_0, u_1, \ldots, u_{2N}\}$ of $V^{2N+1}$ in (\ref{basis-V2N+1}). In other words, we put
\begin{equation*}
\begin{aligned}
\left(\widehat{d\pi_{(\sigma^{2N+1}, \lambda)^*}^\text{{\normalfont{scalar}}}}(C_1^+)\otimes \id_{\C_{m,\nu}}\right)\psi &= \sum_{s=0}^{2N} M_s^\text{{\normalfont{scalar}}}(\psi)u_s^\vee,\\
\left(\widehat{d\pi_{(\sigma^{2N+1}, \lambda)^*}^\text{{\normalfont{vect}}}}(C_1^+)\otimes \id_{\C_{m,\nu}}\right)\psi &= \sum_{s=0}^{2N} M_s^\text{{\normalfont{vect}}}(\psi)u_s^\vee.
\end{aligned}
\end{equation*}
Now, the differential equation (\ref{dpimu-equation}) is equivalent to
\begin{equation}\label{Ms-equal-zero}
M_s(\psi) = M_s^\text{{\normalfont{scalar}}}(\psi) + M_s^\text{{\normalfont{vect}}}(\psi) = 0, \enspace \text{ for all } s = 0, 1, \ldots, 2N.
\end{equation}

\subsection{Vector coefficients of $\left(\widehat{d\pi_{(\sigma^{2N+1}, \lambda)^*}}(C_1^+)\otimes \id_{\C_{m,\nu}}\right)\psi$}\label{section-vector-coefficients}
For $\psi = \sum_{s=0}^{2N} \psi_su_s^\vee \in \Hom_\C\left(V_\lambda^{2N+1}, \C_{m, \nu} \otimes \Pol^a(\n_+)\right)$, the vector coefficients of the scalar part $M_s^\text{scalar}(\psi)$ are clearly given by
\begin{equation}\label{vector-coefficients-scalar}
M_s^\text{{\normalfont{scalar}}}(\psi) = \widehat{d\pi_{\lambda^*}^\text{{\normalfont{scalar}}}}(C_1^+)\psi_s \quad (s = 0, 1, \ldots, 2N).
\end{equation}
On the other hand, by \eqref{dpimu-scalar-vect-formulas}, the vector coefficients of the vector part $M_s^\text{vect}(\psi)$ are given by
\begin{equation}\label{vector-coefficients-vect}
M_s^\text{{\normalfont{vect}}}(\psi) = \sum_{s^\prime = 0}^{2N} A_{ss^\prime}\psi_{s^\prime}\quad (s = 0, 1, \ldots, 2N),
\end{equation}
where
\begin{equation*}
A_{ss^\prime} := \sum_{\ell = 1}^3 (d\sigma^{2N+1}(2X_{\ell,1})_{ss^\prime})\frac{\partial}{\partial \zeta_\ell} \quad (s, s^\prime = 0, 1, \ldots, 2N).
\end{equation*}

Moreover, we have
\begin{lemma}\label{lemma-coefficients-Ms-scalar-vector} For $\psi = \sum_{s=0}^{2N} \psi_su_s^\vee \in \Hom_\C\left(V_\lambda^{2N+1}, \C_{m, \nu} \otimes \Pol^a(\n_+)\right)$, we have
\begin{equation}\label{Ms-scalar}
M_s^\text{{\normalfont{scalar}}}(\psi) = \left(2\lambda \frac{\partial}{\partial \zeta_1} + 2E_{\zeta}\frac{\partial}{\partial \zeta_1} - \zeta
_1 \Delta^\zeta_{\C^3}\right)\psi_s \quad (s = 0, 1, \ldots, 2N),
\end{equation}
\begin{equation}\label{Ms-vect}
\left(M_s^\text{{\normalfont{vect}}}(\psi)\right)_{s=0}^{2N}
= \begin{pmatrix}
-2iN\frac{\partial}{\partial \zeta_2} & 2N\frac{\partial}{\partial \zeta_3} & 0 & \cdots & 0 \\[6pt]
-\frac{\partial}{\partial \zeta_3} & 2i(1-N)\frac{\partial}{\partial \zeta_2} & (2N-1)\frac{\partial}{\partial \zeta_3} & \ddots & \vdots\\[6pt]
0 & -2\frac{\partial}{\partial \zeta_3} & 2i(2-N)\frac{\partial}{\partial \zeta_2} & \ddots & 0\\[6pt]
\vdots & \ddots & \ddots & \ddots & \frac{\partial}{\partial \zeta_3}\\[6pt]
0 & \cdots & 0 & -2N\frac{\partial}{\partial \zeta_3} & 2iN\frac{\partial}{\partial \zeta_2}
\end{pmatrix}
\left(\psi_s\right)_{s=0}^{2N}.
\end{equation}
\end{lemma}
\begin{proof}
The formula for the vector coefficients $M_s^\text{{\normalfont{scalar}}}$ follows from
\begin{equation*}
\widehat{d\pi_{\lambda^*}}(C_1^+) = 2\lambda \frac{\partial}{\partial \zeta_1} + 2E_{\zeta}\frac{\partial}{\partial \zeta_1} - \zeta
_1 \Delta^\zeta_{\C^3},
\end{equation*}
(see \cite[Lem. 6.5]{kob-pev2}). For the vector part $M_s^\text{{\normalfont{vect}}}$, we first compute the differential of $\sigma^{2N+1}$ on the elements $2X_{\ell,1}$, $(\ell = 1,2,3)$. To this end, we utilize the isomorphism $\m \simeq \so(3,\C)$ and identify each $X_{\ell,1}$ with $E_{1,\ell} - E_{\ell, 1} \in \so(3,\C)$ for $\ell = 1,2,3$. A direct computation shows that for $t \in \R$ one has
\begin{align*}
e^{tX_{1,1}} = I_3 = \varpi(I_2),
&& e^{tX_{2,1}} = \varpi(U(e^{it/2},0)),
&& e^{tX_{3,1}} = \varpi(U(\cos(t/2), -\sin(t/2)),
\end{align*}
where $\varpi: SU(2) \twoheadrightarrow SO(3)$ is the covering map (\ref{expression-covering}). Thus, we obtain
\begin{equation*}
\begin{aligned}
d\sigma^{2N+1}(2X_{1,1}) &= 0,\\[6pt]
d\sigma^{2N+1}(2X_{2,1}) &=
2i\begin{pmatrix}
-N \\
& 1-N \\
& & \ddots \\
& & & 0 \\
& & & & \ddots \\
& & & & & N-1\\
& & & & & & N
\end{pmatrix},\\[8pt]
d\sigma^{2N+1}(2X_{3,1}) &= 
\begin{pmatrix}
0 & 2N & 0 & \cdots & 0\\
1 & 0 & 2N-1 & \ddots & \vdots\\
0 & -2 & 0 & \ddots & 0\\
\vdots & \ddots & \ddots & \ddots & 1\\
0 & \cdots & 0 & -2N & 0
\end{pmatrix}.
\end{aligned}
\end{equation*}
Therefore, the matrix $(A_{ss^\prime})$ is given by
\begin{equation*}
(A_{ss^\prime}) = \begin{pmatrix}
-2iN\frac{\partial}{\partial \zeta_2} & 2N\frac{\partial}{\partial \zeta_3} & 0 & \cdots & 0 \\[6pt]
-\frac{\partial}{\partial \zeta_3} & 2i(1-N)\frac{\partial}{\partial \zeta_2} & (2N-1)\frac{\partial}{\partial \zeta_3} & \ddots & \vdots\\[6pt]
0 & -2\frac{\partial}{\partial \zeta_3} & 2i(2-N)\frac{\partial}{\partial \zeta_2} & \ddots & 0\\[6pt]
\vdots & \ddots & \ddots & \ddots & \frac{\partial}{\partial \zeta_3}\\[6pt]
0 & \cdots & 0 & -2N\frac{\partial}{\partial \zeta_3} & 2iN\frac{\partial}{\partial \zeta_2}
\end{pmatrix}.
\end{equation*}
\end{proof}

\subsection{Explicit formul\ae{} for $M_s(\psi)$} 
In this section, for an element $\psi \in \Hom_{L^\prime}\left(V^{2N+1}_\lambda, \C_{m, \nu} \otimes \Pol^a(\n_+)\right)$
\begin{equation}\label{expression-psi}
\psi= \sum_{k=m-N}^{m+N}\left(T_{a-k} g_k\right)h_k^+ = \sum_{k=m-N}^{m+N}\left(T_{a-k} g_k\right)(\zeta_1 + i\zeta_2)^k u_{k}^\vee,
\end{equation}
we give explicit formulas of the vector coefficients
\begin{equation}\label{decomposition-Ms}
M_s(\psi) = M_s^\text{{\normalfont{scalar}}}(\psi) + M_s^\text{{\normalfont{vect}}}(\psi) \quad (s = 0, 1, \ldots, 2N),
\end{equation}
in terms of the polynomials $L_j^{A, \pm}(\widetilde{\mathbf{g}}), L_j^{B, \pm}(\widetilde{\mathbf{g}})$ ($j = 0, \ldots, N$) defined in (\ref{expression-operators-L-AB}). To do this, we first recall the notion of the \emph{$T$-saturation}.

Recalling from (\ref{Def-T_b}), for $\ell \in \N$, the linear map $T_\ell: \C[t] \longrightarrow \C(\zeta_1, \zeta_2, \zeta_3)$ is defined by
\begin{equation*}
(T_\ell g)(\zeta) := Q_2(\zeta^\prime)^\frac{\ell}{2} g\left(\frac{\zeta_3}{\sqrt{Q_2(\zeta^\prime)}}\right),
\end{equation*}
where $Q_2(\zeta^\prime) := \zeta_1^2 + \zeta_2^2$.
A differential operator $D$ on $\C^3$ is said to be $T$-\textbf{saturated} if there exists an operator $S$ on $\C[t]$ such that the following diagram commutes
\begin{equation*}
\begin{tikzcd}
\C[t]\arrow[r,"T_\ell"]\arrow[d, "S"'] & \C(\zeta_1, \zeta_2, \zeta_3)\arrow[d, "D"]\\
\C[t]\arrow[r, "T_\ell"] & \C(\zeta_1, \zeta_2, \zeta_3)
\end{tikzcd}
\end{equation*}
Such operator $S$ is unique whenever it exists and is denoted by $S = T_\ell^\sharp(D)$. In other words, $T_\ell^\sharp$ satisfies
\begin{equation*}
T_\ell \circ T_\ell^\sharp(D) = D \circ T_\ell.
\end{equation*}
Moreover, the following multiplicative property holds whenever it is well-defined:
\begin{equation*}
T_\ell^\sharp(D_1 \cdot D_2) = T_\ell^\sharp(D_1) \cdot T_\ell^\sharp(D_2).
\end{equation*}

\begin{lemma}[{\cite[Lem 6.27]{kkp}}]\label{lemma-Tsaturation-formulas} Denote by $S_{\ell}^\mu$ and by $\vartheta_t$ the imaginary Gegenbauer differential operator \eqref{Gegen-imaginary}, and the Euler operator $\frac{d}{dt}$ respectively. Then, for any $\ell \in \N$ and any $g \in \Pol_\ell[t]_\text{{\normalfont{even}}}$  (see \text{{\normalfont{(\ref{def-Pol_even})}}}), the following identities hold.
\begin{enumerate}[label=\normalfont{(\arabic*)}, topsep=6pt, itemsep=6pt]
\item $\displaystyle{T_\ell^\sharp\left(\frac{Q_2(\zeta^\prime)}{\zeta_1}\widehat{
d\pi_{\lambda^*}}(C_1^+)\right) = S_{\ell}^{\lambda-1}}$.
\item $\displaystyle{\left(T_\ell g\right)(\zeta) = Q_2(\zeta^\prime)\left(T_{\ell-2}g\right)(\zeta)}$.
\item $\displaystyle{\frac{\partial}{\partial \zeta_j}\left(T_\ell g\right)(\zeta) = \frac{\zeta_j}{Q_2(\zeta^\prime)}T_\ell\left((\ell-\vartheta_t)g\right)(\zeta)}$, for $j=1,2$.
\item $\displaystyle{\frac{\partial}{\partial \zeta_3}\left(T_\ell g\right)(\zeta)= T_{\ell-1}\left(\frac{dg}{dt}\right)(\zeta)}$.
\end{enumerate}
\end{lemma}

\begin{lemma}\label{lemma-formulae-Ms-final} Let $\psi = \sum_{k = m-N}^{m+N}(T_{a-k}g_k)h_k^+$. Then, the vector coefficients $M_s(\psi) = M_s^\text{{\normalfont{scalar}}}(\psi) + M_s^\text{{\normalfont{vect}}}(\psi)$ $(s = 0, 1, \ldots, 2N)$ are given as follows:
\begin{equation*}
\begin{aligned}
M_s(\psi) = (&\zeta_1 + i\zeta_2)^{k-1} \bigg[\zeta_1^2 T_{a-k-2}\left(S_{a-k}^{\lambda - 1}g_{k} +(N-k+m)\frac{d}{dt}g_{k+1}\right)\\
&+ \zeta_2^2T_{a-k-2}\left( 2(m-k)(a-k -\vartheta_t) g_k - (N-k+m) \frac{d}{dt}g_{k+1} \right) \\
&+ \zeta_1 \zeta_2 T_{a-k-2} \left(iS_{a-k}^{\lambda-1}g_k + 2i(k-m)(a-k - \vartheta_t)g_k + 2i(N-k+m) \frac{d}{dt}g_{k+1} \right) \\
&+ T_{a-k}\left(2k(\lambda+a-1+m-k)g_k - (N+k-m)\frac{d}{dt} g_{k-1}\right)\bigg],\\
\text{ for } s = k-m&+N, \enspace k = m-N, \ldots, m, \ldots, m+N. 
\end{aligned}
\end{equation*}
\end{lemma}
\begin{proof}
By the expressions of the generators $h_k^+$ (see \eqref{generators-Hom(V2N+1,CmHk)}), we can write $\psi$ in coordinates as follows:
\begin{equation*}
\begin{gathered}
\psi = {}^t(\psi_s)_{s=0}^{2N} = {}^t(\psi_{k-m+N})_{k=m-N}^{m+N} = {}^t\left(\left(T_{a-k}g_k\right)(\zeta_1+i\zeta_2)^k\right)_{k=m-N}^{m+N}\\[4pt]
= \begin{pmatrix}
\displaystyle{(T_{a-m+N}g_{m-N})(\zeta) (\zeta_1 + i\zeta_2)^{m-N}}\\
\vdots\\
\displaystyle{(T_{a-m}g_{m})(\zeta) (\zeta_1+ i\zeta_2)^{m}}\\[4pt]
\vdots\\
\displaystyle{(T_{a-m-N}g_{m+N})(\zeta) (\zeta_1 + i\zeta_2)^{m+N}}
\end{pmatrix}.
\end{gathered}
\end{equation*}

Fix $k \in \{m-N, \ldots, m+N\}$ and set $s=k-m+N \in \{0, 1, \ldots, 2N\}$. By \cite[Lem. 4.6]{kkp} we obtain the following formula:
\begin{equation*}
\widehat{d\pi_{\lambda^*}}(C_1^+)\left((T_{a-k}g_k)h_k^+\right) = \widehat{d\pi_{\lambda^*}}(C_1^+)(T_{a-k}g_k)h_k^+ + 2(\lambda + a -1)(T_{a-k}g_k)\frac{\partial h_k^+}{\partial \zeta_1}.
\end{equation*}
Then, by \eqref{Ms-scalar} and Lemma \ref{lemma-Tsaturation-formulas}(1), we obtain the following expression of $M_s^\text{{\normalfont{scalar}}}$
\begin{equation*}
\begin{aligned}
M_{s}^\text{{\normalfont{scalar}}}(\psi) = &\enspace \frac{\zeta_1}{Q_2(\zeta^\prime)}T_{a-k}\left(S_{a-k}^{\lambda-1}g_{k} \right)(\zeta_1 +i\zeta_2)^k \\
&+ 2k(\lambda + a -1)(\zeta_1 + i \zeta_2)^{k-1}\left(T_{a-k}g_{k} \right),
\end{aligned}
\end{equation*}
and by (\ref{Ms-vect}) the following one of $M_s^\text{{\normalfont{vect}}}$:
\begin{equation*}
\begin{aligned}
M_s^\text{{\normalfont{vect}}}(\psi) = &-s\frac{\partial}{\partial \zeta_3}\psi_{s-1} + 2i(s-N)\frac{\partial}{\partial \zeta_2}\psi_s + (2N-s)\frac{\partial}{\partial \zeta_3}\\
= & -(N+k-m)\frac{\partial}{\partial \zeta_3}\left[\left(T_{a-k+1}g_{k-1}\right)(\zeta_1 + i\zeta_2)^{k-1}\right] \\
&+2i(k-m)\frac{\partial}{\partial \zeta_2} \left[\left(T_{a-k}g_{k}\right)(\zeta_1 +i\zeta_2)^k \right] \\
&+ (N-k+m)\frac{\partial}{\partial \zeta_3}\left[\left(T_{a-k-1}g_{k+1} \right)(\zeta_1 + i\zeta_2)^{k+1}\right].
\end{aligned}
\end{equation*}

Simplifying the expressions above by using Lemma \ref{lemma-Tsaturation-formulas} leads to
\begin{equation*}
M_{s}^\text{{\normalfont{scalar}}}(\psi) = \zeta_1(\zeta_1 +i\zeta_2)^k T_{a-k-2}\left(S_{a-k}^{\lambda-1}g_{k} \right)
+ 2k(\lambda + a -1)(\zeta_1 + i \zeta_2)^{k-1}\left(T_{a-k}g_{k} \right),
\end{equation*}
and
\begin{equation*}
\begin{aligned}
M_s^\text{{\normalfont{vect}}}(\psi) = &-(N+k-m)\frac{\partial}{\partial \zeta_3}\left(T_{a-k+1}g_{k-1}\right)(\zeta_1 + i\zeta_2)^{k-1}\\
&+2i(k-m)\frac{\partial}{\partial \zeta_2}\left(T_{a-k}g_{k}\right)(\zeta_1 +i\zeta_2)^k
-2k(k-m)(\zeta_1 +i\zeta_2)^{k-1}\left(T_{a-k}g_{k}\right)\\
&+ (N-k+m)\frac{\partial}{\partial \zeta_3}\left(T_{a-k-1}g_{k+1} \right)(\zeta_1 + i\zeta_2)^{k+1}\\
= &-(N+k-m)T_{a-k}\left(\frac{d}{dt} g_{k-1}\right)(\zeta_1 + i\zeta_2)^{k-1}\\
&+2i(k-m)\frac{\zeta_2}{Q_2(\zeta^\prime)}T_{a-k}\left((a-k-\vartheta_t)g_{k}\right)(\zeta_1 +i\zeta_2)^k\\
&+2k(m-k)(\zeta_1 +i\zeta_2)^{k-1}\left(T_{a-k}g_{k}\right)\\
&+ (N-k+m)T_{a-k-2}\left(\frac{d}{dt} g_{k+1}\right)(\zeta_1 + i\zeta_2)^{k+1}\\
= & \;(\zeta_1 +i\zeta_2)^{k-1}T_{a-k}\left(2k(m-k)g_{k}- (N+k-m)\frac{d}{dt} g_{k-1}\right)\\
&+2i(k-m)\zeta_2(\zeta_1 +i\zeta_2)^k T_{a-k-2}\left((a-k-\vartheta_t)g_{k}\right)\\
&+(N-k+m)(\zeta_1 + i\zeta_2)^{k+1}T_{a-k-2}\left(\frac{d}{dt} g_{k+1}\right).
\end{aligned}
\end{equation*}
where we applied Lemma \ref{lemma-Tsaturation-formulas}(2) for $M_s^\text{{\normalfont{scalar}}}$, and Lemma \ref{lemma-Tsaturation-formulas}(3)-(4) and (2) in the last two equalities of $M_s^\text{{\normalfont{vect}}}$ respectively. Now, by merging the scalar and vector components of $M_s$, we arrive at
\begin{equation*}
\begin{aligned}
M_s(\psi) &= M_s^\text{{\normalfont{scalar}}}(\psi) + M_s^\text{{\normalfont{vect}}}(\psi) \\
& = (\zeta_1 + i\zeta_2)^{k-1} \Big[\zeta_1(\zeta_1 + i\zeta_2)
T_{a-k-2}\left(S_{a-k}^{\lambda - 1} g_k\right) \\
& +2i(k-m)\zeta_2(\zeta_1+i\zeta_2)T_{a-k-2}\left((a-k-\vartheta_t)g_{k}\right)\\
&+ (N-k+m)(\zeta_1+i\zeta_2)^2T_{a-k-2}\left(\frac{d}{dt} g_{k+1}\right)\\
&+ T_{a-k}\left(2k(m-k +\lambda+a-1)g_{k}-(N+k-m)\frac{d}{dt} g_{k-1}\right)\Big].
\end{aligned}
\end{equation*}

Rearranging the terms above yields the desired formula.
\end{proof}

\subsection{Proof of Theorem \ref{Thm-findingequations}}\label{section-proof_of_finding_equations-subsection}
In this subsection we prove Theorem \ref{Thm-findingequations}, relying on the following key lemma, which was established in \cite{perez-valdes}.

\begin{lemma}[\cite{perez-valdes}] \label{lemma-invariant-polynomials} Suppose that $p_1, \ldots, p_4 \in \Pol(\C^3)$ are $O(2,\C)$-invariant polynomials with respect to the action given by {\normalfont{(\ref{action-SO(2)-on-Pol})}}. Then, the following two conditions on $p_j$ are equivalent:
\begin{enumerate}[label=\normalfont{(\roman*)}, topsep=0pt]
\item $\zeta_1^2 p_1 + \zeta_2^2 p_2 + \zeta_1\zeta_2 p_3 + p_4 = 0$.
\item $p_1 = p_2$, $p_3 = 0$, and $Q_2(\zeta^\prime) p_1 + p_4 = 0$.
\end{enumerate}
\end{lemma}

For $\mathbf{g} = (g_k)_{k=m-N}^{m+N}$ let $R_k^1(\mathbf{g})(t)$ and $R_k^2(\mathbf{g})(t)$ be the following polynomials:
\begin{equation*}
\begin{array}{l}
R_k^1(\mathbf{g})(t) := S_{a-k}^{\lambda-1} g_k -2(m-k)(a-k-\vartheta_t)g_k + 2(N-k+m)\frac{d}{dt}g_{k+1},\\[6pt]

R_k^2(\mathbf{g})(t) := S_{a-k}^{\lambda-1} g_k +2k(\lambda+a-1 +m-k)g_k - (N+k-m)\frac{d}{dt}g_{k-1} + (N-k+m)\frac{d}{dt}g_{k+1}.
\end{array}
\end{equation*}
These polynomials have a close relation with the polynomials $L_j^{A, \pm}(\widetilde{\mathbf{g}})(t)$ and $L_j^{B, \pm}(\widetilde{\mathbf{g}})(t)$ defined in \eqref{expression-operators-L-AB} as the following lemma shows.
\begin{lemma}\label{lemma-equivalence-Rk-Lj}
Let $\mathbf{g}= (g_k)_{k=m-N}^{m+N}$ with $g_k \in \C[t]$ and define $\widetilde{\mathbf{g}}$ as in \eqref{def-g-prime}. Then, the following equivalence holds:
\begin{equation*}
R_k^1(\mathbf{g}) = R_k^2(\mathbf{g}) = 0 \enspace (\forall k=m-N, \dots, m+N) \Leftrightarrow L_j^{A, \pm}(\widetilde{\mathbf{g}}) = L_j^{B, \pm}(\widetilde{\mathbf{g}}) = 0 \enspace (\forall j=0, \dots, N).
\end{equation*}
\end{lemma}
\begin{proof}
First, observe that by the definition of $\widetilde{\mathbf{g}}$, we have $\widetilde{g}_{\pm j} = g_{m\mp j}$ for $j = 0, \ldots, N$, or equivalently, $\widetilde{g}_{m-k} = g_{k}$ for $k = m-N, \ldots, m+N$.

Suppose that $k = m-N, m-N+1, \ldots, m$. Then, a direct computation by using \eqref{Slmu-identity-1} repeatedly shows that
\begin{equation*}
\begin{aligned}
2R_k^2(\mathbf{g}) - R_k^1(\mathbf{g}) = & \enspace S_{a-k}^{\lambda-1}g_k + 2(m-k)(a-k-\vartheta_t)g_k \\
&+ 4k(\lambda + a-1 + m-k)g_k - 2(N+k-m)\frac{d}{dt}g_{k-1}\\
= & \enspace S_{a-k}^{\lambda + m-k - 1}g_k + 4k(\lambda + a-1 + m-k)g_k - 2(N+k-m)\frac{d}{dt}g_{k-1}\\
= & \enspace S_{a+k}^{\lambda +m-k-1}g_k - 2(N+k-m)\frac{d}{dt}g_{k-1},
\end{aligned}
\end{equation*}
which is precisely $L_{j}^{A, +}(\widetilde{\mathbf{g}})$ with $j = m-k = 0, 1, \ldots, N$. On the other hand, for $k = m, m+1, \ldots, m+N$, and again by \eqref{Slmu-identity-1}, we have that
\begin{equation*}
\begin{aligned}
R_k^1(\mathbf{g}) &= S_{a-k}^{\lambda-1} g_k -2(m-k)(a-k-\vartheta_t)g_k + 2(N-k+m)\frac{d}{dt}g_{k+1}\\
& = S_{a-k}^{\lambda-m+k -1} g_k + 2(N-k+m)\frac{d}{dt}g_{k+1},
\end{aligned}
\end{equation*}
which is precisely $L_{j}^{A, -}(\widetilde{\mathbf{g}})$ with $j = k-m = 0, 1, \ldots, N$. 

Similarly, for $k = m-N, m-N+1, \ldots, m-1$, we have that
\begin{equation*}
\begin{aligned}
R_k^1(\mathbf{g}) - R_k^2(\mathbf{g}) = &\enspace 2(-m(\lambda + a-1) + (m-k)(\lambda - 1 + \vartheta_t))g_k \\
&+ (N+m-k)\frac{d}{dt}g_{k+1} + (N-m+k)\frac{d}{dt}g_{k-1},
\end{aligned}
\end{equation*}
which is $L_{j}^{B, +}(\widetilde{\mathbf{g}})$ with $j = m-k = 1, 2, \ldots, N$. The identity above is also valid for $k = m+1, m+2, \ldots, m+N$, which is equal to $-L_{j}^{B, -}(\widetilde{\mathbf{g}})$ with $j = k-m = 1, 2, \ldots, N$. Thus, we have proved the following:
\begin{equation*}
\begin{aligned}
2R_k^2(\mathbf{g}) - R_k^1(\mathbf{g}) &= L_{j}^{A, +}(\widetilde{\mathbf{g}}) & \text{ for } j = 0, 1, \ldots, N, \enspace k = m-j,\\
R_k^1(\mathbf{g}) &= L_{j}^{A, -}(\widetilde{\mathbf{g}}) & \text{ for } j = 0, 1, \ldots, N, \enspace k = m+j,\\
R_k^1(\mathbf{g}) - R_k^2(\mathbf{g}) &= L_{j}^{B, +}(\widetilde{\mathbf{g}}) & \text{ for } j = 1, 2, \ldots, N, \enspace k = m-j,\\
R_k^1(\mathbf{g}) - R_k^2(\mathbf{g}) &= -L_{j}^{B, -}(\widetilde{\mathbf{g}}) & \text{ for } j = 1, 2, \ldots, N, \enspace k = m+j.
\end{aligned}
\end{equation*}
The lemma clearly follows now.
\end{proof}
The following proposition leads to the proof of Theorem \ref{Thm-findingequations}.
\begin{prop}\label{prop-equivalence-Ms-Rk}
Let $M_s \enspace (s = 0, \ldots, 2N)$ be the vector coefficients defined in {\normalfont{(\ref{decomposition-Ms})}} and let $\mathbf{g}= (g_k)_{k=m-N}^{m+N}$. Then, the following holds:
\begin{equation*}
M_{N+k-m} = 0 \Leftrightarrow R_k^1(\mathbf{g}) = R_k^2(\mathbf{g}) = 0, \text{ for all } k= m-N, \ldots, m+N. 
\end{equation*}
\end{prop}
\begin{proof}
Let $k = m-N, \dots, m+N$ and set
\begin{equation*}
\begin{aligned}
p_1 := &\enspace T_{a-k-2}\left(S_{a-k}^{\lambda - 1}g_{k} +(N-k+m)\frac{d}{dt}g_{k+1}\right),\\
p_2 :=  &\enspace T_{a-k-2}\left( 2(m-k)(a-k -\vartheta_t) g_k - (N-k+m) \frac{d}{dt}g_{k+1} \right),\\
p_3 :=  &\enspace T_{a-k-2} \left(iS_{a-k}^{\lambda-1}g_k + 2i(k-m)(a-k - \vartheta_t)g_k + 2i(N-k+m) \frac{d}{dt}g_{k+1} \right) \\
= &\enspace -i(p_2-p_1),\\
p_4 := &\enspace T_{a-k}\left(2k(\lambda+a-1+m-k)g_k - (N+k-m)\frac{d}{dt} g_{k-1}\right).
\end{aligned}
\end{equation*}

From the expression of $M_s$ obtained in Lemma \ref{lemma-formulae-Ms-final} we deduce that 
\begin{equation*}
M_s = 0 \Leftrightarrow \zeta_1^2 p_1 + \zeta_2^2 p_2 + \zeta_1 \zeta_2 p_3 + p_4 = 0,
\end{equation*}
and by Lemma \ref{lemma-invariant-polynomials}, this is equivalent to
\begin{equation*}
p_1 - p_2 = p_3 = \enspace Q_2(\zeta^\prime)p_1 + p_4 = 0,
\end{equation*}
in other words, to
\begin{equation*}
\begin{aligned}
& T_{a-k-2}\left(S_{a-k}^{\lambda - 1}g_{k} - 2(m-k)(a-k -\vartheta_t) g_k +2(N-k+m)\frac{d}{dt}g_{k+1}\right) & = 0,\\
&Q_2(\zeta^\prime)T_{a-k-2}\left(S_{a-k}^{\lambda - 1}g_{k} +(N-k+m)\frac{d}{dt}g_{k+1}\right) &\\ 
& \hspace{1.5cm} + T_{a-k}\left(2k(\lambda+a-1+m-k)g_k - (N-k+m)\frac{d}{dt} g_{k-1}\right) & = 0.
\end{aligned}
\end{equation*}

Since $T_\ell$ is a bijection (see (\ref{bijection-T_b})), and by Lemma \ref{lemma-Tsaturation-formulas} (2), the two equations above are equivalent to
\begin{equation*}
\begin{aligned}
S_{a-k}^{\lambda - 1}g_{k} - 2(m-k)(a-k -\vartheta_t) g_k +2(N-k+m)\frac{d}{dt}g_{k+1} &= 0,\\
S_{a-k}^{\lambda - 1}g_{k} + 2k(\lambda+a-1+m-k)g_k - (N+k-m)\frac{d}{dt} g_{k-1} +(N-k+m)\frac{d}{dt}g_{k+1} &= 0,
\end{aligned}
\end{equation*}
which are nothing but $R_k^1(\mathbf{g}) = R_k^2(\mathbf{g}) = 0$. Thus, the equivalence $M_{N+k-m} = 0 \Leftrightarrow R_k^1(\mathbf{g}) = R_k^2(\mathbf{g}) = 0, \text{ for all } k= m-N, \ldots, m+N$ has been proved.
\end{proof}
We now present a proof of Theorem \ref{Thm-findingequations}.
\begin{proof}[Proof of Theorem \ref{Thm-findingequations}]
As previously noted, (i) holds for all $C \in \n_+^\prime$ if and only if it holds for a single element of $\n_+^\prime$, say $C_1^+ \in \n_+^\prime$ (cf. \cite[Lem. 3.4]{kkp}). In other words, (i) is equivalent to
\begin{equation*}
\left(\widehat{d\pi_\mu}(C_1^+)\otimes \id_{\C_{m,\nu}}\right)\psi = 0,
\end{equation*}
which, in turn, is equivalent to $M_s(\psi) = M_s^\text{scalar}(\psi) + M_s^\text{vect}(\psi) = 0$ $(s=0, \ldots, 2N)$ as noted in (\ref{Ms-equal-zero}). The result now follows from Proposition \ref{prop-equivalence-Ms-Rk} and Lemma \ref{lemma-equivalence-Rk-Lj}.
\end{proof}

\section{Proof of Theorem \ref{Thm-solvingequations}: Solving the System $\Xi(\lambda, a, N, N)$}
\label{section-proof_of_solving_equations}

In this section we prove Theorem \ref{Thm-solvingequations}; in other words, we solve the system (\ref{equation-space}), which is given by the following differential equations for $f_{\pm j}\in \Pol_{a-m\pm j}[t]_\text{{\normalfont{even}}}$, ($j = 0,1, \ldots, N$):
\begin{equation}\label{ODEsystem}
\begin{array}{ll}
(A_j^+) & S_{a+N-j}^{\lambda+j-1} f_j - 2(N-j)\frac{d}{dt}f_{j+1} = 0.\\[6pt]

(A_j^-) & S_{a-N-j}^{\lambda+j-1} f_{-j} + 2(N-j)\frac{d}{dt}f_{-j-1} = 0.\\[6pt]

(B_j^+) & 2(-N(\lambda + a-1) + j(\lambda-1 + \vartheta_t))f_j 
+ (N-j)\frac{d}{dt}f_{j+1} +(N+j) \frac{d}{dt}f_{j-1} = 0.\\[6pt]

(B_j^-) & 2(N(\lambda + a-1) + j(\lambda-1 + \vartheta_t))f_{-j}
- (N+j)\frac{d}{dt}f_{-j+1} -(N-j) \frac{d}{dt}f_{-j-1} = 0.
\end{array}
\end{equation}
Recall that $j \in \{0, 1, \ldots, N\}$ in $(A_j^\pm)$ while $j \in \{1, \ldots, N\}$ in $(B_j^\pm)$.

The proof is a bit long and uses several technical results, including properties of renormalized Gegenbauer polynomials and imaginary Gegenbauer operators which are listed in the appendix (Section \ref{section-appendix}). Note that even if it is a linear system of ordinary differential equations with polynomial solutions (apparently easy), it is an overdetermined system, since we have $2(2N+1)$ equations and $2N+1$ functions.

In order to solve this system, we follow a strategy divided in three phases. Below we give a brief explanation of what does each phase consist of:

\begin{itemize}[leftmargin=2cm, topsep=3pt]
\setlength{\itemsep}{2pt}
\item[\textbf{Phase 1}:] Solve equations $(A_N^\pm)$ and obtain $f_{\pm N}$ up to two constants, say $q_N^\pm \in \C$.

\item[\textbf{Phase 2}:] Inductively, for $j = 0, \dots, N-1$, obtain $f_{\pm j}$ by solving $(B_j^\pm)$ and by using the expressions of $f_{\pm N}$ obtained in Phase 1. After that, check that these $f_{\pm j}$ satisfy equations $(A_j^{\pm})$. For $f_0$ we obtain two different expressions (one coming from the $+$ side and another coming from the $-$ side).

\item[\textbf{Phase 3}:] Check when the two expressions for $f_0$  obtained in Phase 2 coincide. Here, we obtain a compatibility condition for the constants $q_N^\pm$. 
\end{itemize}

We carry out this three-phase strategy in the following and give a proof of Theorem \ref{Thm-solvingequations} at the end in Section \ref{section-proof_of_solving_equations-subsection}.

After completing Phases 1 and 2, we will have that the space of solutions $\Xi(\lambda, a, N, N)$ has dimension less than or equal to $2$, since the polynomials $f_{\pm j}$ depend on two constants $q_N^\pm$ that a priori can take any value.  In Phase 3 we will deduce that this space has dimension less than or equal to one, since we obtain a compatibility condition (i.e., a linear relation) between $q_N^+$ and $q_N^-$. 

The main goal of this section is to prove that the solution $(f_{-N}, \ldots, f_N)$ of the system $\Xi(\lambda, a, N, N)$ is given (up to constant) by the following formula, which is a restatement of \eqref{solution-all}:
\begin{align}\label{expression-fpmj-final}
f_j(t) = (-i)^jA_{N-j}\Geg_{a+j-N}^{\lambda+N-1}(it), &&
f_{-j}(t) = i^j A_{N+j}\Geg_{a-j-N}^{\lambda+N-1}(it),
\end{align}
where we recall that $A_k$ is given by \eqref{const-A}.

\subsection{Phases 1 and 2: Obtaining $f_{\pm j}$ up to constant}\label{section-phase1}

The first phase is quite simple and straightforward. By taking a look at the system \eqref{ODEsystem} we find that:
\begin{align*}
(A_N^+) \enspace S_{a}^{\lambda + N - 1} f_N = 0, && (A_N^-) \enspace  S_{a-2N}^{\lambda + N - 1} f_{-N} = 0.
\end{align*}
Thus, from Theorem \ref{thm-Gegenbauer-solutions} and Lemma \ref{lemma-relationS-G}, we have
\begin{align}\label{expression_fN-and-f-N}
f_N(t) = q_N^+\Geg_{a}^{\lambda+N-1}(it), && f_{-N}(t) = q_N^-\Geg_{a-2N}^{\lambda+N-1}(it),
\end{align}
for some constants $q_N^\pm \in \C$. This completes Phase 1.

Let us start Phase 2. In this phase, we will obtain the rest of the functions $f_{\pm j}$. Before starting actual arguments, we describe the structure of the system and the hierarchy between its equations, explaining how can we solve them:

\begin{itemize}
\setlength{\itemsep}{1pt}
\item By using the expressions of $f_{\pm N}$ in \eqref{expression_fN-and-f-N}, we obtain $f_{\pm(N-1)}$ by solving $(B_{N}^\pm)$.

\item Next, by using the expressions of $f_{\pm N}$ and $f_{\pm(N-1)}$, we obtain $f_{\pm(N-2)}$ by solving $(B_{N-1}^\pm)$.

\item We repeat this process for any $0 < j < N$ and obtain $f_{\pm(j-1)}$ from the previous two functions $f_{\pm j}, f_{\pm (j+1)}$ by solving $(B_j^\pm)$.

\item At the end, for $j = 0$ we obtain $f_0$ by the same procedure, but this time we obtain two expressions; one that comes from the expressions of $f_j$ and another that comes from the expressions of $f_{-j}$. We will have to check that these two expressions coincide in Phase 3.
\end{itemize}
Therefore, the functions $f_{\pm j}$ can be obtained recursively by solving $(B_j^\pm)$. If we represent the order in which we obtain these functions in a diagram, we have the following (we put a zero in both sides to represent that we start from \lq\lq nothing\rq\rq):
\begin{equation}\label{diagram-simple}
\overbrace{0 \rightarrow f_{-N}}^\text{Phase 1} 
\underbrace{\rightarrow f_{-N+1} \rightarrow \cdots \rightarrow f_{-1} \rightarrow}_\text{Phase 2}
\overbrace{f_{0}}^\text{Phase 3}
\underbrace{\leftarrow f_1 \leftarrow \cdots \leftarrow f_{N-1} \leftarrow}_\text{Phase 2}
\overbrace{f_N \leftarrow 0}^\text{Phase 1}
\end{equation}

Until now, we have not talked about equations $(A_j^\pm)$, and in order to solve the system we have to prove that these equations are also satisfied. We do this also in Phase 2 while obtaining $f_{\pm j}$ recursively. More concretely, when in the step $j$ we solve $(B_j^\pm)$ and obtain $f_{\pm(j-1)}$ by using the expressions of $f_{\pm j}$ and $f_{\pm(j+1)}$, we actually obtain $f_{\pm(j-1)}$ up to subtraction by some constant term $c_{j-1}^\pm \in \C$ that comes from integrating $(B_j^\pm)$. Before going to the next step, we check that the obtained $f_{\pm(j-1)}$ solves $(A_{j-1}^\pm)$ if and only if $c_{j-1}^\pm = 0$. By doing this we obtain a complete expression of the functions $f_{\pm j}$ while checking that both $(A_j^\pm)$ and $(B_j^\pm)$ are satisfied. Therefore, the actual order in which we solve the equations is as follows:
\begin{equation*}
\begin{tikzcd}
(A_N^-) \arrow[d] &[-0.8cm] (A_{N-1}^-) \arrow[d] &[-0.8cm]\cdots &[-0.8cm] (A_1^-) \arrow[d]  &[-0.8cm](A_0^-) &[-0.8cm] (A_{0}^+) &[-0.8cm] (A_{1}^+) \arrow[d] &[-0.8cm] \cdots &[-0.8cm] (A_{N-1}^+) \arrow[d] &[-0.8cm] (A_N^+) \arrow[d]\\[-0.2cm]
(B_N^-) \arrow[ru] & (B_{N-1}^-) \arrow[ru] & \cdots \arrow[ru] &(B_1^-) \arrow[ru] &[-0.8cm] &[-0.8cm] &[-0.8cm] (B_{1}^+) \arrow[lu] &[-0.8cm] \cdots \arrow[lu] &[-0.8cm] (B_{N-1}^+) \arrow[lu] &[-0.8cm] (B_N^+) \arrow[lu]
\end{tikzcd}
\end{equation*}

In Figures \ref{diagram_N_1} to \ref{diagram_N_4} we write a little more complex diagrams for several values of $N \in \N$ that represent the hierarchy of the system in a more complete way than the diagram \eqref{diagram-simple}.

\begin{figure}
\centering
\includegraphics[scale=0.65]{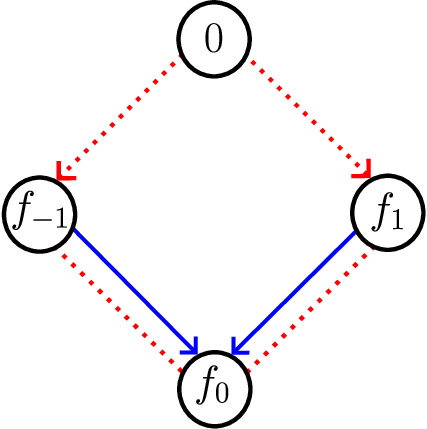}
\caption{Hierarchy for $N = 1$}\label{diagram_N_1}
\end{figure}
\begin{figure}
\centering
\includegraphics[scale=0.65]{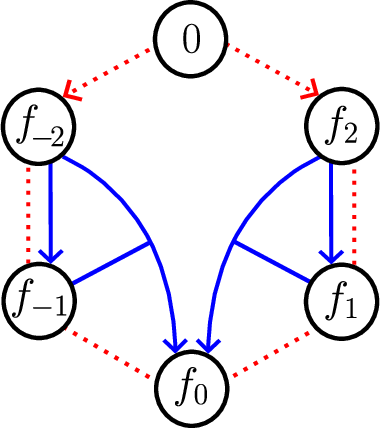}
\caption{Hierarchy for $N = 2$}\label{diagram_N_2}
\includegraphics[scale=0.65]{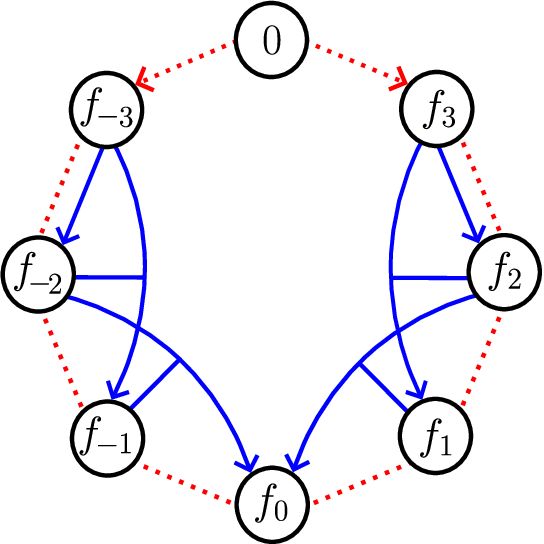}
\caption{Hierarchy for $N = 3$}\label{diagram_N_3}
\includegraphics[scale=0.65]{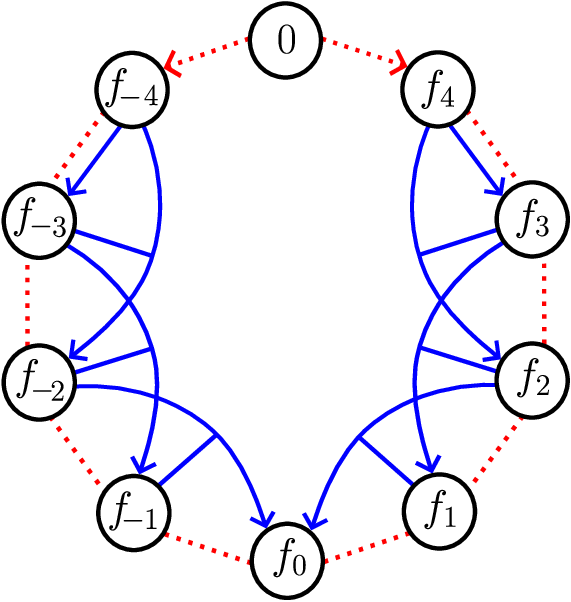}
\caption{Hierarchy for $N = 4$}\label{diagram_N_4}
\end{figure}

Concretely, we write the functions as vertices in a graph (including 0, that can be thought of as the starting point), and represent equations $(A_{j}^\pm)$ and $(B_{j}^\pm)$ as edges linking the functions that they relate. Equations $(A_{j}^\pm)$ are written as dotted red lines, while $(B_{j}^\pm)$ are written as solid blue lines. Moreover, between these edges, we chose a few and write them as arrows, representing the order in which we solve the equations of the system. The first two arrows are dotted and red (corresponding to solving $(A_N^\pm)$ in Phase 1), and the rest are solid and blue (since we obtain the rest of the functions through solving $(B_j^\pm)$ in Phase 2). We remark that in contrast with the dotted red edges (equations $(A_j^\pm)$), the solid blue arrows (equations $(B_j^\pm)$) link three vertices (except for the first one associated to $(B_N^\pm)$, that links only two), as one can deduce from \eqref{ODEsystem}.

Now that we have explained the hierarchy of the system and the order in which we obtain the polynomials $f_{\pm j}$, we prove it in detail in the results below.\\

For any $j = 0, 1, \ldots, N$ let $\Gamma_j^\pm$ be the following gamma factors:
\begin{align*}
\Gamma_j^+ := \frac{\Gamma\left(\lambda+N-1+\left[\frac{a+1}{2}\right]\right)}{\Gamma\left(\lambda+N-1+\left[\frac{a+j-N+1}{2}\right]\right)}, && \Gamma_j^- := \frac{\Gamma\left(\lambda+N-1+\left[\frac{a-2N+1}{2}\right]\right)}{\Gamma\left(\lambda+N-1+\left[\frac{a+j-3N+1}{2}\right]\right)}.
\end{align*}
Observe that we have
\begin{align*}
& \Gamma_j^+ = \begin{cases}
\gamma(\lambda+N-1, a-1)\gamma(\lambda+N-1, a-2) \cdots \gamma(\lambda+N-1, a+j-N), & \text{ if } 0 \leq j < N,\\
1, & \text{ if } j = N.
\end{cases}\\
& \Gamma_j^- = \begin{cases}
\gamma(\lambda, a-3)\gamma(\lambda, a-4) \cdots \gamma(\lambda, a+j-N-2), & \text{ if } 0 \leq j < N,\\
1, & \text{ if } j = N.
\end{cases}
\end{align*}
Here, $\gamma(\mu, \ell)$ is defined as in \eqref{gamma-def}. In particular, we have
\begin{equation}\label{gamma-recurrence-relation}
\begin{aligned}
&\Gamma_j^+ = \Gamma_{j+1}^+ \gamma(\lambda+N-1, a+j-N), \text{ and}
&&\Gamma_j^- = \Gamma_{j+1}^- \gamma(\lambda, a+j-N-2),
\end{aligned}
\end{equation}
for all $j = 0, \ldots, N-1$. Thus,
\begin{equation}\label{gamma-equivalence}
\Gamma_j^\pm \neq 0 \text{ for all } j = 0, 1, \ldots, N-1 \Leftrightarrow \Gamma_0^\pm \neq 0.
\end{equation}

We will consider several cases depending on the value of $a$ and whether $\Gamma_0^\pm$ is zero or not. In the following we suppose that $a\geq 2N$. We consider the case $N < a < 2N$ in the final part of this section, and the case $0 \leq a \leq N$ at the end in Section \ref{section-proof_of_solving_equations-subsection}. We start by the following useful lemma.
\begin{lemma} \label{lemma-recurrence-fj}
\begin{itemize}
\item[\normalfont{(1)}] For a given $j=1, \ldots, N-1$, suppose that $f_j$ and $f_{j+1}$ satisfy \eqref{expression-fj+} below, and suppose further that $\Gamma_{j-1}^+ \neq 0$. Then, $f_{j-1}$ solves $(A_{j-1}^+)$ and $(B_j^+)$ if and only if it also satisfies \eqref{expression-fj+}
\begin{equation}\label{expression-fj+}
(-i)^{N-j} \frac{\Gamma\left(\lambda+N-1+\left[\frac{a+1}{2}\right]\right)}{\Gamma\left(\lambda+N-1+\left[\frac{a+j-N+1}{2}\right]\right)} f_j(t) = q_N^+ \Geg_{a+j-N}^{\lambda +N-1}(it).
\end{equation}
\item[\normalfont{(2)}] For a given $j=1, \ldots, N-1$, suppose that $f_{-j}$ and $f_{-j-1}$ satisfy \eqref{expression-fj-} below, and suppose further that $\gamma(\lambda, a+j-N-3) \neq 0$. Then, $f_{-j+1}$ solves $(A_{j-1}^-)$ and $(B_j^-)$ if and only if it also satisfies \eqref{expression-fj-}
\begin{equation}\label{expression-fj-}
i^{N-j} f_{-j}(t) = \frac{\Gamma\left(\lambda+N-1+\left[\frac{a-j-N+1}{2}\right]\right)}{\Gamma\left(\lambda+N-1+\left[\frac{a-2N+1}{2}\right]\right)}q_N^-  \Geg_{a-j-N}^{\lambda +N-1}(it).
\end{equation}
\end{itemize}
\end{lemma}

\begin{proof}
(1) For a given $j = 1, \ldots, N-1$, suppose that $f_j$ and $f_{j+1}$ satisfy \eqref{expression-fj+}. Let us prove that $f_{j-1}$ solves $(B_{j}^+)$ and $(A_{j-1}^+)$ if and only if it also satisfies \eqref{expression-fj+}. From \eqref{gamma-recurrence-relation} and from $\Gamma_{j-1}^+ \neq 0$, we deduce that $\Gamma_j^+ \gamma(\lambda+N-1, a+j-1-N) \neq 0$. In particular,
\begin{equation*}
(-i)^{N-j}\Gamma_{j}^+ = \left((-i)^{N-j-1}\Gamma_{j+1}^+\right)(-i)\gamma(\lambda + N-1, a+j-N) \neq 0
\end{equation*}
Now, by \eqref{derivative-Gegenbauer-1}, \eqref{derivative-Gegenbauer-2} and \eqref{gamma-product-property}, equation $(B_j^+)$ multiplied by the constant above amounts to
\begin{align*}
-(-i)^{N-j}\Gamma_{j}^+(N+j)\frac{d}{dt}f_{j-1} =& 
-2q_N^+(N-j)(\lambda + a+j-1)\Geg_{a+j-N}^{\lambda+N-1}(it) + 4q_N^+j\Geg_{a+j-N-2}^{\lambda+N}(it)\\
 & + 2q_N^+(N-j)\left(\lambda+N-1 + \left[\frac{a+j-N+1}{2}\right]\right)\Geg_{a+j-N}^{\lambda+N}(it).
\end{align*}
The right-hand side can be simplified by using the three-term relation \eqref{KKP-1}, obtaining
\begin{equation*}
2q_N^+(N+j)\Geg_{a+j-N-2}^{\lambda+N}(it).
\end{equation*}
Now, dividing both sides by $(N+j)$ leads us to
\begin{equation*}
-(-i)^{N-j}\Gamma_{j}^+\frac{d}{dt}f_{j-1} = 2q_N^+\Geg_{a+j-N-2}^{\lambda+N}(it).
\end{equation*}
By using \eqref{derivative-Gegenbauer-1} and multiplying both sides by $i\gamma(\lambda+N-1, a+j-N-1) \neq 0$, we can integrate $f_{j-1}$ and obtain
\begin{equation*}
(-i)^{N-j+1}\Gamma_{j-1}^+f_{j-1}(t) = q_N^+\Geg_{a+j-N-1}^{\lambda+N-1}(it) + c_{j-1}^+,
\end{equation*}
for some constant $c_{j-1}^+ \in \C$. Now, we show that this $f_{j-1}$ solves $(A_{j-1}^+)$ if and only if $c_{j-1}^+ = 0$. From the definition, $(A_{j-1}^+)$ is satisfied if and only if
\begin{equation*}
S_{a+N-j+1}^{\lambda+j-2} f_{j-1} - 2(N-j+1)\frac{d}{dt}f_j = 0.
\end{equation*}
Multiplying this equation by 
\begin{equation*}
(-i)^{N-j+1}\Gamma_{j-1}^+ = \left((-i)^{N-j}\Gamma_j^+\right)(-i)\gamma(\lambda+N-1, a+j-1-N) \neq 0,
\end{equation*}
and using the expression of $f_j$, \eqref{derivative-Gegenbauer-1} and \eqref{derivative-Gegenbauer-2}, we have
\begin{align*}
&S_{a+N-j+1}^{\lambda+j-2}\left(q_N^+\Geg_{a+j-N-1}^{\lambda+N-1}(it) + c_{j-1}\right) \\
&- 4q_N^+(N-j+1)\left(\lambda+N-1 \left[\frac{a+j-N}{2}\right]\right)\Geg_{a+j-N-1}^{\lambda+N}(it) = 0.
\end{align*}
Now, by using \eqref{Slmu-identity-2} for $d = N-j+1$, this amounts to
\begin{multline*}
(a+N-j+1)(2\lambda + N + j + a-3)c_{j-1} \\
+ 4(N-j+1)\Geg_{a+j-N-3}^{\lambda + N-1}(it)
+ 4(N-j+1)(\lambda+a+j-2)\Geg_{a+j-N-1}^{\lambda+N}(it) \\
- 4q_N^+(N-j+1)\left(\lambda+N-1 \left[\frac{a+j-N}{2}\right]\right)\Geg_{a+j-N-1}^{\lambda+N}(it) = 0.
\end{multline*}
The last three terms vanish thanks to the three-term relation \eqref{KKP-1}. Hence, $(A_{j-1}^+)$ is satisfied if and only if
\begin{equation}\label{constant-cj-1}
(2\lambda + N + j + a-3)c_{j-1}^+ = 0.
\end{equation}
If $a + j - N$ is even, $c_{j-1}^+ = 0$ necessarily by the condition $f_{j-1} \in \Pol_{a-N+j-1}[t]_\text{even}$. In particular, \eqref{constant-cj-1} holds. On the other hand, if $a + j - N$ is odd, from $\Gamma_{j-1}^+ \neq 0$ we have
\begin{equation*}
2\gamma(\lambda +N-1, a+j-1-N) = 2\lambda + N + j + a - 3 \neq 0.
\end{equation*}
Hence, \eqref{constant-cj-1} holds if and only if $c_{j-1}^+ = 0$. In both cases we have proved that $c_{j-1}^+ = 0$ necessarily. Hence, $f_{j-1}$ is given by \eqref{expression-fj+} and it satisfies both $(B_j^+)$ and $(A_{j-1}^+)$.

(2) The proof of the second statement is very similar to that of the first one. However, there are some differences.  For a given $j = 1, \ldots, N-1$ suppose that $f_{-j}$ and $f_{-j-1}$ satisfy \eqref{expression-fj-}. Let us prove that $f_{-j+1}$ solves $(B_j^-)$ and $(A_{j-1}^-)$ if an only if it is given by \eqref{expression-fj-}. We proceed as before. A direct computation by using \eqref{derivative-Gegenbauer-2} and \eqref{KKP-1} shows that $(B_j^-)$ amounts to
\begin{align*}
i^{N-j}\frac{d}{dt}f_{-j+1} &=\\
& 2q_N^-\frac{\Gamma\left(\lambda+N-1+\left[\frac{a-j-N+1}{2}\right]\right)}{\Gamma\left(\lambda+N-1+\left[\frac{a-2N+1}{2}\right]\right)}\left(\lambda+N-1 +\left[\frac{a-N-j+1}{2}\right]\right)\Geg_{a-j-N}^{\lambda +N}(it),
\end{align*}
which by \eqref{gamma-product-property} and \eqref{derivative-Gegenbauer-1} amounts to
\begin{equation*}
i^{N-j+1}f_{-j+1}(t) = q_N^-\frac{\Gamma\left(\lambda+N-1+\left[\frac{a-j-N+2}{2}\right]\right)}{\Gamma\left(\lambda+N-1+\left[\frac{a-2N+1}{2}\right]\right)}\Geg_{a-j+1-N}^{\lambda +N-1}(it) + c_{j-1}^-,
\end{equation*}
for some constant $c_{j-1}^- \in \C$. Now, by \eqref{Slmu-identity-1} for $d = N-j+1$ and \eqref{derivative-Gegenbauer-2}, $(A_{j-1}^-)$ is easily shown to be equivalent to
\begin{equation}\label{constant-cj-1-}
(a-N-j+1)(2\lambda+a+j-N-3)c_{j-1}^- = 0.
\end{equation}
As before, if $a-j-N$ is even, $c_{j-1}^- = 0$ necessarily by the condition $f_{-j+1} \in \Pol_{a-N-j+1}[t]_\text{even}$. In particular \eqref{constant-cj-1-} is satisfied.
On the other hand, if $a-j-N$ is odd, by $\gamma(\lambda, a+j-N-3) \neq 0$ we have
\begin{equation*}
2\gamma(\lambda, a+j-N-3) = 2\lambda+a+j-N-3 \neq 0.
\end{equation*}
Therefore, in both cases, \eqref{constant-cj-1-} is satisfied if and only if $c_{j-1}^- = 0$, showing that $f_{-j+1}$ is given by \eqref{expression-fj-} and that it solves $(B_j^-)$ and $(A_{j-1}^-)$.
\end{proof}

As a consequence of the previous lemma we have the following result, that allows us to complete Phase 2 when $\Gamma_0^\pm \neq 0$.

\begin{lemma}\label{lemma-expressions-fj}
\begin{itemize}
\item[\normalfont{(1)}] Suppose that $f_{N}$ is given by \eqref{expression_fN-and-f-N} and suppose further that $\Gamma_0^+ \neq 0$. Then,
$f_0, f_1, \ldots, f_{N}$ solve $(A_j^+)$ and $(B_{j}^+)$ (for all $j = 0, \ldots, N$) if and only if $f_j$ is given by \eqref{expression-fj+} for all $j = 0, \ldots, N$.
\item[\normalfont{(2)}] Suppose that $f_{-N}$ is given by \eqref{expression_fN-and-f-N} and suppose further that $\Gamma_0^- \neq 0$. Then,  $f_{-N}, f_{1-N}, \ldots, f_0$ solve $(A_j^-)$ and $(B_{j}^-)$ (for all $j = 0, \ldots, N$) if and only if $f_{-j}$ is given by \eqref{expression-fj-} for all $j = 0, \ldots, N$.
\end{itemize}
\end{lemma}

Note that \eqref{expression-fj+} and \eqref{expression-fj-} coincide with \eqref{expression_fN-and-f-N} when $j = N$.

\begin{proof}
(1) First, by a similar process to the one used in the proof of Lemma \ref{lemma-recurrence-fj}, we prove that $f_{N-1}$ solves $(B_N^+)$ and $(A_{N-1}^+)$ if and only if it satisfies \eqref{expression-fj+}. By $(B_N^+)$ we have
\begin{equation*}
-\frac{d}{dt}f_{N-1} = (\vartheta_t - a)f_{N} = 2q_N^+\Geg_{a-2}^{\lambda+N}(it),
\end{equation*}
where in the second equality we used \eqref{expression_fN-and-f-N} and \eqref{derivative-Gegenbauer-2}.
We multiply both sides by
\begin{equation*}
-i\Gamma_{N-1}^+ =  -i\gamma(\lambda + N-1, a-1),
\end{equation*}
which is non-zero by the hypothesis $\Gamma_0^+ \neq 0$ and by \eqref{gamma-equivalence}, and integrate the obtained expression by using \eqref{derivative-Gegenbauer-1}, which leads to
\begin{equation*}
-i\gamma(\lambda + N-1, a-1) f_{N-1}(t) = q_N^+\Geg_{a-1}^{\lambda+N-1}(it) + c_{N-1}^+,
\end{equation*}
for some constant $c_{N-1}^+ \in \C$. Let us prove that this $f_{N-1}$ solves $(A_{N-1}^+)$ if and only if $c_{N-1}^+ = 0$. 
From the definition, $(A_{N-1}^+)$ is satisfied if and only if
\begin{equation*}
S_{a+1}^{\lambda+N-2} f_{N-1} - 2\frac{d}{dt}f_N = 0,
\end{equation*}
which multiplied by $-i\gamma(\lambda + N-1, a-1) \neq 0$ amounts to
\begin{equation*}
S_{a+1}^{\lambda+N-2}\left(q_N^+\Geg_{a-1}^{\lambda+N-1} + c_{N-1}^+\right) - 4\left(\lambda + N-1+ \left[\frac{a}{2}\right]\right)q_N^+\Geg_{a-1}^{\lambda+N} = 0,
\end{equation*}
where in the second term we have used \eqref{derivative-Gegenbauer-1} and \eqref{gamma-product-property}. Now, by \eqref{Slmu-identity-2} for $d = 1$, \eqref{derivative-Gegenbauer-2}, and \eqref{Gegen-imaginary}, this can be rewritten as
\begin{multline*}
(a+1)(2\lambda + 2N + a - 3)c_{N-1}^+ + 4q_N^+ \Geg_{a-3}^{\lambda+N}(it) \\+ 4(\lambda + N + a -2)q_N^+\Geg_{a-1}^{\lambda+N-1}
- 4\left(\lambda + N-1+ \left[\frac{a}{2}\right]\right)q_N^+\Geg_{a-1}^{\lambda+N} = 0.
\end{multline*}
The last three terms vanish thanks to the three-term relation \eqref{KKP-1}. Thus, $(A_{N-1}^+)$ is satisfied if and only if
\begin{equation}\label{constant-cN-1}
(2\lambda + 2N + a - 3)c_{N-1}^+ = 0.
\end{equation}
Now, let us take a look at the parity of $a$. If $a$ is even, we deduce that $c_{N-1}^+ = 0$ from the condition $f_{N-1} \in \Pol_{a-1}[t]_\text{even}$. In particular, \eqref{constant-cN-1} holds. On the other hand, if $a$ is odd, we have
\begin{equation*}
0 \neq 2\Gamma_{N-1}^+ = 2\lambda + 2N +a -3.
\end{equation*}
Thus, \eqref{constant-cN-1} holds if and only if $c_{N-1}^+ = 0$. In both cases we have proved that $c_{N-1}^+ = 0$ necessarily, which implies that $f_{N-1}$ is indeed given by \eqref{expression-fj+} and it solves $(B_N^+)$ and $(A_{N-1}^+)$.

Now that the result is proved for $f_N$ and $f_{N-1}$, it suffices to use Lemma \ref{lemma-recurrence-fj}(1) recursively for $j = 1, \ldots, N-1$, which shows the lemma. Note that we can apply Lemma \ref{lemma-recurrence-fj} from $\Gamma_0^+ \neq$ and \eqref{gamma-equivalence}.

(2) The proof of the second statement is analogous to the first one, where we can apply Lemma \ref{lemma-recurrence-fj} due to the condition $\Gamma_0^- \neq 0$.
\end{proof}

\begin{rem} Note that, from the proof of Lemma \ref{lemma-expressions-fj} we deduce that Lemma \ref{lemma-recurrence-fj} is also true for $j= N$, where we define $f_{\pm(N+1)} \equiv 0$.
\end{rem}

With the lemma above, we have completed Phase 2 when $\Gamma_0^\pm \neq 0$. All that remains now is to solve the opposite case (i.e., $\Gamma_0^+\cdot\Gamma_0^- = 0$). \\
One may think that the case $\Gamma_0^+ = 0$ and the case $\Gamma_0^- = 0$ can be done similarly, but the truth is that they are actually quite different. These cases are considered below. Concretely, we consider $\Gamma_0^+ = 0$ in Lemma \ref{lemma-Gamma+zero} and $\Gamma_0^- = 0$ in Lemma \ref{lemma-Gamma-zero}.\\

We start from an easy but important observation. Since we have used the condition $\Gamma_0^+ \neq 0$ in order to have $f_j$, one may expect that the expression \eqref{expression-fj+} does not hold when $\Gamma_0^+ = 0$. However, it turns out that it also holds when $\Gamma_0^+ = 0$ if we redefine the constant $q_N^+$. This will be clear from the next remark and from Lemma \ref{lemma-Gamma+zero} below. 

\begin{rem}\label{rem-normalization-fj+-} 
(1) Note that if $\Gamma_0^+ \neq 0$, then, for some $\widetilde{q_N^+} \in \C$, \eqref{expression-fj+} is equivalent to
\begin{equation}\label{expression-fj+-normalized}
f_j(t) = (-i)^j \frac{\Gamma\left(\lambda+N-1+\left[\frac{a+j-N+1}{2}\right]\right)}{\Gamma\left(\lambda+N-1+\left[\frac{a-N+1}{2}\right]\right)}\widetilde{q_N^+}\Geg_{a+j-N}^{\lambda+N-1}(it).
\end{equation}
In fact, it suffices to define $\widetilde{q_N^+}$ such that
\begin{equation*}
(-i)^N\frac{\Gamma\left(\lambda+N-1+\left[\frac{a+1}{2}\right]\right)}{\Gamma\left(\lambda+N-1+\left[\frac{a-N+1}{2}\right]\right)} \widetilde{q_N^+} = q_N^+. 
\end{equation*}
(2) Analogously, if we define $\widetilde{q_N^-}$ such that
\begin{equation*}
 i^N \widetilde{q_N^-} = q_N^-,
\end{equation*}
then \eqref{expression-fj-} is clearly equivalent to
\begin{equation}\label{expression-fj--normalized}
f_{-j}(t) = i^j \frac{\Gamma\left(\lambda+N-1+\left[\frac{a-j-N+1}{2}\right]\right)}{\Gamma\left(\lambda+N-1+\left[\frac{a-2N+1}{2}\right]\right)}\widetilde{q_N^-}\Geg_{a-j-N}^{\lambda+N-1}(it).
\end{equation}
\end{rem}
Now, we have the following:

\begin{lemma}\label{lemma-Gamma+zero}
Suppose that $\Gamma_0^+ = 0$, or equivalently, that
\begin{equation}\label{Gamma-zero+-condition}
\gamma(\lambda+N-1, a-s) = 0, \enspace \text{ for some } s = 1, 2, \ldots, N.
\end{equation}
Then, $f_0, f_1, \ldots, f_N$ solve $(A_j^+)$ and $(B_j^+)$ for all $j = 0, \ldots, N$ if and only if $f_{j}$ is given by \eqref{expression-fj+-normalized} for all $j = 0, 1, \ldots, N$ for some $\widetilde{q_N^+} \in \C$. In particular $f_N = f_{N-1} = \cdots = f_{N-s+1} \equiv 0$.
\end{lemma}
\begin{proof}
By definition (see \eqref{gamma-def}), $\gamma(\lambda+N-1, a-s) = 0$ amounts to
\begin{equation*}
a-s \in 2\N,\text{ and } \lambda+N-1+\frac{a-s}{2}=0.
\end{equation*}
Then, in particular, we have $\gamma(\lambda+N-1, a-r) \neq 0$ for any $r = 1, 2, \ldots, N$, $r\neq s$. By \eqref{gamma-recurrence-relation}, this implies that $\Gamma_{j-1}^+ \neq 0$ for any $j = N-s+2, N-s+3, \ldots, N$. Therefore, by \eqref{expression_fN-and-f-N} and by using Lemma \ref{lemma-recurrence-fj}(1) recursively, we obtain that $f_{j}$ is given by \eqref{expression-fj+} for any $j = N-s+1, N-s+2, \ldots, N$ for some constant $q_N^+ \in \C$. Now, for $j = N-s$, we have that $(B_{N-s+1}^+)$ multiplied by $(-i)^{s-1}\Gamma_{N-s+1}^+ \neq 0$ amounts to
\begin{equation*}
-(-i)^{s-1}\Gamma_{N-s+1}^+\frac{d}{dt}f_{N-s} = 2q_N^+\Geg_{a-s-1}^{\lambda+N}(it).
\end{equation*}
Hence, by \eqref{Gamma-zero+-condition}, $(A_{N-s}^+)$ multiplied by $(-i)^{s-1}\Gamma_{N-s+1}^+$ amounts to
\begin{equation*}
\begin{aligned}
0 =  \enspace &-(-i)^{s-1}\Gamma_{N-s+1}^+\left(S_{a+s}^{\lambda+N-s-1}f_{N-s} - 2s\frac{d}{dt}f_{N-s+1}\right)\\
= \enspace &-\left((1+t^2)\frac{d}{dt} + (2\lambda+2N-2s-1)t\right)2q_N^+\Geg_{a-s-1}^{\lambda+N}(it) + 4isq_N^+\Geg_{a-s}^{\lambda+N}(it)\\
= \enspace  &-4iq_N^+(1+t^2)\Geg_{a-s-2}^{\lambda+N+1}(it) + (2\lambda+2N-2s-1)2q_N^+ t\Geg_{a-s-1}^{\lambda+N}(it)\\
& + 4isq_N^+\Geg_{a-s}^{\lambda+N}(it).
\end{aligned}
\end{equation*}
If we evaluate the above equation in $t = 0$ and divide it by $4i$, we have
\begin{equation*}
\frac{(-1)^\frac{a-s}{2}}{\left(\frac{a-s}{2}\right)!}
\frac{a+s}{2}q_N^+ = 0,
\end{equation*}
which clearly implies $q_N^+ = 0$. Then, $(A_{N-s}^+)$ is satisfied if and only if $q_N^+=0$. This implies in particular that $f_j \equiv 0$ for $j = N-s+1, \ldots, N$. Therefore, $f_{N-s+1}, \ldots, f_N$ satisfy \eqref{expression-fj+-normalized} as
\begin{equation*}
\frac{\Gamma\left(\lambda+N-1+\left[\frac{a+j-N+1}{2}\right]\right)}{\Gamma\left(\lambda+N-1+\left[\frac{a-N+1}{2}\right]\right)} = \gamma(\lambda+N-1, a+j-N-1) \dots \gamma(\lambda+N-1, a-N) = 0,
\end{equation*}
for any $j = N-s+1, \ldots, N$.

Now, if we take a look again at $(A_{N-s}^+)$, by Theorem \ref{thm-Gegenbauer-solutions} and Lemma \ref{lemma-relationS-G} we obtain
\begin{equation*}
S_{a+s}^{\lambda+N-1} f_{N-s} = 0 \Rightarrow f_{N-s}(t) = q_{N-s}^+\Geg_{a+s}^{\lambda+N-s-1}(it),
\end{equation*}
for some $q_{N-s}^+ \in \C$. Note that since $\lambda + N-1 + \frac{a-s}{2} = 0$, the polynomials $\Geg_{a+s}^{\lambda+N-s-1}$ and $\Geg_{a-s}^{\lambda+N-1}$ are constant. Therefore, we can rewrite $f_{N-s}$ as
\begin{equation*}
f_{N-s}(t) = (-i)^{N-s}\frac{\Gamma \left(\lambda+N-1 + \left[\frac{a-s+1}{2}\right]\right)}{\Gamma\left(\lambda+N-1 + \left[\frac{a-N+1}{2}\right]\right)}\widetilde{q_N^+}\Geg_{a-s}^{\lambda+N-1}(it),
\end{equation*}
where $\widetilde{q_N^+} \in \C$ is an appropriate multiple of $q_{N-s}^+$.

Now, we can repeat the proof of Lemma \ref{lemma-recurrence-fj}(1) recursively and obtain $f_j$ for $j = 0, 1, \ldots, N-s-1$. The same argument works since $\gamma(\lambda + N-1, a-r) \neq 0$ for all $r = 1, \ldots, s-1$. The expression obtained can be easily proved to be \eqref{expression-fj+-normalized}, showing that, indeed,  $f_0, f_1, \ldots, f_N$ satisfy $(A_j^+), (B_j^+)$ for all $j = 0, \ldots, N$, completing the proof.
\end{proof}

\begin{rem}\label{rem-Gamma+zero}
An important consequence from the lemma above is that $f_0, \ldots, f_N$ solve $(A_j^+), (B_j^+)$ (for all $j = 0, \ldots, N)$ if and only if \eqref{expression-fj+-normalized} holds for any $j = 0, 1, \ldots, N$ whether $\Gamma_0^+$ vanishes or not.
\end{rem}

Next, we consider the case $\Gamma_0^- = 0$. In this case, we can still apply Lemma \ref{lemma-recurrence-fj}(2) recursively to obtain all $f_{-j}$. However, we cannot assure that the constants $c_j^-$ appearing in this process must vanish in order to solve $(A_j^-)$. That is why in this case we have to carry forward all these constants until the step $j = 0$. It will be after that, in Phase 3, when we will be able to assure that they all must vanish.\\
In order to state the main lemma of this part, let us fix some notation.

Given $j = 0, 1, \ldots, N$, $s = 1, \ldots, N$ and $n = 0, 1, \ldots, N+s$, and given constants $c_{N-s}^-$, $c_{N-s-2}^-$, $\ldots$, $c_{N-s-2\left[\frac{n}{2}\right]}^- \in \C$, we define a polynomial $P_{-j}(t) \in \C[t]$ as follows:
\begin{itemize}[topsep = 5pt]
\item For $j = N-s+1, N-s+2, \ldots, N$ we set $P_{-j} \equiv 0$.
\item For $j = N-s-n = 0, 1, \ldots, N-s$ we set
\begin{equation*}
P_{n+s-N}(t) := \sum_{k=0}^{\left[\frac{n}{2}\right]}c_{N-s-2k}^- \left(\sum_{d=k}^{\left[\frac{n}{2}\right]}M(n, k, d) (it)^{n-2d}\right).
\end{equation*}
Here, the coefficients $M(n,k,d)$ are defined as follows for $0 \leq k \leq \left[\frac{n-1}{2}\right]$ and $k \leq d \leq \left[\frac{n-1}{2}\right]$
\begin{equation*}
M(n,k,d) := \frac{2^{n-2d}(n-2k-1)!(s+n-d-1)_{n-d-k}\left(\frac{2N-s+a-2k}{2}\right)_{n-d-k}}{(d-k)!(n-2d-1)!(n-d-k)!(2N-s-2k)_{n-2k}},
\end{equation*}
where $(a)_\ell$ denotes Pochhammer's falling factorial defined as
\begin{equation*}
(a)_0 = 1 \enspace \text{ and } (a)_\ell = a (a-1) \cdots (a-\ell +1). 
\end{equation*}
When $n$ is even, we define $M(n,k,\frac{n}{2})$ as
\begin{align*}
M(n, k, \frac{n}{2}) = 
\begin{cases}
0, & \text{ for } 0 \leq k \leq \frac{n-2}{2},\\
1, & \text{ for } k = \frac{n}{2}.
\end{cases}
\end{align*}

The definition of the polynomials $P_{-j}$ may seem arbitrary. However, it is the key to solve the system when $\Gamma_0^- = 0$, as the following lemma shows.
\end{itemize}

\begin{lemma}\label{lemma-Gamma-zero}
Suppose that $f_{-N}$ is given by \eqref{expression_fN-and-f-N} and suppose further that $\Gamma_0^- = 0$, or equivalently, that
\begin{equation}\label{Gamma-zero--condition}
\gamma(\lambda+N-1, a-2N-s) = 0, \enspace \text{ for some } s = 1, 2, \ldots, N.
\end{equation}
Then, $f_{-N}, f_{1-N}, \ldots, f_0$ solve $(A_j^-)$ and $(B_j^-)$ for all $j = 0, \ldots, N$ if and only if $f_{-j}$ is given by \eqref{expression-fj--Pj-} below for all $j= 0, 1, \ldots, N$.
\begin{equation}\label{expression-fj--Pj-}
i^{N-j} f_{-j}(t) = q_N^-  \frac{\Gamma\left(\lambda+N-1+\left[\frac{a-j-N+1}{2}\right]\right)}{\Gamma\left(\lambda+N-1+\left[\frac{a-2N+1}{2}\right]\right)}\Geg_{a-j-N}^{\lambda +N-1}(it) + P_{-j}(t),
\end{equation}
where the constants $c_{N-s}^-, c_{N-s-2}^-, \ldots, c_{N-s-2\left[\frac{n}{2}\right]}^- \in \C$ appearing in $P_{n+s-N}$ are any constants satisfying the following relation for all $\ell = 0, 1, \ldots, \left[\frac{n-2}{2}\right]$.
\begin{equation}\label{constants-cj-relation}
\begin{aligned}
&\sum_{k = 0}^{\ell}c_{N-s-2k}^-\Big(2M(n,k,\ell+1)(\ell+1)(2N-s-a-2(\ell+1)) \\
&+ M(n,k,\ell)(n-2\ell)(n-2\ell-1) -2M(n-1, k, \ell)(n+s)(n-2\ell -1)\Big)\\
&+2c_{N-s-(\ell+1)}^-M(n, \ell+1, \ell+1)(\ell+1)(2N-s-a-2(\ell+1)) = 0.
\end{aligned}
\end{equation}
\end{lemma}
\begin{rem} The constants $c_{N-s}^-, c_{N-s-2}^-, \ldots, c_{N-s-2\left[\frac{n}{2}\right]}^-$ can take any value as soon as they satisfy \eqref{constants-cj-relation}. For example, if we take $c_{N-s}^- = c_{N-s-2}^- = \cdots = c_{N-s-2\left[\frac{n}{2}\right]}^- = 0$, we have that \eqref{constants-cj-relation} is trivially satisfied. As we pointed out before, at the end we will have that these constants must vanish, but this will be showed in Phase 3.
\end{rem}
\begin{proof}
By definition (see \eqref{gamma-def}), $\gamma(\lambda+N-1, a-2N-s) = 0$ amounts to
\begin{equation*}
a-s \in 2\N,\text{ and } \lambda+\frac{a-s-2}{2}=0.
\end{equation*}
Then, in particular, we have $\gamma(\lambda, a-r-2) \neq 0$ for any $r = 1, 2, \ldots, N$, $r\neq s$, which implies $\gamma(\lambda,a+j-N-3) \neq 0$ for $j = N-s+2, \dots, N-1$. By \eqref{expression_fN-and-f-N} and by using Lemma \ref{lemma-recurrence-fj}(2) recursively, we have that $f_{s-1-N}, \ldots, f_{-N}$ solve $(A_{N}^-), \ldots, (A_{N-s+1}^-)$ and $(B_N^-), \ldots, (B_{N-s+1}^-)$ if and only if $f_{-j}$ is given by \eqref{expression-fj--Pj-} for any $j = N-s+1, \ldots, N$ (since $P_{-j} \equiv 0$ in this case).

Hence, it suffices to prove the statement for $j = N-s-n$ when $n= 0, 1, \ldots, N-s$. We prove this by induction on $n$. Concretely, we prove that a solution of the system must satisfy \eqref{expression-fj--Pj-}, and later, we prove that if $f_{-j}$ for $j=N-s-n$ ($n = 0, 1, \ldots, N-s$) is given by \eqref{expression-fj--Pj-}, then, $f_{-j}$ provide a solution of the system if and only if \eqref{constants-cj-relation} is satisfied. 

Suppose $f_{s-N}, \ldots, f_0$ form part of a solution of the system (i.e., that they solve the corresponding $(A_\bullet^-)$, $(B_\bullet^-)$), and suppose $n=0$. Then, as proved in Lemma \ref{lemma-recurrence-fj}(2), by using, $f_{s-N-1}$ and $f_{s-N-2}$ we have that $(B_{N-s+1}^-)$ amounts to
\begin{equation*}
i^{s-1}\frac{d}{dt}f_{s-N} = q_N^-\frac{\Gamma\left(\lambda+N-1 + \left[\frac{a+s-2N+1}{2}\right]\right)}{\Gamma\left(\lambda+N-1 + \left[\frac{a-2N+1}{2}\right]\right)}\gamma(\lambda+N-1, a+s-2N)\Geg_{a+s-2N-1}^{\lambda+N}(it).
\end{equation*} 
Then, as before, we can integrate the expression above by \eqref{derivative-Gegenbauer-1}, which amounts to
\begin{equation*}
i^{s}f_{s-N}(t) = q_N^-\frac{\Gamma\left(\lambda+N-1 + \left[\frac{a+s-2N+1}{2}\right]\right)}{\Gamma\left(\lambda+N-1 + \left[\frac{a-2N+1}{2}\right]\right)}\Geg_{a+s-2N}^{\lambda+N-1}(it) + c_{N-s}^-,
\end{equation*}
for some constant $c_{N-s}^- \in \C$. Thus, for $n=0$, \eqref{expression-fj--Pj-} holds. Next, suppose that $n=1$. By the same argument, by using the expressions of $f_{s-N-1}$ and $f_{s-N}$, we have that $(B_{N-s}^-)$ amounts to
\begin{equation*}
\begin{aligned}
i^{s}\frac{d}{dt}f_{s-N+1} = \; &q_N^-\frac{\Gamma\left(\lambda+N-1 + \left[\frac{a+s-2N+2}{2}\right]\right)}{\Gamma\left(\lambda+N-1 + \left[\frac{a-2N+1}{2}\right]\right)}\gamma(\lambda+N-1, a+s-2N+1)\Geg_{a+s-2N}^{\lambda+N}(it) \\[4pt]
&+ \frac{(N-j)(N+j+a)}{N+j}c_j^-,
\end{aligned}
\end{equation*}
where the second term comes from a direct computation of $2(N(\lambda+a-1) +j(\lambda-1))c_j^-$ by using \eqref{Gamma-zero--condition}. Again, we integrate the expression above by using \eqref{derivative-Gegenbauer-1} and obtain
\begin{equation*}
\begin{aligned}
i^{s+1}f_{s-N+1}(t) = \; &q_N^-\frac{\Gamma\left(\lambda+N-1 + \left[\frac{a+s-2N+2}{2}\right]\right)}{\Gamma\left(\lambda+N-1 + \left[\frac{a-2N+1}{2}\right]\right)}\Geg_{a+s-2N+1}^{\lambda+N-1}(it) \\[4pt]
&+ \frac{(N-j)(N+j+a)}{N+j}c_j^-it,
\end{aligned}
\end{equation*}
Note that this expression coincides with \eqref{expression-fj--Pj-} since $M(1,0,0) = \frac{(N-j)(N+j+a)}{N+j}$. Thus \eqref{expression-fj--Pj-} also holds for $n=1$.

Now, suppose that \eqref{expression-fj--Pj-} holds for $n$ and $n+1$, (i.e., for $f_{s+n-N}$ and $f_{s+n-N-1}$), for some $n = 0, \ldots, N-s-1$. Let us show that it also holds for $n-1$ (i.e., for $f_{s+n-N+1}$). For $j = N-s-n, \ldots, N$ we write $f_{-j}$ as
\begin{equation*}
i^{N-j}f_{-j}(t) = G_{-j}(t) + P_{-j}(t),
\end{equation*}
where $G_{-j}(t)$ represents the \lq\lq Gegenbauer part\rq\rq, i.e., the first term in \eqref{expression-fj--Pj-} corresponding to the Gegenbauer polynomial. By using this notation, we have that $(B_{N-s-n}^-)$ multiplied by $(-i)^{s+n}$ can be written as:
\begin{equation}\label{proof-equation-B-}
\begin{aligned}
i^{s+n}(2N-s-n)\frac{d}{dt}f_{s+n-N+1} = &\;
2(N(\lambda+ a-1) + (N-s-n)(\lambda-1+\vartheta_t))G_{n+s-N} \\
& - i(s+n)\frac{d}{dt}G_{n+s-N-1}\\
&+ 2(N(\lambda+ a-1)+ (N-s-n)(\lambda-1+\vartheta_t))P_{n+s-N} \\
&- i(s+n)\frac{d}{dt}P_{n+s-N-1}.
\end{aligned}
\end{equation}
By following the same procedure used in Lemma \ref{lemma-recurrence-fj}(2), the first two terms can be simplified to
\begin{multline}\label{proof-G-part}
2(2N-s-n)q_N^-\frac{\Gamma\left(\lambda+N-1 + \left[\frac{a+s+n-2N+2}{2}\right]\right)}{\Gamma\left(\lambda+N-1 + \left[\frac{a-2N+1}{2}\right]\right)}\\
\times\gamma(\lambda+N-1, a+s+n-2N+1)\Geg_{a+s+n-2N}^{\lambda+N}(it).
\end{multline}
On the other hand, by $(A_{N-s-n}^-)$, we have
\begin{equation*}
S_{a+s+n-2N}^{\lambda+N-s-n-1}(G_{n+s-N} + P_{n+s-N}) = -2i(s+n)\frac{d}{dt}(G_{n+s-N-1} + P_{n+s-N-1}). 
\end{equation*}
The $G_\bullet$ parts can be proved to cancel out by using the same arguments in the proof of Lemma \ref{lemma-recurrence-fj} (2). Thus, we have
\begin{equation}\label{proof-eq-A-P-part}
S_{a+s+n-2N}^{\lambda+N-s-n-1}P_{n+s-N} = -2i(s+n)\frac{d}{dt}P_{n+s-N-1}. 
\end{equation}
By using this identity we obtain that the last two terms of \eqref{proof-equation-B-} amounts to
\begin{equation*}
\left(2(N(\lambda+ a-1)+ (N-s-n)(\lambda-1+\vartheta_t))+ \frac{1}{2}S_{a+s+n-2N}^{\lambda+N-s-n-1}\right)P_{n+s-N}. 
\end{equation*}
By using \eqref{Gamma-zero--condition}, \eqref{Slmu-identity-concrete} and the expression of $P_{n+s-N}$, the identity above amounts to
\begin{equation*}
\begin{aligned}
\sum_{k=0}^{\left[\frac{n}{2}\right]}c_{N-s-2k}^-\Big(&\sum_{d=k}^{{\left[\frac{n}{2}\right]}}M(n,k,d)(s+n-d)(2N-s+a-2n+2d)(it)^{n-2d}\\
&+\sum_{d=k}^{{\left[\frac{n}{2}\right]}}\frac{1}{2}M(n,k,d)(n-2d)(n-2d-1)(it)^{n-2d-2}\Big),
\end{aligned}
\end{equation*}
which after elementary computations, can be rewritten as
\begin{equation}\label{proof-P-part}
\begin{aligned}
\sum_{k=0}^{\left[\frac{n+1}{2}\right]}c_{N-s-2k}^-\Big(\sum_{d=k}^{{\left[\frac{n+1}{2}\right]}}M(n+1,k,d)(2N-s-n)(n+1-2d)(it)^{n-2d}\Big).
\end{aligned}
\end{equation}
Now, by using \eqref{proof-G-part}, \eqref{proof-P-part}  and \eqref{derivative-Gegenbauer-1}, we integrate \eqref{proof-equation-B-} and obtain
\begin{equation*}
\begin{aligned}
i^{s+n+1}f_{s+n-N+1} = &\; 
q_N^-\frac{\Gamma\left(\lambda+N-1 + \left[\frac{a+s+n-2N+2}{2}\right]\right)}{\Gamma\left(\lambda+N-1 + \left[\frac{a-2N+1}{2}\right]\right)}\Geg_{a+s+n+1-2N}^{\lambda+N-1}(it)\\
& +\sum_{k=0}^{\left[\frac{n+1}{2}\right]}c_{N-s-2k}^-\Big(\sum_{d=k}^{{\left[\frac{n+1}{2}\right]}}M(n+1,k,d)(it)^{n-2d+1}\Big).
\end{aligned}
\end{equation*}
Hence, $f_{s+n-N+1}$ is also given by \eqref{expression-fj--Pj-}.\\

Now, suppose that $f_{-j}$ is given by \eqref{expression-fj--Pj-} for all $j=0, \ldots, N$. Let us show that $f_{-j}$ provide a solution of the system if and only if the relations \eqref{constants-cj-relation} hold. For $j = N-s+1, \ldots, N$ there is nothing to prove, since $P_{-j} \equiv 0$. Suppose then $j = 0, \ldots, N-s$. Since we obtained $f_{-j}$ by solving $(B_{j+1}^-)$ recursively, it suffices to show that $f_{-j}$ satisfies $(A_{j}^-)$. Fix $j = N-s-n$ with $n = 0, 1, \ldots, N$. As pointed out above, $(A_{N-s-n}^-)$ is satisfied if and only if \eqref{proof-eq-A-P-part} holds. By using \eqref{Slmu-identity-concrete}, and the definition of $P_{N-s-n}$, the left-hand side amounts to
\begin{equation*}
\begin{aligned}
\sum_{k=0}^{\left[\frac{n}{2}\right]}c_{N-s-2k}^-\Bigg(&
2kM(n,k,k)(2N-s-a-2k)(it)^{n-2k} + \\
&\sum_{d=k+1}^{\left[\frac{n}{2}\right]}\Big(2dM(n,k,d)(2N-s-a-2d) \\
&+ M(n,k,d-1)(n-2d+2)(n-2d+1)\Big)(it)^{n-2d}\Bigg),
\end{aligned}
\end{equation*}
while the right-hand side is
\begin{equation*}
\sum_{k=0}^{\left[\frac{n-1}{2}\right]}c_{N-s-2k}^-\Big(\sum_{d=k}^{\left[\frac{n-1}{2}\right]}
2M(n-1,k,d)(n+s)(n-2d-1)(it)^{n-2d-2}\Big).
\end{equation*}
Both sides are polynomials of degree $n-2$, so they clearly coincide if and only if each term $t^{n-2\ell-2}$ (for $\ell = 0, 1, \ldots, \left[\frac{n-2}{2}\right]$) coincide. By a straightforward computation, this is easily shown to by equivalent to \eqref{constants-cj-relation}. Hence, the lemma is proved. 
\end{proof}

Now we consider the case $N < a < 2N$. This case can be done by using the same machinery we developed in the case $a \geq 2N$. First, note that Lemma \ref{expression-fj-}(1) and Remark \ref{rem-normalization-fj+-}(1) still holds in this case, so there are no changes in the $+$ side. Thus, it suffices to think about the $-$ side. In this case, instead of distinguishing between $\Gamma_0^- \neq 0$ and $\Gamma_0^- = 0$, we consider the following two cases:
\begin{itemize}
\item[(a)] $\gamma(\lambda, a-s-2) \neq 0$ for all $s=2N-a, \ldots, N$.
\item[(b)] $\gamma(\lambda, a-s-2) = 0$ for some $s=2N-a, \ldots, N$.
\end{itemize}
Since $a-N-j<0$ for $j = a-N+1, a-N+2, \ldots, N$, we have that $f_{-j} \equiv 0$ for $j = a-N+1, a-N+2, \ldots, N$ by the condition $f_{-j}\in\Pol_{a-N-j}[t]_\text{even}$. Hence, by \eqref{expression_fN-and-f-N} and by using Lemma \ref{lemma-recurrence-fj}(2) recursively for $j =1, 2, \ldots, a-N+1$ we can obtain $f_{N-a}, \ldots, f_0$ as soon as $\gamma(\lambda, a+j-N-3) \neq 0$ for all $j = 1, 2, \ldots, a-N+1$, which is precisely the condition (a). Hence, in this case we have that $f_{-j}$ is given by \eqref{expression-fj--normalized} for all $j = 0, 1, \ldots, N$.

On the other hand, (b) is just the condition \eqref{Gamma-zero--condition} for $s = 2N-a, 2N-a+1, \ldots, N$. Thus, in this case, by Lemma \ref{lemma-Gamma-zero} we obtain that $f_{-j}$ is given by \eqref{expression-fj--Pj-} for all $j = 0, 1, \ldots, N$.

\subsection{Phase 3: Proving $f_{-0} = f_{+0}$}\label{section-phase3}
Now that we have completed Phase 2, i.e., that we have obtained $f_{-N}, \ldots, f_0$ in the $-$ side, and $f_0, \ldots, f_N$ in the $+$ side, the only task that remains is to check that the expressions obtained for $f_0$ in both sides (that we can denote by $f_{-0}$ and $f_{+0}$ respectively) coincide.

As in Section \ref{section-phase1}, we assume $a\geq 2N$ and consider the case $N < a < 2N$ at the end in Remark \ref{rem-Phase3-N<=a<2N}.

Suppose $\Gamma^-_0 \neq 0$. Then, by Lemma \ref{lemma-expressions-fj} and Remarks \ref{rem-normalization-fj+-} and \ref{rem-Gamma+zero}, a solution of the system $f_{\pm j}$ is given by \eqref{expression-fj+-normalized} and \eqref{expression-fj--normalized} respectively. Thus, $f_{-0} = f_{+0}$ amounts to
\begin{equation}\label{equality-fpm0}
\frac{\Gamma\left(\lambda+N-1+\left[\frac{a-N+1}{2}\right]\right)}{\Gamma\left(\lambda+N-1+\left[\frac{a-2N+1}{2}\right]\right)}\widetilde{q_N^-}\Geg_{a-N}^{\lambda+N-1}(it) = \widetilde{q_N^+}\Geg_{a-N}^{\lambda+N-1}(it).
\end{equation}
Hence, we obtain
\begin{lemma}\label{lemma-phase3-generic}
Suppose $\Gamma_0^- \neq 0$ and suppose that $f_{\pm j}$ are given by \eqref{expression-fj+-normalized} and \eqref{expression-fj--normalized} respectively. Then, \eqref{equality-fpm0} holds if and only if the tuple $(\widetilde{q_N^-}, \widetilde{q_N^+})$ satisfies the following relation.
\begin{equation}\label{relation-q_N-pm}
\frac{\Gamma\left(\lambda+N-1+\left[\frac{a-N+1}{2}\right]\right)}{\Gamma\left(\lambda+N-1+\left[\frac{a-2N+1}{2}\right]\right)}\widetilde{q_N^-} = \widetilde{q_N^+}.
\end{equation}
\end{lemma}

Now, suppose that $\Gamma_0^- = 0$, then by Lemma \ref{lemma-expressions-fj}(1), Remark \ref{rem-normalization-fj+-}(1) and Lemma \ref{lemma-Gamma-zero}, a solution of the system $f_{\pm j}$ is given by \eqref{expression-fj+-normalized} and \eqref{expression-fj--Pj-} respectively. Thus, $f_{-0} = f_{+0}$ amounts to
\begin{equation}\label{equality-fpm0-Gamma-zero}
\widetilde{q_N^-} \frac{\Gamma\left(\lambda+N-1+\left[\frac{a-N+1}{2}\right]\right)}{\Gamma\left(\lambda+N-1+\left[\frac{a-2N+1}{2}\right]\right)}\Geg_{a-N}^{\lambda +N-1}(it) + (-i)^{N} P_{0}(t) = \widetilde{q_N^+}\Geg_{a-N}^{\lambda+N-1}(it).
\end{equation}
Now, we have

\begin{lemma}\label{lemma-phase3-Gamma-zero}
Suppose $\Gamma_0^- = 0$, or equivalently, that \eqref{Gamma-zero--condition} holds, and suppose further that $f_{\pm j}$ is given by \eqref{expression-fj+-normalized} and \eqref{expression-fj--Pj-} respectively. Then, \eqref{equality-fpm0-Gamma-zero} holds if and only if the two following conditions are satisfied:
\begin{itemize}
\item  The constants $c_{N-s}^-, c_{N-s-2}^-, \ldots, c_{N-s-2\left[\frac{N-s}{2}\right]}^-$ appearing in $P_0$ all vanish. In particular, $P_{-j} = 0$ for all $j=0, \ldots, N$.
\item The tuple $(\widetilde{q_N^-}, \widetilde{q_N^+})$ satisfies \eqref{relation-q_N-pm}.
\end{itemize}
\end{lemma}
\begin{proof}
If the conditions of the lemma are satisfied, it is clear that \eqref{equality-fpm0-Gamma-zero} is satisfied. Thus, let us show the opposite direction.

If $a \geq 2N$, then $\deg(\Geg_{a-N}^{\lambda+N-1}) = a-N > \deg(P_0) = N-s$. Thus, the top terms of both sides of \eqref{equality-fpm0-Gamma-zero} coincide if and only if  $(\widetilde{q_N^-}, \widetilde{q_N^+})$ satisfies \eqref{relation-q_N-pm}. Then, from \eqref{equality-fpm0-Gamma-zero} we have that
\begin{equation*}
P_0(t) = \sum_{k=0}^{\left[\frac{N-s}{2}\right]}C_{N-s-2k}^-\Big(\sum_{d = 0}^{\left[\frac{N-s}{2}\right]}M(N-s, k, d) (it)^{N-s-2d}\Big) = 0.
\end{equation*}
In particular, all terms of $P_0$ must vanish.
By the expression above it is clear that the term $t^{N-s}$ amounts to
\begin{equation*}
C_{N-s}^-M(N-s, 0, 0)(it)^{N-s}.
\end{equation*}
Thus, it vanishes if and only if $c_{N-s}^- = 0$ as $M(N-s,0,0) \neq 0$.
Similarly, the term $t^{N-s-2}$ amounts to
\begin{equation*}
\big(C_{N-s}^-M(N-s,0,1) + C_{N-s-2}^-M(N-s,1,1)\big)(it)^{N-s-2}.
\end{equation*}
Therefore, since $c_{N-s}^- = 0$, it vanishes if and only if $c_{N-s-2}^- = 0$ as $M(N-s,1,1) \neq 0$. By repeating this process, we will have that for some $\ell = 0, 1, \ldots, \left[\frac{\ell}{2}\right]$, the terms $t^{N-s}, t^{N-s-2} \ldots, t^{N-s-2\ell}$ vanish if and only if $c_{N-s}^- =c_{N-s-2}^- = \cdots = c_{N-s-2\ell}^- = 0$. Then, the term $t^{N-s-2\ell-2}$ amounts to
\begin{equation*}
\sum_{k=0}^{\ell+1}c_{N-s-k}^-M(N-s,k,\ell+1)(it)^{N-s-2\ell-2} = c_{N-s-2\ell -2}^-M(N-s,\ell+1, \ell+1)(it)^{N-s-2\ell-2},
\end{equation*}
which vanishes if and only if $c_{N-s-2\ell-2}^- = 0$ as $M(N-s, d, d) \neq 0$ for any $d$. Hence, by finite induction on $\ell$ we arrive to the conclusion that $P_0 \equiv 0$ implies $c_{N-s}^- = c_{N-s-2}^- =  \cdots =  c_{N-s-2\left[\frac{N-s}{2}\right]}^- = 0$. 
\end{proof}

\begin{rem}\label{rem-Phase3-N<=a<2N}
We note that if $N < a < 2N$ the results of this section remain true. As we pointed out at the end of Section \ref{section-phase1}, instead of distinguish between $\Gamma_0^- \neq 0$ and $\Gamma_0^- = 0$, we should consider the cases:
\begin{itemize}
\item[(a)] $\gamma(\lambda, a-s-2) \neq 0$ for all $s=2N-a, \ldots, N$.
\item[(b)] $\gamma(\lambda, a-s-2) = 0$ for some $s=2N-a, \ldots, N$.
\end{itemize}
The first case is covered by Lemma \ref{lemma-phase3-generic}, so we obtain the same result. On the other hand, case (b) can be proved by following the same argument in the proof of Lemma \eqref{lemma-phase3-generic}. This can be done since as $s = 2N-a, \ldots, N$, we have $N-s < a-N$. So in this case we obtain the same result as Lemma \ref{lemma-phase3-Gamma-zero}.
\end{rem}

\subsection{Proof of Theorem \ref{Thm-solvingequations}}
\label{section-proof_of_solving_equations-subsection}

Now, we are ready to show Theorem \ref{Thm-solvingequations}.
\begin{proof}[Proof of Theorem \ref{Thm-solvingequations}]
Take polynomials
\begin{equation}\label{polynomial-even-order}
f_{\pm j} \in \Pol_{a-N\pm j}[t]_\text{{\normalfont{even}}}, \enspace \text{ for } j=0, 1, \ldots, N.
\end{equation} 
We show that if $(f_{-N}, \ldots, f_0, \ldots, f_N)$ is a solution of the system $\Xi(\lambda, a, N, N)$, then
\begin{equation}\label{proof-tuple-f-g}
(f_{-N}, \ldots, f_0, \ldots, f_N) = \alpha (g_{2N}, \ldots, g_N, \ldots, g_{0}),
\end{equation}
for some $\alpha \in \C$, where $(g_k)_{k=0}^{2N}$ is the tuple given in (\ref{solution-all}).

We divide the proof into the following two cases:
\begin{enumerate}[leftmargin=1cm, topsep=0pt]
\item[(1)] $0 \leq a \leq N$.
\item[(2)] $a > N$.
\end{enumerate}

We consider (2) first. By the three-phase strategy explained in Sections \ref{section-phase1} and \ref{section-phase3}, we have that $(f_{-N}, \ldots, f_N)$ is a solution of the system if and only if $f_{\pm j}$ is given by \eqref{expression-fj+-normalized} and \eqref{expression-fj--normalized} respectively for all $j = 0, \ldots, N$, where the constants $\widetilde{q_N^\pm}$ satisfy \eqref{relation-q_N-pm}. In other words, $f_{\pm j}$ provides a solution of the system if and only if
\begin{equation*}
f_{\pm j}(t) = i^{\mp j} \frac{\Gamma\left(\lambda+N-1+\left[\frac{a\pm j-N+1}{2}\right]\right)}{\Gamma\left(\lambda+N-1+\left[\frac{a-2N+1}{2}\right]\right)}\widetilde{q_N^-}\Geg_{a\pm j-N}^{\lambda+N-1}(it).
\end{equation*}
The tuple $(f_{-N}, \ldots, f_N)$ clearly satisfies
\eqref{proof-tuple-f-g} for $\alpha = (-1)^N\widetilde{q_N^-}$.

(1) Suppose now that $0 \leq a \leq N$. Since $a-j-N < 0$ for all $j = 1, 2, \ldots, N$, we have that $f_{-j} \equiv 0$ for all $j = 1, 2, \ldots, N$ from \eqref{polynomial-even-order}. Hence, all the functions in the $-$ side vanish, so it suffices to determine $f_j$ (for all $j=0, 1, \ldots, N$). In particular, there is no Phase 3 in this case. Moreover, also by \eqref{polynomial-even-order}, we deduce that $f_j \equiv 0$ for $j = 0, \ldots, N-a-1$, so the non-vanishing functions are $f_{N-a}, f_{N-a+1}, \ldots, f_N$.

The arguments in the proofs of Lemmas \ref{lemma-expressions-fj} and \ref{lemma-Gamma+zero} also work in this case. However, instead of renormalizing $q_N^+$ as in Remark \ref{rem-normalization-fj+-}, we do it as follows:
\begin{equation}\label{renormalization-0<=a<=N}
 (-i)^N\frac{\Gamma\left(\lambda+N-1+\left[\frac{a+1}{2}\right]\right)}{\Gamma\left(\lambda+N-1\right)} \widetilde{q_N^+} = q_N^+. 
\end{equation}
We do this change in the normalization in order to avoid that the constructed solution $f_j$ vanishes if $\lambda$ takes some particular values as we explain below.\\
When $a > N$, the normalization is done such that the gamma factor in the right-hand side of \eqref{expression-fj+-normalized} is 1 when $j = 0$, $\gamma(\lambda+N-1, a-N)$ when $j = 1$, $\gamma(\lambda +N-1, a-N)\gamma(\lambda+N-1, a-N+1)$ when $j=2$, and so on. If we use the same normalization when $0 \leq a \leq N$, the gamma factor of the first (starting from $f_0$) non-vanishing function $f_{N-a}$ would be $\gamma(\lambda+N-1, a-N)\gamma(\lambda+N-1, a-N+1) \cdots \gamma(\lambda+N-1, -1)$, which can vanish for some values of $\lambda$. We renormalize the constant $q_N^+$ so that this gamma factor is 1. This normalization is precisely \eqref{renormalization-0<=a<=N}, which leads to
\begin{equation*}
f_{j}(t) = (-i)^{j} \frac{\Gamma\left(\lambda+N-1+\left[\frac{a+ j-N+1}{2}\right]\right)}{\Gamma\left(\lambda+N-1\right)}\widetilde{q_N^+}\Geg_{a+j-N}^{\lambda+N-1}(it),
\end{equation*}
for all $j = 0, 1, \ldots, N$. Hence, we clearly have \eqref{proof-tuple-f-g}, for $\alpha = (-1)^N\widetilde{q_N^+}$.

\end{proof}
\section{Duality Between $m$ and $-m$}\label{section-case_m_lessthan_-N}

In this section, we establish a duality between the F-systems for $m \geq N$ and for $m \leq -N$ (Proposition \ref{prop-duality_m_and_-m}). This result, combined with Theorem \ref{Thm-step2}, allows us to show Theorems \ref{mainthm1} and \ref{mainthm2} for $m = -N$. We begin by introducing some notation.

Let $m \in \Z_{\leq -N}$ and set $p = -m \in \N_{\geq N}$. Given $a\in \N$, we define a map
\begin{equation*}
\psi^\pm: \bigoplus_{k = p-N}^{p+N} \Pol_{a-k}[t]_\text{{\normalfont{even}}} \rightarrow \Hom_{SO(2)}\left(V^{2N+1}, \C_{\pm p} \otimes \Pol^a(\n_+)\right)
\end{equation*}
by
\begin{equation*}
\psi^\pm((g_k)_{k=p-N}^{p+N})(\zeta) = \sum_{k = p-N}^{p+N}\left(T_{a-k}g_k\right)(\zeta)h_k^\pm.
\end{equation*}
Recall that $h_k^\pm$ are the generators defined in \eqref{generators-Hom(V2N+1,CmHk)}. In coordinates we have
\begin{equation*}
\begin{gathered}
\psi^+((g_k)_{k=p-N}^{p+N})(\zeta) =
\begin{pmatrix}
\displaystyle{(T_{a-p+N}g_{p-N})(\zeta) (\zeta_1 +i\zeta_2)^{p-N}}\\[4pt]
\vdots\\[4pt]
\displaystyle{(T_{a-p-N}g_{p+N})(\zeta) (\zeta_1 + i\zeta_2)^{p+N}}
\end{pmatrix},\\[4pt]
\psi^-((g_k)_{k=p-N}^{p+N})(\zeta) =
\begin{pmatrix}
\displaystyle{(T_{a-p-N}g_{p+N})(\zeta) (\zeta_1 -i\zeta_2)^{p+N}}\\[4pt]
\vdots\\[4pt]
\displaystyle{(T_{a-p+N}g_{p-N})(\zeta) (\zeta_1 - i\zeta_2)^{p-N}}
\end{pmatrix}.
\end{gathered}
\end{equation*}
For simplicity, for a given $(g_k)_{k=p-N}^{p+N}$ we also write $\psi^\pm(\zeta) \equiv \psi^\pm((g_k)_{k=p-N}^{p+N})(\zeta)$. 

Let $\Phi$ be the following involution on $\Hom_\C\left(V^{2N+1}, \Pol(\n_+)\right)$
\begin{equation*}
\begin{gathered}
\Phi(\varphi)(\zeta_1, \zeta_2, \zeta_3) = 
((-1)^s\varphi_{2N-s}(\zeta_1, -\zeta_2, \zeta_3))_{s=0}^{2N},
\end{gathered}
\end{equation*}
for $\varphi(\zeta) = (\varphi_s(\zeta))_{s=0}^{2N} \in \Hom_\C\left(V^{2N+1}, \Pol(\n_+)\right)$.
Note that for $\varphi = \psi^+(\zeta) \equiv \psi^+((g_k)_{k=p-N}^{p+N})(\zeta)$, we have
\begin{equation}\label{relation-Phi(psi+)-and-psi-}
\Phi(\psi^+)(\zeta_1, \zeta_2, \zeta_3) = \psi^-(((-1)^{k-p+N}g_k)_{k=p-N}^{p+N})(\zeta_1, \zeta_2, \zeta_3).
\end{equation}
We obtain the following
\begin{lemma}\label{lemma-duality_m_and_-m}
Let $C_1^+ \in \n_+$ defined as in \eqref{elements}. Then, we have
\begin{equation}\label{interwtining-property}
\Phi\left(\left(\widehat{d \pi_{(\sigma^{2N+1}, \lambda)^*}}(C_1^+) \otimes \id_{\C_{p,\nu}}\right)\psi^+\right) = \left(\widehat{d \pi_{(\sigma^{2N+1}, \lambda)^*}}(C_1^+) \otimes \id_{\C_{-p,\nu}}\right)\Phi(\psi^+).
\end{equation}
\end{lemma}
\begin{proof}
Let $M_s(\psi^+)(\zeta)$ and $M_s(\Phi(\psi^+))(\zeta)$ denote the vector coefficients of the following operators, respectively:
\begin{equation*}
\begin{gathered}
\left(\widehat{d \pi_{(\sigma^{2N+1}, \lambda)^*}}(C_1^+) \otimes \id_{\C_{p,\nu}}\right)\psi^+, \quad \left(\widehat{d \pi_{(\sigma^{2N+1}, \lambda)^*}}(C_1^+) \otimes \id_{\C_{-p,\nu}}\right)\Phi(\psi^+),
\end{gathered}
\end{equation*}
(see Section \ref{section-vector-coefficients}). Then, in order to prove \eqref{interwtining-property} it suffices to show
\begin{equation*}
\left((-1)^sM_{2N-s}(\psi^+)(\zeta_1, -\zeta_2, \zeta_3)\right)_{s=0}^{2N} = \left(M_s(\Phi(\psi^+))(\zeta_1, \zeta_2, \zeta_3)\right)_{s=0}^{2N}
\end{equation*}
Since $M_s = M_s^\text{{\normalfont{scalar}}} + M_s^\text{{\normalfont{vect}}}$, we consider the scalar and vector parts $M_s^\text{{\normalfont{scalar}}}$ and $M_s^\text{{\normalfont{vect}}}$ of each side of the identity above separately.

1) \underline{$M_s^\text{{\normalfont{scalar}}}$}: For $\varphi(\zeta_1, \zeta_2, \zeta_3) \in \Hom_\C\left(V^{2N+1}, \Pol(\n_+)\right)$ we define $(\alpha_2\varphi)(\zeta_1, \zeta_2, \zeta_3) = \varphi(\zeta_1, -\zeta_2, \zeta_3)$.
Then, clearly $\Delta_{\C^3} \circ \alpha_2 = \alpha_2 \circ \Delta_{\C^3}$ and $E_\zeta \circ \alpha_2 =  \alpha_2 \circ E_\zeta$; which implies that
\begin{equation*}
\widehat{d\pi_{(\sigma^{2N+1}, \lambda)^*}^\text{{\normalfont{scalar}}}}(C_1^+) \circ \alpha_2 = \alpha_2 \circ \widehat{d\pi_{(\sigma^{2N+1}, \lambda)^*}^\text{{\normalfont{scalar}}}}(C_1^+).
\end{equation*}
Therefore
\begin{equation*}
\begin{gathered}
\left((-1)^sM_{2N-s}^\text{scalar}(\psi^+)(\zeta_1, -\zeta_2, \zeta_3)\right)_{s=0}^{2N} = \alpha_2\left((-1)^sM_{2N-s}^\text{scalar}(\psi^+)(\zeta_1, \zeta_2, \zeta_3)\right)_{s=0}^{2N} \\
= \left((-1)^sM_{2N-s}^\text{scalar}(\alpha_2\psi^+)(\zeta_1, \zeta_2, \zeta_3)\right)_{s=0}^{2N} = \left(M_s^\text{scalar}(\Phi(\psi^+))(\zeta_1, \zeta_2, \zeta_3)\right)_{s=0}^{2N}
\end{gathered}
\end{equation*}

2) \underline{$M_s^\text{{\normalfont{vect}}}$}:
By (\ref{Ms-vect}), the left-hand side amounts to
\begin{equation*}
\begin{aligned}
\Big((-1)^s[-(2N-s)\frac{\partial}{\partial \zeta_3}\psi^+_{2N-s-1}(\zeta_1, \zeta_2, \zeta_3) + 2i(N-s)\frac{\partial}{\partial(-\zeta_2)}\psi^+_{2N-s}(\zeta_1, \zeta_2, \zeta_3) \\
+s\frac{\partial}{\partial \zeta_3}\psi^+_{2N-s+1}(\zeta_1, \zeta_2, \zeta_3)]\Big)_{s=0}^{2N}
\end{aligned}
\end{equation*}
but by the definition of $\Phi$ and by using the direct formula $\frac{\partial}{\partial (-\zeta_2)} = -\frac{\partial}{\partial \zeta_2}$ this amounts exactly to
$\left(M_s(\Phi(\psi^+))^\text{vect}(\zeta_1, \zeta_2, \zeta_3)\right)_{s=0}^{2N}$. 
Hence, the identity (\ref{interwtining-property}) holds.
\end{proof}
Observe that
\begin{equation*}
\Sol(\n_+; \sigma_\lambda^{2N+1}, \tau_{\pm p, \nu}) \subset \Hom_\C\left(V^{2N+1}, \Pol(\n_+)\right).
\end{equation*}
Now, if we define
\begin{equation*}
\Phi_\text{Sol}^+ := \Phi\big\rvert_{\Sol(\n_+; \sigma_\lambda^{2N+1}, \tau_{p, \nu})},
\end{equation*}
by Lemma \ref{lemma-duality_m_and_-m}, we have
\begin{equation*}
\Phi_\text{Sol}^+: \Sol(\n_+; \sigma_\lambda^{2N+1}, \tau_{p, \nu}) \rightarrow \Sol(\n_+; \sigma_\lambda^{2N+1}, \tau_{-p, \nu}).
\end{equation*}
Moreover,
\begin{prop}\label{prop-duality_m_and_-m} The map
\begin{equation}\label{duality-isomorphism}
\begin{gathered}
\Phi_{\normalfont{\text{Sol}}}^+: \Sol(\n_+; \sigma_\lambda^{2N+1}, \tau_{p, \nu}) \arrowsimeq \Sol(\n_+; \sigma_\lambda^{2N+1}, \tau_{-p, \nu}),
\end{gathered}
\end{equation}
is a linear isomorphism.
\end{prop}
\begin{proof}
Since $\Phi$ is an involution on $\Hom_\C\left(V^{2N+1}, \Pol(\n_+)\right)$, $\Phi_\text{Sol}^+$ 
is clearly injective. Now, take $\varphi(\zeta) \in \Sol(\n_+; \sigma_\lambda^{2N+1}, \tau_{-p, \nu})$. Since
\begin{equation*}
\Sol(\n_+; \sigma_\lambda^{2N+1}, \tau_{-p, \nu}) \subset \Hom_{SO(2)}\left(V^{2N+1}, \C_{-p, \nu}\otimes \Pol(\n_+)\right),
\end{equation*}
it follows from Proposition \ref{prop-FirststepFmethod} and from
 (\ref{relation-Phi(psi+)-and-psi-}) that
 there exists 
\begin{equation*}
\psi^+((g_k)_{k=p-N}^{p+N}) \in \Hom_{SO(2)}\left(V^{2N+1}, \C_{p, \nu}\otimes \Pol(\n_+)\right)
\end{equation*} 
such that $\Phi(\psi^+)(\zeta) = \varphi(\zeta)$. Now, by Lemma \ref{lemma-duality_m_and_-m}, we have
\begin{equation*}
\begin{aligned}
\Phi\left(\left(\widehat{d \pi_{(\sigma^{2N+1}, \lambda)^*}}(C_1^+) \otimes \id_{\C_{p,\nu}}\right)\psi^+\right)(\zeta) & = \left(\widehat{d \pi_{(\sigma^{2N+1}, \lambda)^*}}(C_1^+) \otimes \id_{\C_{-p,\nu}}\right)\Phi(\psi^+)(\zeta) \\
& =
\left(\widehat{d \pi_{(\sigma^{2N+1}, \lambda)^*}}(C_1^+) \otimes \id_{\C_{-p,\nu}}\right)\varphi(\zeta) \\
& = 0 \quad (\text{by hypothesis}).
\end{aligned}
\end{equation*}
Since $\Phi$ is injective, this shows that
\begin{equation*}
\left(\widehat{d \pi_{(\sigma^{2N+1}, \lambda)^*}}(C_1^+) \otimes \id_{\C_{p,\nu}}\right)\psi^+(\zeta) = 0,
\end{equation*}
that is, $\psi^+(\zeta) \in \Sol(\n_+; \sigma_\lambda^{2N+1}, \tau_{p, \nu})$. The proposition follows now.
\end{proof}

\begin{rem}
(1) From Proposition \ref{prop-duality_m_and_-m} above, we deduce that the space $\Sol(\n_+, \sigma_\lambda^{2N+1}, \tau_{-N, \nu})$ has dimension at most one and that the necessary and sufficient condition on the parameters $(\lambda, \nu, N)$ for it to be non-zero is the same as in the case $m = N$, namely $\nu - \lambda \in \N$. This establishes Theorem \ref{mainthm1} for $m = -N$.

(2) As for the differential operators $\D_{\lambda, \nu}^{N, -N}$, from the result above, we deduce that the operator $\D_{\lambda, \nu}^{N,-N}$ can be obtained by reordering the scalar operators in each coordinate and replacing $i$ by $-i$ in the expression of the operator $\D_{\lambda, \nu}^{N,N}$. Specifically, from the expression of the map $\Phi_\text{Sol}^+$, $\D_{\lambda, \nu}^{N, -N}$ can be obtained by swapping the $s$-coordinate with $(-1)^s$ times the $(2N-s)$-coordinate and replacing $i$ by $-i$ in the operator $\D_{\lambda, \nu}^{N, N}$. This proves Theorem \ref{mainthm2} for $m = -N$.
\end{rem}

\begin{rem} Note that the results of this section can be applied when solving Problems A and B for $|m| > N$. In fact, it suffices to solve them when $m > N$ and then use Proposition \ref{prop-duality_m_and_-m} to obtain the solution for $m < -N$.
\end{rem}

\section{Appendix: About Gegenbauer Polynomials}\label{section-appendix}
In this section, we define the \emph{renormalized} Gegenbauer polynomials $\Geg_\ell^\mu$ and discuss some of their properties. These Gegenbauer polynomials are just a renormalized version of the well-known classical Gegenbauer polynomials $C_\ell^\mu$, named after the Austrian mathematician Leopold H. Gegenbauer \cite{gegenbauer}. The key property of these polynomials is that they solve the Gegenbauer differential equation.

\subsection{Renormalized Gegenbauer polynomials}
For any $\mu \in \C$ and any $\ell \in \N$, the ultraspherical polynomial (also known as Gegenbauer polynomial), is defined as follows (\cite[Sec. 6.4]{aar}, \cite[Sec. 3.15.1]{emot}):
\begin{equation*}
\begin{aligned}
C_\ell^\mu(z) & := \sum_{k = 0}^{[\frac{\ell}{2}]}(-1)^k \frac{\Gamma(\ell - k + \mu)}{\Gamma(\mu)k!(\ell -  2k)!}(2z)^{\ell - 2k}.
\end{aligned}
\end{equation*}

This polynomial satisfies the Gegenbauer differential equation $G_\ell^\mu f(z) = 0$, where $G_\ell^\mu$ is the Gegenbauer differential operator:
\begin{equation*}
G_\ell^\mu := (1-z^2)\frac{d^2}{dz^2} - (2\mu +1)z\frac{d}{dz} + \ell(\ell +2\mu).
\end{equation*}
Note that $C_\ell^\mu \equiv 0$ if $\mu = 0, -1, -2, \cdots, - [\frac{\ell-1}{2}]$. In order to obtain a non-zero polynomial for any value of the parameters, we renormalize the Gegenbauer polynomial as follows:
\begin{equation}\label{Gegenbauer-polynomial(renormalized)}
\widetilde{C}^\mu_\ell(z) := \frac{\Gamma(\mu)}{\Gamma\left(\mu + [\frac{\ell + 1}{2}]\right)}C^\mu_\ell(z) =\sum_{k=0}^{[\frac{\ell}{2}]}\frac{\Gamma(\mu + \ell - k)}{\Gamma\left(\mu+ [\frac{\ell + 1}{2}]\right)}\frac{(-1)^k }{k! (\ell - 2k)!}(2z)^{\ell - 2k}.
\end{equation}
We write below the first five Gegenbauer polynomials.

\begin{itemize}
\item $\widetilde{C}^\mu_0(z) = 1$.
\item $\widetilde{C}^\mu_1(z) = 2z$.
\item $\widetilde{C}^\mu_2(z) = 2(\mu + 1) z^2 -1$.
\item $\widetilde{C}^\mu_3(z) = \frac{4}{3}(\mu + 2) z^3 -2z$.
\item $\widetilde{C}^\mu_4(z) = \frac{2}{3}(\mu +2)(\mu+3)z^4 -
2(\mu + 2)z^2 + \frac{1}{2}$.
\end{itemize}

Observe that from the definition, all terms of $\Geg_\ell^\mu(z)$ have the same parity, namely, the parity of $\ell$. Moreover, the number of terms is at most $\left[\frac{\ell}{2}\right]+1$. Depending on the value of $\mu$, some terms may vanish, but the term of degree $\ell - 2\left[\frac{\ell}{2}\right] \in \{0,1\}$ is non-zero for any $\mu \in \C$.
By a direct computation one can show that this term is given by:
\begin{equation*}
\begin{aligned}
\frac{(-1)^{\left[\frac{\ell}{2}\right]}}{\left[\frac{\ell}{2}\right]!} (2z)^{\ell - 2\left[\frac{\ell}{2}\right]}.
\end{aligned}
\end{equation*}
Note that if $\mu + \left[\frac{\ell+1}{2}\right] = 0$, then $\Geg_{\ell}^\mu(z)$ is constant for any $\ell \geq 0$. For $\ell < 0$ we just define the renormalized Gegenbauer polynomial as $\Geg_\ell^\mu \equiv 0$. 

The next result shows that the solutions of the Gegenbauer differential equation on $\Pol_\ell[z]_\text{{\normalfont{even}}}$ are precisely the renormalized Gegenbauer polynomials, where
\begin{equation*}
\Pol_\ell[z]_\text{even} = \spanned_\C\{z^{\ell - 2b}: b = 0, 1, \ldots, \left[\frac{\ell}{2}\right]\}.
\end{equation*}
\begin{thm}[{\cite[Thm 11.4]{kob-pev2}}]\label{thm-Gegenbauer-solutions} For any $\mu \in \C$, $\ell \in \N$ the following holds.
\begin{equation*}
\{f(z) \in \Pol_\ell[z]_\text{{\normalfont{even}}} : G_\ell^\mu f(z) = 0\} = \C \widetilde{C}_\ell^\mu(z).
\end{equation*}
\end{thm}
In $\C[z] \setminus \Pol_\ell[z]_\text{{\normalfont{even}}}$, the Gegenbauer differential equation may have other solutions depending on the parameters $\mu, \ell$. For a concrete statement of this fact see \cite[Thm. 11.4]{kob-pev2}.

For any $\mu \in \C$ and any $\ell \in \N$, the \textit{imaginary} Gegenbauer differential operator $S_\ell^\mu$ is defined as follows:
\begin{equation}\label{Gegen-imaginary}
S_\ell^\mu = - (1 + t^2) \frac{d^2}{dt^2} - (1 + 2\mu)t\frac{d}{dt} + \ell(\ell + 2\mu).
\end{equation}
This operator and $G_\ell^\mu$ are in the following natural relation.

\begin{lemma}[{\cite[Lem. 14.2]{kkp}}]\label{lemma-relationS-G} Let $f \in \C[z]$ define $g(t) = f(z)$, where $z = it$. Then,
\begin{equation*}
\left(S_\ell^\mu g\right)(t) = \left(G_\ell^\mu f\right)(z).
\end{equation*}
\end{lemma}

\subsection{Algebraic properties of renormalized Gegenbauer polynomials and imaginary Gegenbauer operators}
In this section we collect some useful properties of (imaginary) Gegenbauer operators and (renormalized) Gegenbauer polynomials. We start by showing some algebraic identities involving derivatives and follow by recalling one three-term relation among Gegenbauer polynomials.

Given $\mu \in \C$ and $\ell \in \N$, we define the following gamma factor:
\begin{equation}\label{gamma-def}
\gamma(\mu, \ell) := \displaystyle{\frac{\Gamma(\mu + \left[\frac{\ell+2}{2}\right])}{\Gamma(\mu + \left[\frac{\ell + 1}{2}\right])} = \begin{cases}
1,&\text{if } \ell \text{ is odd},\\
\mu + \frac{\ell}{2},&\text{if } \ell \text{ is even}.
\end{cases}}
\end{equation}
Note that this definition can be extended to $\ell \in \Z$. For $\ell < 0$ it suffices to take $k \in \N$ such that $\ell + 2k \in \N$ and define $\gamma(\mu, \ell) := \gamma(\mu-k,\ell + 2k)$. 

Another remarkable property of $\gamma(\mu, \ell)$ is that it satisfies the following identity, which is straightforward from the definition:
\begin{equation}\label{gamma-product-property}
\gamma(\mu, \ell)\gamma(\mu, \ell+1) = \mu + \left[\frac{\ell +1}{2}\right].
\end{equation}

During the rest of the section, we denote by $\vartheta_t$ the one-dimensional Euler operator $t\frac{d}{dt}$.

\begin{lemma}[{\cite[Lem. 14.4]{kkp}}] For any $\mu \in \C$ and any $\ell \in \N$, the following identities hold:
\begin{align}
\label{derivative-Gegenbauer-1}
\frac{d}{dt}\Geg_\ell^\mu(it) &= 2i\gamma(\mu, \ell)\Geg_{\ell -1}^{\mu + 1}(it),\\[4pt]
\label{derivative-Gegenbauer-2}
\left(\vartheta_t-\ell\right)\widetilde{C}_\ell^\mu(it) &= \Geg_{\ell -2}^{\mu + 1}(it).
\end{align}
\end{lemma}

\begin{lemma}[{\cite[Lem. 14.6]{kkp}}]\label{lemma-KKP-identities}  For any $\mu \in \C$ and any $\ell \in \N$ the following three-term relation holds:
\begin{align}
\label{KKP-1}
&(\mu + \ell)\Geg_\ell^\mu(it) + \Geg_{\ell-2}^{\mu+1}(it) = \left(\mu + \left[\frac{\ell+1}{2}\right]\right)\Geg_\ell^{\mu+1}(it).
\end{align}
\end{lemma}

In \eqref{Gegen-imaginary} we have defined the imaginary Gegenbauer operator $S_\ell^\mu$. In the following lemma we give some algebraic properties of this operator.

\begin{lemma} For any $\mu \in \C$ and any $\ell, d \in \N$ the following identities hold:
\begin{align}
\label{Slmu-identity-1}
S_\ell^{\mu} - S_\ell^{\mu+d} &=  2d(\vartheta_t-\ell),\\
\label{Slmu-identity-2}
S_\ell^\mu - S_{\ell -2d}^{\mu + d} &= 2d(\vartheta_t + 2\mu + \ell).
\end{align}
\end{lemma}
\begin{proof}
The proof is straightforward from the definition of $S_\ell^\mu$.
\end{proof}
We remark that identities \eqref{Slmu-identity-1} and \eqref{Slmu-identity-2} above are a simple generalization of 
\cite[Eq. (14.9)]{kkp} and \cite[Eq. (14.12)]{kkp} respectively.

\begin{lemma}
Let $\mu \in \C$ and $\ell \in \N$ and suppose that $2\mu + \ell = 0$. Then, for any $n, d \in \N$ we have:
\begin{equation}\label{Slmu-identity-concrete}
S_{\ell+n}^{\mu-n} t^{n-2d} = -2d(\ell+2d)t^{n-2d} - (n-2d)(n-2d-1)t^{n-2d-2}.
\end{equation}
\end{lemma}
\begin{proof}
The proof is straightforward from the definition of $S^\mu_\ell$ and the condition $2\mu+\ell = 0$.
\end{proof}

\section*{Acknowledgements}
I would like to express my deepest gratitude to Professors Ali Baklouti and Hideyuki Ishi for the extraordinary organization, and the warm welcome extended to me during the 7th Tunisian-Japanese Conference that occurred in Monastir (Tunisia) in November 2023. Your efforts made the experience truly memorable.

The author was supported by the Japan Society for the Promotion of Science (JSPS) as an International Research Fellow (ID. No. P24018).

\renewcommand{\refname}{References}

\begin{flushleft}
V. Pérez-Valdés, JSPS International Research Fellow, Ryukoku University, Tsukamoto-cho 67, Fukakusa, Fushimi-ku, Kyoto 612-8577, Japan.\\
Email address: \textbf{perez-valdes@mail.ryukoku.ac.jp}
\end{flushleft}
\end{document}